\documentclass[aap]{amsart}
\usepackage[english]{babel}
\usepackage[normalem]{ ulem }
\usepackage{soul}

\usepackage[utf8]{inputenc}
\usepackage{amsmath,amsfonts,amssymb,amsopn,amscd,amsthm}
\usepackage{gensymb}
\usepackage{comment}
\usepackage{dsfont}
\usepackage{graphicx}
\usepackage{color}
\usepackage{epigraph}
\usepackage{enumerate}
\usepackage{hyperref}
\hypersetup{colorlinks,linkcolor={red},citecolor={blue},urlcolor={blue}}
\usepackage{bigints}
\usepackage{tikz}
\usepackage{xcolor} 
\usetikzlibrary {arrows}
\usepackage{tikz}
\usetikzlibrary{positioning,backgrounds}

\usepackage{marginnote}
\usepackage{todonotes}
\usepackage[left=3.5cm,top=2.5cm,bottom=2cm,right=4cm, marginparwidth=3cm]{geometry}

\title[ Superdiffusion transition  ]
      { Superdiffusion transition for a phonon Boltzmann equation}

\author{Ga\"etan Cane \\\\
\texttt{\href{https://gaetancane.ch/}{gaetancane.ch}}}

\allowdisplaybreaks[4]

\def\1{{\mathbf 1}}

\def\N{{\mathbb N}}
\def\Z{{\mathbb Z}}

\def\R{{\mathbb R}}
\def\C{{\mathbb C}}

\def\T{\mathbb{T}}
\def\P{{\mathbb P}}
\def\E{{\mathbb E}}

\def\Bc{{\mathcal B}}
\def\Cc{{\mathcal C}}

\def\Ec{{\mathcal E}}
\def\Fc{{\mathcal F}}
\def\Gc{{\mathcal G}}

\def\Lc{{\mathcal L}}

\def\Oc{{\mathcal O}}

\def\Sc{{\mathcal S}}
\def\Tc{{\mathcal T}}

\def\Wc{{\mathcal W}}

\def\i{{\mathbf{i}}}

\newtheorem{thm}{Theorem}[section]
\newtheorem{proposition}[thm]{Proposition}

\newtheorem{lemma}[thm]{Lemma}

\newtheorem{rmk}[thm]{Remark}

\newcommand{\ba}{\begin{array}}
\newcommand{\ea}{\end{array}}
\newcommand{\be}{\begin{equation}}
\newcommand{\ee}{\end{equation}}
\newcommand{\bea}{\begin{eqnarray}}
\newcommand{\eea}{\end{eqnarray}}
\newcommand{\beaa}{\begin{eqnarray*}}
\newcommand{\eeaa}{\end{eqnarray*}}
\newcommand{\ignore}[1]{}
\newcommand{\vertiii}[1]{{\left\vert\kern-0.25ex\left\vert\kern-0.25ex\left\vert #1 
    \right\vert\kern-0.25ex\right\vert\kern-0.25ex\right\vert}}

\begin{document}

\begin{abstract}
We consider an infinite harmonic chain of charged particles submitted to the action of a magnetic field of intensity $B$ and subject to the action of a stochastic noise conserving the energy. In \cite{JKO09} it has been proved that if $B=0$ the transport of energy is described by a $3/4$-fractional diffusion while it has been proved in \cite{SSS19} that if $B\ne 0$ it is described by a $5/6$-fractional diffusion. In \cite{JKO09,SSS19} the authors used a two step argument, i.e. they first proved that the kinetic limit of the Wigner distribution is the solution of a phonon Boltzmann equation and then proved that this solution converges to the solution of a fractional diffusion equation with exponent $3/4$ if $B=0$ (see \cite{JKO09}) and exponent $5/6$ if $B\neq 0$ (see \cite{SSS19}). In this paper we quantify the intensity of the magnetic field required to switch from one macroscopic regime to the other one from the phonon Boltzmann equation.  We also describe the transition mechanism to cross the two different phases.
\end{abstract}


\maketitle

\setcounter{tocdepth}{1}
\medskip
\medskip
\medskip
\hrule
\tableofcontents
\hrule
\newpage

\section{Introduction}
Since the seminal numerical experiments of Fermi-Pasta-Ulam-Tsingou in 1953 \cite{DPR05}, the understanding of energy transport in very long anharmonic chains of coupled oscillators attracted a lot of attention but still remains a fascinating challenging open problem in mathematical physics. During the two last decades many researchers have been interested to the one dimensional case for which the energy transport is anomalous (we refer the reader to the physical reviews \cite{D08,LLP03}). The most elaborated theory to describe the form of this anomalous transport is probably the recent nonlinear fluctuating hydrodynamics theory initiated by Spohn \cite{S14, SS15} which predicts, for interacting particle systems with several conserved quantities (like energy, momentum etc.) and local interactions, several different universality classes containing not only the famous  KPZ universality class \cite{QS15} but also many fractional diffusion classes, apart from the standard Edwards-Wilkinson class (i.e. normal diffusion class). The theory is macroscopic and based on formal arguments starting from the hydrodynamic equations associated to the interacting particle system under investigation.  Unfortunately its validity for systems with more than one conservation law has been proved rigorously only for few stochastic models \cite{BGJ16, BGJS18-2,JKO15}. It has also to be noticed that while the link between the KPZ universality class and the Edwards-Wilkinson universality class is provided by the now well understood KPZ equation \cite{GP17,H13}, almost nothing is known about the potential equations connecting the other universality classes, even at some heuristic level.\\

The system we are interested to belongs to the class of systems introduced in \cite{FFL94}, revisited with the heat conduction problem perspective in \cite{BBO06, BBO09,BO05}, and studied since in several subsequent works by various authors (see e.g. \cite{BBJKO16} and \cite{KO20} for some reviews). In this paper we consider the model of \cite{SSS19} (introduced first in \cite{SS18}), i.e. a harmonic chain of charged particles submitted to a magnetic field and a stochastic exchange of velocity between neighbor sites. Each particle is labeled by its rest position $x$ in $\mathbb Z$ and lives in the two dimensional space $\mathbb R^2$. Its displacement from its rest position is denoted by $q (x) \in \mathbb R^2$, its velocity by $p (x) \in \mathbb R^2$ and its energy by $e(x)$. The magnetic field of intensity $B$ is constant and orthogonal to the plan of motion of the chain. A picture of the model is given in Fig. \ref{model_Makiko}.  The equations of motion of the deterministic system are then given for any time $t$ and $i$ in $\{1,2\}$ by 
\begin{equation}
\label{microscopic_system_intro}
\begin{array}{l}
  \frac{d}{dt}q_{i}(t,x)  = p_{i}(t,x)\ , \\\\
  \frac{d}{dt}p_{i}(t,x)= \Delta_{d} \left[q_i(t,x)\right] +\delta_{i,1}B p_2(t,x) - \delta_{i,2}B p_1(t,x)\ ,
\end{array}
\end{equation}
where $\delta_{i,j}$ denotes the Kronecker symbol  and $\Delta_{d}$ is the discrete\footnote{We recall that for a function $f$ defined on $\mathbb{Z}$ the discrete Laplacian of $f$  is defined for any $x$ in $\mathbb{Z}$ by $\Delta_d \left[f(x)\right]= f(x+1)+f(x-1)-2f(x)$.} Laplacian on $\mathbb{Z}$. 
The total energy $E$, which is conserved by the dynamics, is given by
\[E= \sum_{x \in \mathbb Z} e (x)= \frac{1}{2} \left(\sum_{x \in \mathbb{Z}} \vert p(x) \vert ^2 + \sum_{x \in \mathbb{Z}} \vert q(x)-q(x+1) \vert ^2 \right).\]
We superpose to the deterministic dynamics \eqref{microscopic_system_intro} a stochastic noise which exchanges continuously the velocities of nearest neighbor particles. The noise \textbf{conserves the total energy} of the chain and \textbf{the total pseudo-momentum}{\footnote{We recall that the pseudo-momentum of particle $x$ is $p(x) +B\sigma q(x)$ where $\sigma =\left( \begin{array}{cc} 0&1\\-1&0\end{array} \right)$. If $B=0$ the pseudo-momentum coincides with the velocity. If $B\ne0$, the velocity is not conserved by the Hamiltonian dynamics.}}.  The goal of this paper is to understand the mechanism to cross different universality classes by varying the intensity of the magnetic field.\\
\begin{figure}[h]
\label{model_Makiko}
\includegraphics[scale=0.25]{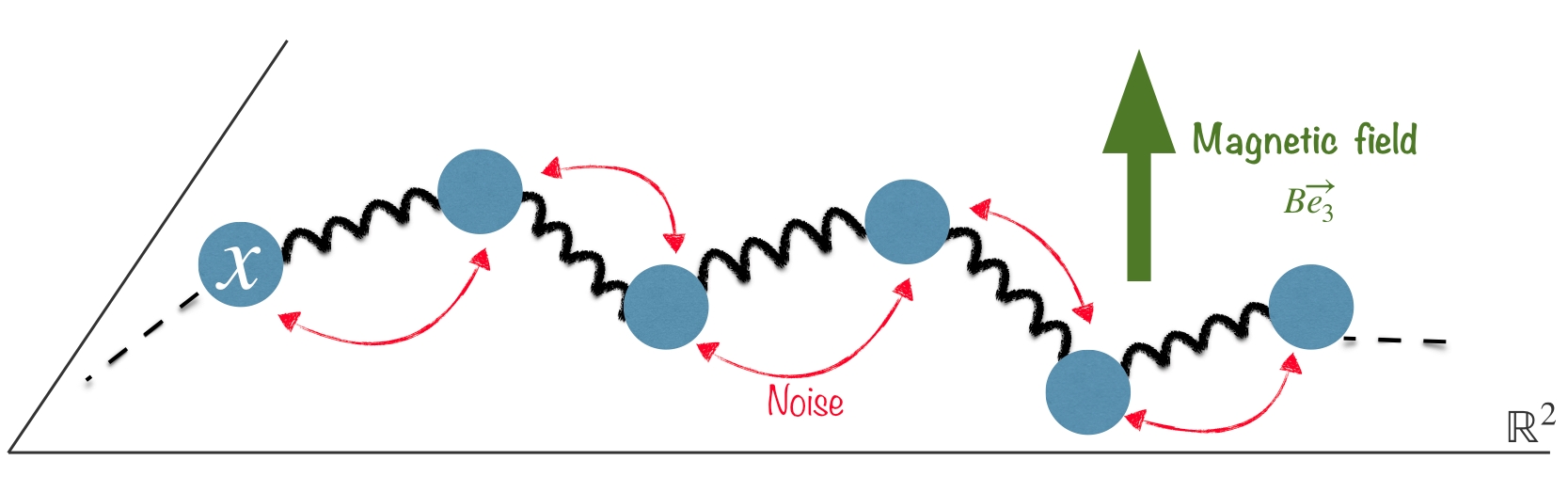}
\caption{Picture of the model studied.}
\end{figure}
\subsection*{Case without magnetic field, i.e. $B=0$.} In order to study the macroscopic evolution of the energy the authors of \cite{BOS09}, following \cite{S06}, introduced the Wigner distribution  $\Wc^\varepsilon$ (defined in Sec.\;\ref{section_wigner}) in the time scale $t \varepsilon^{-1}$, and  in the weak noise limit, i.e. the intensity of the noise is of order $\varepsilon$. Here $\varepsilon$ is a space scaling parameter, i.e. the lattice $\mathbb Z$ is rescaled in $\varepsilon \mathbb Z$, and going to zero. The Wigner distribution is a kind of localized Fourier transform of the space correlations of the eigenmodes, also called phonons, of the purely deterministic harmonic chain (the presence of the stochastic noise couple their time evolutions which become then non trivial). Each eigenmode is labeled by a $k\in {\mathbb T}$, the continuous torus of length one. For each time $t$, $\Wc^{\varepsilon} (t) $ is a distribution acting on a class of test functions $J: (u,k) \in (\mathbb R \times \mathbb T) \to J(u,k) \in \mathbb C$. To get some intuition on the relevance of the Wigner distribution for the study of the macroscopic behavior of the energy, we observe that if $J: (u,k)  \in (\mathbb R \times \mathbb T) \to {\tilde J} (u) \in \mathbb R$ is a test function not depending on the $k$ variable, we have
\begin{align*} \left\langle \Wc^\varepsilon (t), J \right\rangle =  \varepsilon \sum_{x \in \mathbb{Z}} \mathbb{E}_{\mu^\varepsilon} \left[e\left( t\varepsilon^{-1},x\right) \right] {\tilde J} (\varepsilon x) + \mathcal{O}_J(\varepsilon) \ ,\end{align*}
where $e(t,x)$ is the value of the energy of the particle $x$ at time $t$, $\mu^\varepsilon$ the initial distribution of the dynamics and $\mathcal{O}_J(\varepsilon)$ is an error term which depends on the test function $J$ and goes to zero when $\varepsilon$ goes to zero. From the previous equation it is clear that if we want to understand the macroscopic behavior of the energy when $\varepsilon$ goes to zero we have to understand the behavior of $\Wc^\varepsilon$ as $\varepsilon \to 0$. 

In \cite{BOS09} it is proved that at kinetic time scale $\varepsilon^{-1}$, the Wigner distribution ${\mathcal W}^\varepsilon$ converges to the unique solution $f_0(t,u,k,i)$, of the following phonon linear Boltzmann equation
\begin{equation}
\label{boltzmann_intro_stefano}
\partial_t f_0(t,u,k,i) + \frac{\textbf{v}_0(k)}{2\pi} \partial_u f_0(t,u,k,i) = \mathcal{L}_0\left[f_0\right](t,u,k,i) \ , 
\end{equation}
with

\begin{equation}
\label{def_C_intro}
    \mathcal{L}_0\left[f_0\right](t,u,k,i) = \sum_{j=1}^2 \int_{\T} R_0(k,k^\prime,i,j) \left[ f_0\left(t,u,k^\prime,j\right) - f_0\left(t,u,k,i\right) \right] \; dk^\prime \ .
    \end{equation}
Here, $u \in \R$ represents the position along the chain after the kinetic limit,  $t\ge 0$ the time and $k \in \T$ the wave number of a phonon whereas $i$ is the type of phonon. $\mathcal{L}_0$ is a collisional operator due to by the noise introduced on the system and $\textbf{v}_0$ is the group velocity.\\

In \cite{BOS09}, the authors studied a system where the particles live in $\mathbb{R}$ instead of $\mathbb{R}^2$ as presented in this introduction. As a consequence, the Boltzmann equation obtained in \cite{BOS09} does not depend on the variables $(i,j)$. However, following the proof performed in \cite{BOS09} we can derive Eq.\;\eqref{boltzmann_intro_stefano} for particles living in $\mathbb{R}^2$. \\

Since $R_0(k,k^\prime,i,j)$ is positive we can interpret the solution of Eq.\;\eqref{boltzmann_intro_stefano} as the evolution of the density of a continuous time Markov process $\left(Z_0(\cdot), K_0(\cdot), I_0(\cdot)\right)$. Here $\left(K_0(\cdot), I_0(\cdot)\right) \in {\mathbb T}\times \{ 1,2\}$ is the pure jump Markovian process with generator given by Eq.\;\eqref{def_C_intro} and $Z_0(\cdot)$ is the additive functional defined for any positive time $t$ by
\begin{align}
\label{definition_intro_Z0}
Z_0(t) = -\int_0^t \frac{\textbf{v}_0\left(K_0(s)\right)}{2\pi } \; ds\ .
\end{align}
Using this interpretation, the authors of \cite{JKO09} proved first that the finite-dimensional distributions of $N^{-1}Z_0\left(N^{\frac{3}{2}} \cdot\right)$ converge to the finite-dimensional distributions of a L\'evy process generated (up to a constant) by $-(-\Delta)^{\tfrac{3}{4}}$. In a second time, they used the previous result to show that $f_0\left(N^{\frac{3}{2}}t, Nu,\cdot\right)$ converges in $\mathbb{L}^2\left( \T \right)$ (see Theorem \ref{thm_limi_hydro_u_old}) to the unique solution $\rho_0$, of the following fractional diffusion equation 
\begin{equation}
\label{fractionnal_intro}
\forall u \in \R, \quad \forall t \in \; ]0,T], \quad \partial_t \rho_0(t,u) = -D\left(-\Delta\right)^{\tfrac{3}{4}}\left[\rho_0\right](t,u) \ , 
\end{equation}
where $D$ is a strictly positive constant. \\

As we explained for the results in \cite{BOS09}, in \cite{JKO09} the Boltzmann equation studied does not depend on $(i,j)$ but following the proof of \cite{JKO09} we can derive the fractional diffusion equation \eqref{fractionnal_intro} from the Boltzmann equation \eqref{boltzmann_intro_stefano}. \\

Hence, this two-step argument which consists to first prove the convergence of the Wigner distribution to the unique solution $f_0$ of a phonon Boltzmann equation and then the convergence of $f_0$ to the solution $\rho_0$ of a fractional Laplacian equation gives the nature of the superdiffusion of energy proved in \cite{BBO09}.  
Fractional diffusion equations can also be derived from some Boltzmann equations by purely analysis arguments, see for example \cite{CMP18,CMP21,MMM11} and references therein  for more details .
In 2015, by investigating further the time evolution of the Wigner distribution in a longer time scale, the authors of \cite{JKO15},  proved that we can obtain \textbf{in one step} Eq.\;\eqref{fractionnal_intro} from the microscopic model in a suitable time scale. The results above are consistent with the predictions of the nonlinear fluctuating hydrodynamics theory of Spohn.\\

\subsection*{Case with a magnetic field $B\ne 0$} By following the strategy initiated in \cite{BOS09,JKO09}, the authors of \cite{SSS19} proved that at kinetic time scale $\varepsilon^{-1}$ the Wigner distribution converges to the unique solution $f_B(t,u,k,i)$, of a phonon linear Boltzmann equation of the same form as in Eq.\;\eqref{boltzmann_intro_stefano}
\begin{equation}
\label{boltzmann_intro_makiko}
\partial_t f_B(t,u,k,i) + \frac{\textbf{v}_B(k)}{2\pi} \partial_u f_B(t,u,k,i) = \mathcal{L}_B\left[f_B\right](t,u,k,i) \ , 
\end{equation}
with 
\begin{align*}\mathcal{L}_B\left[f_B\right](t,u,k,i) = \sum_{j=1}^2 \int_{\T} R_B(k,k^\prime,i,j) \left[ f_B\left(t,u,k^\prime,j\right) - f_B(t,u,k,i) \right] \; dk^\prime \ .
\end{align*}
However, because of the presence of the magnetic field, as it is explained in the introduction of \cite{SSS19}, the group velocity  $\textbf{v}_B$ and the scattering kernel $R_B(k,i) = \int_{\T} R_B\left(k,k^\prime,i,j\right) dk^\prime \delta_j(i)$ have different behavior when $k$ goes to zero. Indeed when $B$ is positive we have for $k$ near to zero
\begin{align*}  \textbf{v}_B(k) \sim k , \quad  R_B(k,1) \sim k^2 \quad \text{and} \quad  R_B(k,2) \sim k^4 \ , \end{align*}
whereas in the Boltzmann equation \eqref{boltzmann_intro_stefano} studied in \cite{JKO09} $R_0(k,k^\prime,i,j)$ does not depend on $i$ and $j$ and satisfies when $k$ is close to zero
\begin{align*} \textbf{v}_0(k) \sim 1 \quad \text{and} \quad  R_0(k) \sim k^2  \quad \text{where} \quad R_0(k)= \int_{\T} R_0(k,k^\prime,i,j) \; dk^\prime \  .\end{align*}
This difference has a drastic effect on the energy transport properties of the chain moving the chain from one universality class to an other one. Indeed, the energy superdiffusion is still described by a fractional diffusion equation as in Eq.\;\eqref{fractionnal_intro} but with a different exponent. Following \cite{JKO09} it is proved in \cite{SSS19} that the finite-dimensional distributions of\footnote{Here $Z_B$ is defined as $Z_0$ in Eq.\;\eqref{definition_intro_Z0} with $\mathbf{v}_0$ replaced by $\mathbf{v}_B$.} $N^{-1}Z_B\left(N^{\frac{5}{3}} \cdot\right)$ converge weakly to the finite-dimensional distributions of a L\'evy process generated (up to a constant) by $-(-\Delta)^{\tfrac{5}{6}}$  which implies the convergence of $f_B\left(N^{\frac{5}{3}}t,Nu,\cdot\right)$ in $\mathbb{L}^2\left(\T \right)$ to the unique solution $\rho_B$ of the following fractional diffusion equation 
\begin{equation}
\forall u \in \R, \quad \forall t \in \; ]0,T], \quad \partial_t \rho_B(t,u) = -D_B\left(-\Delta \right)^{\tfrac{5}{6}} \left[\rho_B\right](t,u) \ ,
\end{equation}
where $D_B$ is a strictly positive constant. Notice that since the hydrodynamic limits of this chain are trivial in the Euler time scale, the nonlinear fluctuating hydrodynamics theory of Spohn does not give any prediction for this model. \\

\subsection*{Contribution} The aim of this paper is to study the transition between the study of \cite{JKO09} and the one of \cite{SSS19}. As we have seen, the presence of the magnetic field moves the model from the $3/4$-fractional universality class to the $5/6$-fractional universality class and it makes sense to ask if we can quantify the intensity of the magnetic field necessary to cross from one universality class to some other one, and to understand what is the mechanism occurring at the transition. As far as we know, while very interesting, these questions did not receive answer, even at some heuristic level.  We notice however that in \cite{BGJ18} and \cite{ BGJS18,BGJSS15}  are obtained results describing two transition mechanisms between standard diffusion universality class and fractional diffusion universality class in a Hamiltonian system stochastically perturbed and having two conserved quantities. We would like to mention that a family of fractional diffusion equations have been derived from stochastic harmonic chains with long-range interactions by Suda in \cite{S21}.\\

In the first instance we introduce a small magnetic field of intensity $B \varepsilon^\delta$ with $\delta >0$ and $B\neq0 $. We prove first (see Theorem \ref{main_theorem_kinetic}) that at the kinetic time scale of order $\varepsilon^{-1}$ the transition is trivial in the sense that for $\delta=0$ the Wigner distribution converges to the Boltzmann equation of \cite{SSS19} $\left(\text{i.e. Eq.\;} \eqref{boltzmann_intro_makiko}\right)$ and for $\delta>0$ the Wigner distribution converges to the one of \cite{BOS09} $\left(\text{i.e. Eq.\;} \eqref{boltzmann_intro_stefano}\right)$. We believe however that the effect of the small magnetic field could be seen in a longer time scale.

Therefore, in a second time, we study the hydrodynamic limit of the solution $f_{B_N}$ of the Boltzmann equation 
 \eqref{boltzmann_intro_makiko} when $B$ is changed to $B_N:=BN^{-\delta}$ with $\delta$ in $\R^+$ and $B \neq 0$. Let $\beta_\delta$ defines as follows
 \begin{align*}
 \beta_\delta=\frac{3}{2} \quad \text{if} \quad \delta \geq \frac{1}{2} \quad \text{and} \quad \beta_\delta = \frac{5-\delta}{3} \quad \text{for} \quad \delta \leq \frac{1}{2} \ . 
 \end{align*}
 We prove, in Theorem \ref{main_theorem_hydro} that the finite-dimensional distributions of $N^{-1}Z_{B_N}\left(N^{\beta_\delta}\cdot\right)$ converge weakly to the finite-dimensional distributions of a L\'evy process generated (up to a constant) by an operator $\mathcal{L}_\delta$ whose action on smooth functions $\phi$ which decay sufficiently fast is defined as follows
\begin{equation}
\label{Levy_process_genere_intro}
\forall u \in \R, \quad \mathfrak{L}_\delta\left[\phi\right] \ (u) = \left\{
\begin{array}{l}
  -(-\Delta)^{\frac{3}{4}}\left[\phi\right] \ (u) \quad \text{if} \quad \delta > \frac{1}{2} \ , \\\\
  \mathfrak{L}_{B}\left[\phi\right] \ (u) \quad \text{if} \quad \delta = \frac{1}{2} \ ,\\\\
  -(-\Delta)^{\frac{5}{6}}\left[\phi\right] \ (u) \quad \text{if} \quad \delta < \frac{1}{2} \ ,
\end{array}
\right.
\end{equation}
with $\mathfrak{L}_B$ defined in Eq.\;\eqref{def_Lev_gen_limit}. From Theorem \ref{main_theorem_hydro}, we prove then in Theorem \ref{main_thm_limit_u} that $f_{B_N} \left( N^{\beta_\delta}t,Nu,\cdot\right)$ converges to the unique solution $\rho_\delta$ of the following integro-differential equation 
\begin{equation}
\forall u \in \mathbb{R}, \quad \forall t \in ]0,T[, \quad \partial_t \rho_\delta(t,u) = \mathfrak{L}_\delta\left[\rho_\delta\right] (t,u) \ .
\end{equation}

Finally, we prove in Theorem \ref{main_thm_limit_B_u_barre} that, up to a constant,  $\mathfrak{L}_B \rightarrow -(-\Delta)^{\frac{3}{4}}$ when $B \rightarrow 0$ and $\mathfrak{L}_B \rightarrow -(-\Delta)^{\frac{5}{6}}$ when $B \rightarrow \infty$. Hence, $\mathfrak{L}_B$ is the infinitesimal generator of a L\'evy process which interpolates between the two fractional universality classes. 

These results are summarized in Fig.\;\ref{fig:transition}. On the horizontal axis, $\delta$ represents the intensity of the magnetic field and on the vertical axis $\beta_\delta$ represents the exponent of the scaling in time we have to do in order to obtain the hydrodynamic limit of $f_{B_N}$.\\

\subsection*{Structure of the paper} In Sec.\;\ref{section_notation}, we precise the notations of the paper. In Sec.\;\ref{section_micro}, we present the microscopic dynamics and the Wigner distribution. In Sec.\;\ref{section_previous}, we recall the historical results obtained in \cite{BOS09,JKO09,SSS19}. In Sec.\;\ref{section_result}, we state the main results of this paper which are proved in Sec.\;\ref{section_proof}. In order to make the reading easier, intermediate results are shown in the Appendices. 
\begin{figure}
    \centering
    \includegraphics[scale=0.26]{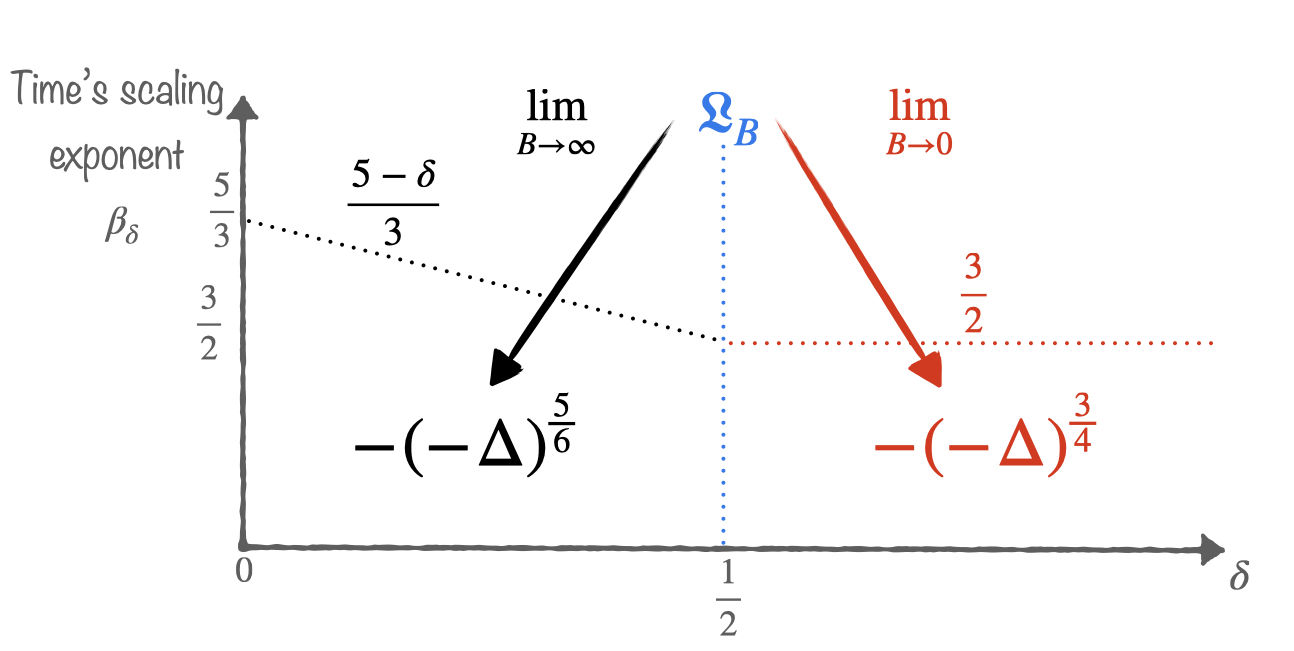}
    \caption{Transition graph.}
    \label{fig:transition}
\end{figure}
\section{Notations}
\label{section_notation}
Let $a$ and $b$ be two positive real numbers, we will write $a \lesssim b$ when there exists a positive constant $C$ such that $a \leq C b$. The conjugate of a complex number $z$ will be denoted by $z^*$ and $\mathbf{i}$ will denote the complex number of modulus 1. We will denote the one dimensional torus $\left[-\tfrac{1}{2}, \tfrac{1}{2}  \right[$ by $\T$, the euclidean norm on $\mathbb{R}^n$ by $\vert \cdot \vert$  and the set of natural numbers by\footnote{Here $\mathbb{N}$, includes $0$.} $\mathbb{N}$. In order to lighten the notations we will denote $\mathbb{R}$ \textbackslash $\lbrace 0 \rbrace$ and $\mathbb{T}$ \textbackslash $\lbrace 0 \rbrace$ by $\mathbb{R}^*$  and $\mathbb{T}^*$ respectively.\\

If $X$ is a topological space we denote the Borelian $\sigma$-field of $X$ by $\Bc (X)$. We denote by $\Fc\left([0,T], X \right)$ the set of $X$-valued functions on $[0,T]$, by $\Cc\left([0,T], X \right)$ the subspace of $X$-valued continuous functions on $[0,T]$ and by $\Cc_{\mathbf{b}}\left([0,T], X \right)$ the subspace of $X$-valued bounded continuous functions on $[0,T]$. Let $n$ in $\N \cup \{ \infty \}$, the space of $\mathbb R$-valued functions on $X$ with compact support and $n$ times differentiable is denoted by $\Cc_{\mathbf{c}}^n \left(X\right)$. The space of $\mathbb R$- valued c\`adl\`ag functions on $[0,T]$ will be denoted by $\mathcal{D}\left( [0,T], \R\right).$\\

For $f$ in $\ell^1(\Z)$, we define  its (discrete) Fourier transform $\widehat{f} : \T \rightarrow \C$ by

\begin{align*} 
\forall k \in \mathbb{T}, \quad \widehat{f}(k) = \sum_{x \in \mathbb{Z}} f(x) \exp(2\i k\pi x) \ . 
\end{align*}
As usual we extend this notation for all functions in $\ell^2(\Z)$. In order to study the Wigner distribution, defined in Sec.\;\ref{section_wigner}, we introduce the set of test functions $\Sc$ given by
 \[ \Sc := \left\{ H \in \Cc_{\mathbf{c}}^\infty( \R \times \T) \; \bigg{|} \; \; \forall (n,m,l) \in \mathbb{N}^3, \; \sup_{k \in \mathbb{T}} \sup_{u \in \mathbb{R}} \left\vert (1+u)^l \partial_k^n \partial_u^m H(u,k) \right\vert < \infty \right\} \ .\]
\begin{flushleft}
For any $H$ in $\Sc$, we denote its (continuous)  Fourier transform in the first variable by $\Fc \left[H\right] : \R \times \T \rightarrow \C$ where
\end{flushleft}
\[ \forall (p,k) \in \R\times \T, \quad \mathcal{F}\left[H\right](p,k) := \int_{\mathbb{R}} H(u,k) \exp(2\i \pi pu) \; du \ . \]
The set $\Sc$ is stable under the action of $\Fc$.\\
In the whole paper, for $H$ in $\Sc$, we denote the Laplacian of $H$ in the first variable by $\Delta \left[H\right]$ and for $\beta$ in $(1,2)$ the fractional Laplacian of $H$ in the first variable by $-(-\Delta)^{\frac{\beta}{2}}\left[H\right]$ where we recall that for any $p$ in $\R$ and $k$ in $\T$
\begin{align*}
 \mathcal{F} [\Delta \left[H\right]](p,k) &= - (2\pi p)^2  \mathcal{F}  \left[H\right](p,k) \ . \\  
 \mathcal{F} \left[-(-\Delta)^{\frac{\beta}{2}} \left[H\right]\right](p,k) &= - \vert 2\pi p \vert^{\beta} \mathcal{F}\left[H\right](p,k) \ .
\end{align*}
The space $\Sc\times \Sc$ is equipped with the norm $\Vert \cdot \Vert $ defined by
\begin{align*}
\forall J:=(J_1,J_2) \in \Sc \times \Sc, \quad \Vert J \Vert_{} =  \sum_{i=1}^2 \int_{\mathbb{R}} \sup_{k} \vert \mathcal{F}\left[J_i\right] (p,k) \vert \; dp \ . \end{align*}
The space $(\Sc\times \Sc , \Vert \cdot \Vert )$ is then a separable space $\left(\text{\cite[p. 572]{LP93}}\right)$. 
We denote by $(\Sc\times\Sc)^\prime$ the dual of $\Sc\times \Sc$ for the weak-* topology (we refer the reader to Sec.\;\ref{section_wigner} for a precise definition). For $\Wc$ in $(\Sc\times\Sc)^\prime$ and $J$ in $\Sc\times \Sc$ we denote by $\left\langle \Wc, J \right\rangle$ the duality bracket between $\Wc$ and $J$.\\

Throughout the article, the random variables will be defined on an underlying probability space $\left(\Omega, \mathcal{F}, \mathbb{P} \right)$ and $T$ will denote a fixed positive time.

\section{Microscopic dynamics}
\label{section_micro}
In this section, we define the microscopic dynamics studied in this paper. This dynamic was first introduced in \cite{BBO06,BBO09,BOS09} without a magnetic field and later in \cite{SS18,SSS19} with a magnetic field. 
\subsection{Deterministic dynamics}
We consider a one dimensional chain of coupled harmonic oscillators each having two transverse degrees of freedom and subject to the action of a magnetic field perpendicular to the plane of motion. At rest, the atoms are aligned according to the lattice $\mathbb{Z}$ and each $x$ in $\mathbb{Z}$ represents the balance position of one atom. We denote the velocity of the atom with rest position $x$ in $\mathbb{Z}$ by $p(x)=(p_1(x),p_2(x))$ in $\mathbb{R}^2$  and the displacement from its rest position by $q(x)=(q_1(x),q_2(x))$ in $\mathbb{R}^2$ . We denote the strength of the magnetic field by $B$ in $\mathbb R$. The deterministic dynamics is defined at any positive time $t$ and for any $i$ in $\{1,2\}$ by
\begin{equation}
\label{microscopic_system}
\begin{array}{l}
  \frac{d}{dt}q_{i}(t,x)  = p_{i}(t,x)\ , \\\\
  \frac{d}{dt}p_{i}(t,x)= \Delta_{d} \left[q_i(t,x)\right] +\delta_{i,1}B p_2(t,x) - \delta_{i,2}B p_1(t,x)\ ,
\end{array}
\end{equation}
where $\Delta_d$ denotes the discrete Laplacian on $\mathbb Z$.
We denote a typical configuration of the system by $(\mathbf{q},\mathbf{p}) := (q(x),p(x))_{x \in \mathbb{Z}}$ and the configuration  over time by $ \lbrace(\mathbf{q}(t),\mathbf{p}(t); t \geq 0)\rbrace := \lbrace(q(x,t),p(x,t))_{x \in \mathbb{Z}}\; | \; t \geq 0\}$. The initial configuration $(q(x,0),p(x,0))_{x \in \mathbb{Z}}$ is denoted by $(\mathbf{q}^0,\mathbf{p}^0)$.\\
Let $\alpha$ be the function defined on $\mathbb{Z}$ by
 \begin{align*} \forall z \in \mathbb{Z}, \quad \alpha(z) = \left\{ \begin{array}{l} 2\quad \text{if} \quad z=0 \ ,\\  -1 \quad \text{if} \quad \vert z \vert = 1 \ , \\ 0 \quad \text{otherwise} \ .
\end{array}
\right.
\end{align*}
Observe that 
\begin{equation}
\label{def_widehat_alpha}
\widehat{\alpha}(k)=4\sin^2(\pi k) \ ,
\end{equation} 
with $\widehat{\alpha}(0) = \widehat{\alpha}^\prime(0)= 0$ and $\widehat{\alpha}^{\prime\prime}(0) = 8\pi^2$.\\
The infinitesimal generator of this Markovian dynamics is given by $A+B G$ where for every smooth and local complex-valued functions\footnote{Here smooth and local means that for any $i$ in $\{1,2\}$, $\phi$ depends only on a finite number of the sequence $(q_i(x),p_i(x))_{x \in \Z}$ and is smooth with respect to these coordinates. } $\phi$ we have
\begin{align*} A\left[\phi\right] & =\frac{1}{2} \left(\sum_{x \in \mathbb{Z}} \sum_{i=1}^2 p_i(x) \partial_{q_i(x)}\left[\phi\right] + \sum_{x,x^{\prime} \in \mathbb{Z}} \sum_{i=1}^2\alpha(x-x^{\prime})q_i(x^{\prime}) \partial_{p_i(x)} \left[\phi\right]\right) \  \ . \\
G\left[\phi\right] &= \sum_{x \in \mathbb{Z}}\left(p_2(x)\partial_{p_1(x)}\left[\phi\right] -p_1(x) \partial_{p_2(x)} \left[\phi\right] \right) \ .
\end{align*}
We denote by $E(\mathbf{q},\mathbf{p})$ the total energy of the configuration $(\mathbf{q},\mathbf{p})$  where
\begin{align}
\label{definition_energy}
E(\mathbf{q},\mathbf{p})&= \frac{1}{2} \left(\sum_{x \in \mathbb{Z}} \vert p(x) \vert ^2 + \sum_{x \in \mathbb{Z}} \vert q(x)-q(x+1) \vert ^2 \right) \ .
\end{align}
In the whole paper, we will study only configurations $(\mathbf{q},\mathbf{p})$ with finite total energy. Observe that the energy is conserved during the time evolution, i.e.
\begin{align}
\label{conservation_energy}
 \forall t \in [0,T], \quad \frac{d}{dt}E(\mathbf{q}(t),\mathbf{p}(t)) &= 0 \ .  
 \end{align}
 
 \subsection{Eigenvalues and eigenvectors of the deterministic dynamics}
We define on $\mathbb{T}$ the following real valued functions 
 \begin{equation}
 \label{omega}
\omega_{1,B}(k) = \sqrt{\widehat{\alpha}(k) + \frac{ B^2  }{4}} + \frac{ B  }{2} \; \; \text{and} \; \; \omega_{2,B}(k) = \sqrt{\widehat{\alpha}(k) + \frac{ B^2   }{4}} - \frac{ B   }{2}  \ ,
\end{equation}
\begin{equation}
\label{theta}
 \forall i \in \lbrace1,2 \rbrace, \quad \theta_{i,B}(k)=\sqrt{\frac{\omega_{i,B}(k)}{\omega_{1,B}(k)+\omega_{2,B}(k)}}\ ,
 \end{equation}
 \begin{equation}
 \label{omega_prime}
\textbf{v}_{B}(k) := \frac{d\omega_{1,B}}{dk}(k) = \frac{d\omega_{2,B}}{dk}(k)= \frac{\widehat{\alpha}^\prime(k)}{2\sqrt{\widehat{\alpha}(k) + \frac{ B^2}{4}}} \ ,
 \end{equation}
 where $k$ is in $\T$. Observe that for any $i$ in $\{1,2\}$, $\theta_{i,B}$ is a bounded function. \\
 
\begin{flushleft}
For every configuration $(\mathbf{q},\mathbf{p})$ we define a couple $(\widehat{\psi}_{1,B},\widehat{\psi}_{2,B}) : \mathbb{T} \rightarrow \mathbb{C}^2  $ by
\end{flushleft} 

\begin{equation}
\label{def_psi_1}
\widehat{\psi}_{1,B}\left[\mathbf{q},\mathbf{p}\right](k):=\theta_{1,B}(k)\big[\widehat{p}_1(k)-\mathbf{i}\omega_{2,B}(k)\widehat{q}_1(k)+\mathbf{i}\widehat{p}_2(k) +\omega_{2,B}(k)\widehat{q}_2(k) \big] \ .
\end{equation}

\begin{equation}
\label{def_psi_2}
\widehat{\psi}_{2,B}\left[\mathbf{q},\mathbf{p}\right](k):=\theta_{2,B}(k)\big[\widehat{p}_1(k)-\mathbf{i}\omega_{1,B}(k)\widehat{q}_1(k)-\mathbf{i}\widehat{p}_2(k) -\omega_{1,B}(k)\widehat{q}_2(k) \big] \ .
\end{equation}
In order to lighten the notations, for any positive time $t$, $k$ in $\mathbb{T}$ and $i$ in $\{ 1,2 \}$ we will denote $\widehat{\psi}_{i,B}[q(t),p(t)](k)$ by $\widehat{\psi}_{i}(t,k)$. 
\begin{lemma} 
\label{eig}We have that for each $k$ in $\mathbb T$ and each $ i$ in  $\lbrace 1,2 \rbrace$,  $\widehat{\psi}_{i,B}(k)$ is an  eigenvector of $(A+B G)$ corresponding to  the eigenvalue $-\mathbf{i}\;\omega_{i,B}(k)$, i.e. 
 \begin{align*} \forall i \in \lbrace 1,2 \rbrace, \quad (A+B G)\left[\widehat{\psi}_{i,B}\right](k)=-\mathbf{i}\;\omega_{i,B}(k)\widehat{\psi}_{i,B}(k) \ .\end{align*}
\end{lemma}
\begin{proof}
    Let $i$ in $\{1.2\}$, $k$ in $\mathbb{T}$ and $t$ in $[0,T]$, using Eq.\;\eqref{microscopic_system}, Eq.\;\eqref{def_psi_1} and Eq.\;\eqref{def_psi_2}  we can prove that 
    \begin{equation}
       \partial_t \widehat{\psi}_{i}(t,k) = -\mathbf{i} \omega_{i,B}(k) \widehat{\psi}_{i}(t,k) \ . 
    \end{equation}
    This concludes the proof. 
\end{proof}
As it is explained in \cite[Section 3.2]{SSS19}, we normalised $\widehat{\psi}_{1/2,B}$ in Eq.\;\eqref{def_psi_1} and in Eq.\;\eqref{def_psi_2} by $\theta_{1,2}$ in order to have the following equality 
\begin{align*} E(t)= \int_\mathbb{T} \left( \left\vert \widehat{\psi}_1(t,k) \right\vert ^2 + \left\vert \widehat{\psi}_2(t,k) \right\vert ^2 \right) dk \ .\end{align*}
 
\subsection{Stochastic dynamics}
We introduce now the local stochastic perturbation which is defined through its infinitesimal generator $S$ whose action on any smooth and local complex valued functions $\phi$ is given by
\begin{equation*}
S\left[\phi\right] = \frac{1}{4} \sum_{x \in \mathbb{Z}} \sum \limits_{\underset{i \neq j}{i,j =1} }^2 \left(Y_{x,x+1}^{i,j}\right)^2\left[\phi\right] \ ,
\end{equation*}
where
\begin{eqnarray*}
Y_{x,x+1}^{i,j} & = &(p_j(x+1)-p_j(x))  (\partial_{p_i(x+1)} - \partial_{p_i(x)})\\
& - &(p_i(x+1)-p_i(x))(\partial_{p_j(x+1)} - \partial_{p_j(x)}) \ .
\end{eqnarray*}
The operator $S$  conserves the total energy $\left(\text{defined in Eq.\;\eqref{definition_energy}} \right)$ and the total pseudomomentum defined by 
\begin{align*}
    \mathcal{P}_m &= \left(\sum_{x \in \mathbb{Z}} p_1(x) - B q_2(x), \sum_{x \in \mathbb{Z}} p_2(x) + B q_1(x) \right) \ . 
\end{align*}
We introduce a scaling parameter $\varepsilon >0$  and we take $\gamma > 0$ which represents the intensity of the stochastic noise. We denote by $L^\varepsilon$ the  infinitesimal generator of the dynamics defined as follows
\begin{equation}
\label{generator_inf_dynamic}
L^\varepsilon = A + B  G + \varepsilon \gamma S\ .
\end{equation} 
We denote by $(\mathbf{q}^\varepsilon(t), \mathbf{p}^\varepsilon(t))_{t \geq 0}$ the Markovian dynamics generated by the  infinitesimal generator $L^\varepsilon$. For a rigorous definition of the dynamics we refer the reader to \cite[Sec 3.4]{SSS19} or to \cite[Chapter 6]{D14}.\\

We can check that $L^\varepsilon$ conserves the total energy and the total pseudo-momentum. Since the energy of the dynamics is conserved we  will denote it by $E^\varepsilon$. In order to simplify the notations for all $i$ in $\lbrace 1,2 \rbrace$, $k$ in $\T$ and positive time $t$, we will write  $\widehat{\psi}_{i}^\varepsilon(t,k)$ instead of $\widehat{\psi}_{i,B}\left[\mathbf{q}^\varepsilon (t),\mathbf{p}^\varepsilon (t)\right](k)$.

We assume that the initial configuration of the system $\left(\mathbf{q}^0,\mathbf{p}^0\right)$ is distributed according to a measure $\mu^\varepsilon$ which satisfies the following condition 
\begin{equation}
\label{hypo_distrib}
 K_0 = \sup_{0<\varepsilon<1}  \varepsilon \; \int_{\mathbb{T}}\mathbb{E}_{\mu^\varepsilon} \left[ \left\vert \widehat{\psi}^\varepsilon_1(0,k) \right\vert^2 +  \left\vert \widehat{\psi}^\varepsilon_2(0,k)\right\vert ^2 \right] \;dk \quad < \infty \ .
\end{equation}
Since the energy of the system is preserved, this condition is true at any time $t$, i.e.
\begin{equation} \label{K0}
\sup_{0<\varepsilon<1} \sup_{ t \geq 0} \; \varepsilon \; \int_{\mathbb{T}}\mathbb{E}_{\mu^\varepsilon} \left[ \vert \widehat{\psi}^\varepsilon_1(t,k) \vert^2 +  \vert \widehat{\psi}^\varepsilon_2(t,k)\vert ^2 \right] \; dk =K_0 \ .
 \end{equation} 
 \subsection{Wigner distribution}
 \label{section_wigner}
The space $(\Sc\times \Sc)^\prime $ is equipped with the weak-* topology, i.e. a sequence  $(\Wc^{N})_{N \in \mathbb{N}}=(\Wc_1^{N}, \Wc_2^N)_{N \in \mathbb{N}}$ in $(\Sc \times \Sc)^\prime$ converges to $\Wc$ in $(\Sc\times\Sc)^\prime$ if and only if for any $J:=(J_1,J_2)$ in $\Sc \times \Sc$ and $i$ in $\{ 1,2\}$
\begin{align*} \lim\limits_{ N \rightarrow \infty}  \left\vert \left\langle \Wc_i^N , J_i \right\rangle - \left\langle \Wc_i , J_i \right\rangle \right\vert = 0 \ .\end{align*}
We say then that a sequence $\left(\Wc^N\right)_{N \in \N}$ in $\Fc\left([0,T], (\Sc\times \Sc)^\prime\right)$ converges pointwise to $\Wc$ if and only if for any $t$ in $[0,T]$, $(\Wc^N(t))_{N \in\N}$ converges to $\Wc(t)$ in $(\Sc\times \Sc)^\prime$.

\medskip
For each $t$ in $[0,T]$, we define the Wigner distribution, denoted by $\Wc^\varepsilon(t)$, as the element of $(\Sc\times\Sc)^\prime$ defined for any $J=(J_1,J_2)$  in $\Sc \times \Sc$ by 
\[ \langle\mathcal{W}^\varepsilon(t) , J \rangle = \sum_{i=1}^2 \langle \Wc_i^\varepsilon(t),  J_i \rangle \ , \]
\vspace{-0.1cm}
where, for $i$ in $\{ 1,2 \}$,
 \begin{align}
&\langle\Wc_i^{\varepsilon}(t),J_i \rangle \label{wigner_def}  \\
& =  \frac{\varepsilon}{2} \sum_{x,x^{\prime} \in \mathbb{Z}} \mathbb{E}_{\mu^\varepsilon}\left[ \psi^\varepsilon_i \left(t\varepsilon^{-1},x^\prime\right)^* \psi^\varepsilon_i \left(t\varepsilon^{-1},x\right)\right] \int_{\mathbb{T}} e^{2\mathbf{i} \pi (x^{\prime}-x)k} J_{i}\left(\tfrac{\varepsilon}{2}(x+x^{\prime}),k\right)^* \; dk \nonumber\\
 & =  \frac{\varepsilon}{2} \int_{\mathbb{R}} \int_{\mathbb{T}}\E_{\mu^\varepsilon} \left[\widehat{\psi}^\varepsilon_i\left(t\varepsilon^{-1},k-\tfrac{\varepsilon p}{2}\right)^*\widehat{\psi}^\varepsilon_i\left(t\varepsilon^{-1},k+\tfrac{\varepsilon p}{2}\right)\right] \mathcal{F}\left[J_i\right](p,k)^* \;dk dp \ .  \nonumber
\end{align}
The well-posedness of these expressions are proved in Appendix \ref{appendix_kinetic}.
The following lemma proves that the Wigner distribution is an element of $\Cc\left([0,T], (\Sc\times \Sc)^\prime \right)$ and gives some of its properties.
\begin{lemma}\label{Lem_Wigner} The Wigner distribution satisfies the following properties
\begin{itemize}
\item[i)] For all $t$ in $[0, T]$, $\Wc^\varepsilon (t)$ belongs to  $(\Sc\times \Sc)^\prime$.  
\item[ii)] For any $J$ in $\Sc\times \Sc$, the family $(\left\langle \mathcal{W}^\varepsilon, J \right\rangle)_{\varepsilon > 0}$ is bounded in $(\Cc ([0,T], \C), \Vert \cdot \Vert_\infty)$ and
\begin{align*} \Vert \left\langle \mathcal{W}^\varepsilon , J \right\rangle \Vert_\infty \lesssim K_0 \Vert J \Vert \ .\end{align*}
Furthermore, the application $t \mapsto  \mathcal{W}^\varepsilon(t)$ belongs to $\Cc ([0,T], (\Sc\times \Sc)^\prime)$.
\end{itemize}
\end{lemma}
\begin{proof}
We refer the reader to Appendix \ref{appendix_properties_Wigner}.
\end{proof}
\begin{flushleft}
To study the asymptotic behavior of $(\Wc^\varepsilon)_{ \varepsilon >0}$ we need to introduce two distributions on $\Sc\times \Sc$ denoted by $\mathcal{A}^\varepsilon$ and $\left({\mathcal{A}}^\varepsilon\right)^*$ defined for any $J:=(J_1,J_2)$ in $\Sc \times \Sc$ by
\end{flushleft}
\begin{equation}
\langle\mathcal{A}^\varepsilon , J \rangle = \sum_{i=1}^2\langle\mathcal{A}_i^\varepsilon,J_i \rangle \quad  \text{and} \quad \langle\left({\mathcal{A}}^\varepsilon\right)^* , J \rangle = \sum_{i=1}^2\langle\left({\mathcal{A}}^\varepsilon\right)^*,J_i \rangle \ ,
\end{equation}
where, for any $i$ in $\{1,2\}$, by letting $i^*=3-i$, we set
\begin{align}
&\langle\mathcal{A}_i^\varepsilon(t),J_i \rangle \label{def_anti_wigner}\\
& =\frac{\varepsilon}{2} \int_{\mathbb{R}} \int_{\mathbb{T}} \mathbb{E}_{\mu^\varepsilon}\left[ \widehat{\psi}^\varepsilon_i \left(t\varepsilon^{-1},\frac{\varepsilon p}{2}-k \right)\widehat{\psi}^\varepsilon_{i^*} \left(t\varepsilon^{-1},k+\frac{\varepsilon p}{2} \right)\right] \mathcal{F}\left[J_i\right](p,k)^* \; dk dp \nonumber \ ,
\end{align}
and
\begin{align}
\label{def_anti_wigner_conjugate}
&\langle\left({\mathcal{A}}_i^\varepsilon\right)^*(t),J_i \rangle\\
&=\frac{\varepsilon}{2} \int_{\mathbb{R}} \int_{\mathbb{T}} \mathbb{E}_{\mu^\varepsilon}\left[ \widehat{\psi}^\varepsilon_i\left(t\varepsilon^{-1},k-\frac{\varepsilon p}{2}\right)^*\widehat{\psi}^\varepsilon_{i^*}\left(t\varepsilon^{-1},-k-\frac{\varepsilon p}{2} \right)^*\right] \mathcal{F}\left[J_i\right](p,k)^* \; dk dp \ . \nonumber
\end{align}
\section{Review of previous results}
In this section, we recall some results on this lattice. We first recall Theorem \ref{boltzmann_equa_theore} from \cite{BOS09,SSS19} which states the convergence of the Wigner distribution to the solution $f_B$ of some phonon Boltzmann equation \eqref{IBE}. Then we recall Theorem \ref{thm_fractionnal_diffusion_old} and Theorem \ref{thm_limi_hydro_u_old} from \cite{JKO09,SSS19} which show the convergence of $f_B$ in some hydrodynamic scaling to the solution of some fractional diffusion equations \eqref{old_fractionnaire}.
\label{section_previous}
\subsection{Kinetic limit of the Wigner distribution}
Let $J:=(J_1,J_2)$ in $\Sc\times \Sc$ and recall Eq.\;\eqref{omega}, \eqref{theta} and \eqref{omega_prime}. We define a collisional operator $C_B: J=\left( J_1,J_2 \right) \in  \Sc \times \Sc \to \left(\left[C_BJ\right]_1, \left[C_BJ\right]_2 \right) \in \Sc \times \Sc $ in the following way. For any $u$ in $\R$, $k$ in $\T$, $i$ in $\{1,2\}$ and $J:=(J_1,J_2)$ in $\Sc \times \Sc$, 
\begin{equation}
\label{collisional_operator}
[C_BJ]_i (u,k) =\sum_{j=1}^2 \int_{\mathbb{T}} \theta_{i,B}^2(k) R(k,k^\prime) \theta_{j,B}^2(k^\prime) \left[ J_j\left(u,k^\prime\right)- J_i(u,k)   \right] \; dk^\prime \ ,
\end{equation}
with 
\begin{equation}
\label{def_R_historique}
\forall (k,k^\prime) \in \T \times \T, \quad R(k,k^\prime) = 16\sin^2(\pi k) \sin^2(\pi k^\prime) \ .
\end{equation}
Let $\Wc$ in $\mathcal{C}\left([0,T], (\Sc\times \Sc)^\prime \right)$ and $\Wc^0$ in $(\Sc\times \Sc)^\prime$. 
\begin{itemize}
\item[i)]We say that $\Wc$ is a weak solution on $[0,T]$  of the linear Boltzmann equation
\begin{equation} \label{BE} \partial_t \Wc + \frac{1}{2\pi} \textbf{v}_B \partial_u \Wc =  \gamma C_B \Wc \ ,
\end{equation}
with $\Wc^0$ as initial condition if and only if for any $J$ in $\Sc\times \Sc$ and any $t$ in $[0,T]$
\begin{eqnarray}
\label{IBE}
\left\langle \Wc(t), J \right\rangle - \left\langle \Wc(0), J \right\rangle & = & \int_0^t \frac{1}{2 \pi} \left\langle \Wc(s),  \textbf{v}_B \partial_u J \right\rangle \; ds \\
&  + &  \gamma \int_0^t \left\langle \Wc(s), C_BJ \right\rangle \; ds \ . \nonumber
\end{eqnarray}
\item[ii)] We say that $\Wc$ is a \textbf{Borel measure valued} weak solution on $[0,T]$ of the linear Boltzmann equation (\ref{BE}) if it is a weak solution of (\ref{BE}) such that for all $t$ in $[0,T]$ and $i$ in $\{1,2\}$, $\Wc_i(t)$ is a bounded Borel measure on $\R \times \T$.
\end{itemize} 
\begin{lemma} \label{lemma_unicite} Let $\mu^0:=(\mu_1^0, \mu_2^0)$ in $(\Sc\times \Sc)^\prime$ be a couple of bounded Borel measure on $\R \times \T$. Then, there exists a unique \textbf{Borel measure valued} weak solution $\mu$ on $[0,T]$ to the linear Boltzmann equation \eqref{BE} with $\mu^0$ as initial condition.
\end{lemma}
\begin{proof}
We refer the reader to Appendix \ref{appendix_unicite_Boltzmann}.
\end{proof}
 The following theorem summarizes the results obtained by the authors of \cite[Theorem 1]{SSS19}  and by those of \cite[Theorem 9]{BOS09} respectively.
\begin{thm}
\cite{BOS09,SSS19} \label{boltzmann_equa_theore} Let $T >0$. Assume that the condition (\ref{hypo_distrib}) holds and that $(\mathcal{W}^\varepsilon(0))_{\varepsilon >0}$ converges in $(\Sc\times \Sc)^\prime$ to a bounded positive distribution $\mathcal{W}^0$.
\begin{itemize}
\item[i)]If $B \neq 0$, then there exists $\Wc$ in $\Cc\left([0,T], (\Sc\times \Sc)^\prime\right)$ such that $(\Wc^\varepsilon)_{\varepsilon > 0}$ converges pointwise to $\Wc$ in $\Fc\left([0,T], (\Sc\times \Sc)^\prime\right)$. 
Moreover for each time $t$ in $[0,T]$, the limit $\mathcal{W}(t)$ and $\mathcal{W}^0$ can be extended to a couple of bounded Borel measures on $\R \times \T$ respectively denoted by $\mu(t):=(\mu_1(t),\mu_2(t))$ and $\mu^0:=(\mu_1^0,\mu_2^0)$.
Furthermore, $\mu$ belongs to $\Cc\left([0,T], (\Sc\times \Sc)^\prime \right)$ and is the unique \textbf{Borel measure valued} weak solution of the  Boltzmann equation (\ref{BE}) with initial condition $\mu^0$. 
\item[ii)] If $B=0$ and furthermore
\begin{equation}
\label{hypo_Stefano}
\lim\limits_{ \rho \rightarrow 0} \limsup\limits_{ \varepsilon \rightarrow 0} \frac{\varepsilon}{2} \int_{ \vert k \vert < \rho}\mathbb{E}_{\mu^\varepsilon}\left[ \left\vert\widehat{\psi}^\varepsilon_i\left(0,k\right) \right\vert^2 \right] dk = 0 \ ,
\end{equation}
then the same conclusion as in i) holds with $B=0$ in Eq.\;\eqref{BE}.
\end{itemize} 
\end{thm}
\begin{rmk}
When $B=0$, the assumption \eqref{hypo_Stefano} is required to compensate the lack of differentiability of the function $\mathbf{v}_{B}$ defined in Eq.\;\eqref{omega_prime} at $k=0$.
\end{rmk}
\subsection{Hydrodynamic limit of the Boltzmann equation}
\label{section_previous_results_hydro}
The aim of this section is to study the hydrodynamic limit of the Boltzmann equation \eqref{IBE}. Let us start by a short reminder about L\'evy processes \cite{B96}. Given a measure $\nu$ on $\mathbb{R}^*$, we say that $\nu$ is a L\'evy measure if and only if 
 \begin{equation}
 \label{def_Levy_measure_general}
 \int_{\mathbb{R}^*} \min \left( 1, r^2\right) \; d\nu(r) < \infty \ .
\end{equation} 
Let $Y_u(\cdot)$ be a real valued stochastic process starting from $u$ in $\mathbb{R}$. We say that $Y_u(\cdot)$ is a L\'evy process with (L\'evy) measure $\nu$ if and only if for any positive $t$ and $\theta$ in $\R$,
\begin{align*} 
\E\left[ \exp\left(\mathbf{i}\theta Y_u(t) \right)  \right] = \exp\left(t \Phi_{Y}(\theta) + \mathbf{i}\theta u \right)  \ . 
\end{align*}
Here, $\Phi_Y$ denotes the L\'evy exponent associated to the L\'evy process $Y_u(\cdot)$ and is given for any $\theta$ in $\R$ by
\begin{equation}
\label{fonction_caract_Levy_process_general_intro_chapter_2}
\Phi_{Y} (\theta) = a \theta^2 + \int_{\mathbb{R}^*}\left(\exp\left( \mathbf{i} \theta r \right)-1+\mathbf{i}\theta r \mathds{1}_{ \lbrace \vert r \vert < 1 \rbrace } \right) \; d\nu(r) \ ,
\end{equation} 
where $a$ is in $\mathbb{R}^+$.
The action of the infinitesimal generator $\mathfrak{L}$ of $Y_u(\cdot)$ on a smooth function $\phi : \mathbb{R} \rightarrow \mathbb{R}$ which decays sufficiently fast is given by
\begin{equation}
\label{generator_Levy_process_intro_chapter_2}
\forall p \in \R, \quad \mathfrak{L}\left[\phi\right](p) =  \int_{\mathbb{R}} \Fc\left[\phi\right](\xi) \Phi_{Y}(\xi) \exp\left( 2 \mathbf{i}\pi p \xi\right) \; d\xi \ .
\end{equation}
In this article we will only study pure jumps Lévy processes, i.e. $a=0$. 
Let $\beta$ in $(1,2)$, and assume that $d\nu(r):=\vert r \vert^{-\beta-1} \; dr$ then, up to a constant, we have
\begin{align*}
\forall \theta \in \mathbb{R}, \quad \Phi_{Y} (\theta) =  - \vert \theta \vert^\beta \ .
\end{align*}
For this choice of $\nu$, the infinitesimal generator of \,$Y_u(\cdot)$ is (up to a constant) the pseudo-differential operator $-\left( -\Delta \right)^{\tfrac{\beta}{2}}$. In this case, we say that $Y_u(\cdot)$ is an $\beta$-stable Lévy process. 

\medskip 

\begin{flushleft}
For $i$ in $\{1,2\}$ and $\left(k,k^\prime\right)$ in $\T^2$, we recall that $\theta_{i,B}(k)$ and $R\left(k,k^\prime\right)$ are defined in Eq.\;\eqref{theta} and Eq.\;\eqref{def_R_historique} respectively. Let $f : \mathbb{T}\times \{1,2\} \rightarrow \mathbb{R}$ be a real-valued function such that for any $i$ in $\{1,2\}$  $f(\cdot,i) \in \Cc \left( \T , \mathbb{R}\right) $. We define an operator $\Lc_B$ acting on $f$   by
\end{flushleft}
\begin{align}
    \label{def_infinit_generator_boltzmann}
\Lc_B [f](k,i) & = \lambda^{-1}_{B}(k,i) \sum_{j=1}^2 \int_{\mathbb{T}} P_B\left(k,i,dk^\prime,j\right) \left( f\left(k^\prime,j\right)- f(k,i) \right) \; dk^\prime \ .
\end{align} 
Here for any $\left(k,k^\prime\right)$ in $\T^2$ and $(i,j)$ in $\{1,2\}^2$
\begin{align}
\label{transition_prob_stefano}
P_B\left(k,i,dk^\prime,j\right) &= \gamma\lambda_B(k,i) \theta_{i,B}^2(k) \theta_{j,B}^2\left(k^\prime\right) R\left(k,k^\prime\right) \; dk^\prime \ ,
\end{align}
where
\begin{align} 
\label{def_lambda}
\lambda_B(k,i) &= \left[\gamma\theta_{i,B}^2(k)R(k)\right]^{-1} \quad \text{and } \quad R(k) = \int_{\T} R\left(k,k^\prime\right) \; dk^\prime \ .
\end{align}

\medskip
Let $\mu$ be the \textbf{Borel measure valued} weak solution of the Boltzmann equation \eqref{BE}  with initial condition $\mu^0$. Let $t$ in $[0,T]$ and $i$ in $\{ 1,2 \}$,  we assume  that $\mu_i(t)$ and $\mu_i^0$ have a density with respect to the measure $dudk$ denoted by $f_B(t,u,k,i)$ and $f^0(t,u,k,i)$ respectively. Then we have for any time $t$ in $[0,T]$, $k$ in $\T$, $u$ in $\R$ and $i$ in $\{1,2\}$
\begin{equation}
\label{equation_u_historique}
\partial_t f_B(t,u,k,i) + \frac{\textbf{v}_B(k)}{2\pi}  \partial_u f_B(t,u,k,i) = \mathcal{L}_B\left[f_B\right](t,u,k,i) \ ,
\end{equation}
with $f^0(u,k,i)$ as initial condition. To study the hydrodynamic behavior of $f_B$ we interpret the operator $\Lc_B$ as the infinitesimal generator of a pure jump continuous time Markov process on $\mathbb T \times \{1,2\}$ as follows. 

\medskip
Let $\left(X_B^n\right)_{n \in \N}:=\left(K_B^n, I_B^n\right)_{n \in \N}$ be the Markov chain on $\mathbb{T} \times \{1,2\}$ with transition probability $P_B$ defined in Eq.\;\eqref{transition_prob_stefano}.
Let $(\tau^n)_{n \in \N}$ be an i.i.d sequence of random variables, independent of $\left(X_B^n\right)_{n \in \N}$ such that $\tau^0 \sim \mathcal{E}(1)$. We define the random variable $\Tc^N$ by\footnote{We decided not to write $\mathcal{T}_B^N$ in order to lighten the notations.}
\begin{equation}
\label{definition_temps_saut}
\Tc^0 = 0 \quad \text{and} \quad \forall N \in \N \backslash \{ 0 \}, \quad \Tc^N = \sum_{n=1}^N \lambda_B\left(X_B^{n-1}\right) \tau^{n-1} \ .
\end{equation}   
Then we can define a pure jump Markovian process $\left( K_B(\cdot),I_B(\cdot)\right)$ with values in $\T \times \{1,2\}$ where for any positive time $t$ in $[0,T]$
\begin{align*}
K_B(t) &= K_B^n, \quad \text{and} \quad  I_B(t) = I_B^n, \quad \forall t \in \left[\Tc^n,\Tc^{n+1}\right[ \ .
\end{align*}
The infinitesimal generator of $\left( K_B(\cdot),I_B(\cdot)\right)$ is $\Lc_B$ defined in Eq.\;\eqref{def_infinit_generator_boltzmann}. From this process, we can define an additive functional of Markov process $Z_{u,B}(\cdot)$ such that for any time $t$ in $[0,T]$ and $u$ in $\mathbb{R}$
\begin{align*}
Z_{u,B}(t) &= u - \int_0^t \frac{\textbf{v}_B\left(K_B(s)\right)}{2\pi} \; ds \ .
\end{align*}
Then by Dynkin's formula we get that for any $i$ in $\{1,2\}$ and $(u,k)$ in $\mathbb{R} \times \mathbb{T}$
\begin{equation}
f_B(t,u,k,i) = \mathbb{E}_{(k,i)}\left[f^0\left(Z_{u,B}(t),K_B(t),I_B(t)\right)   \right] \ .
\end{equation} 
The hydrodynamic behavior of $f_B$ is completely determined by the one of the process $\left( K_B(\cdot),I_B(\cdot)\right)$. Hence, in the following we recall how the authors of \cite{JKO09,SSS19} studied this process. 
Let  $\pi_B$ the probability measure on $\mathbb{T} \times \{ 1,2 \}$ defined as follows
\begin{equation}
\label{invariante_measure_delta}
\pi_B(dk, di) =\sum_{j=1}^2 \frac{\lambda_B(k,j)^{-1}}{\gamma \bar{R}} \; dk \delta_j(di) \quad \text{with} \quad \overline{R}=\int_{\T} R(k) \; dk \ .
\end{equation}
We can prove that $\pi_{B}$ is a reversible probability measure of the Markov chain $\left(X_B^n\right)_{n \in \N}$. Observe that
\begin{equation} 
\label{def_P}
P_B(k,i,dk^\prime,j)= \pi_B(dk^\prime,j) \ ,
\end{equation} 
where $P_{B}$ is defined in Eq.\;\eqref{transition_prob_stefano}.
\begin{rmk}
According to Eq.\;\eqref{def_P}, $\left(X_B^n\right)_{n \geq 1}$ is an i.i.d sequence of random variables on $\mathbb{T}\times \{1,2 \}$ but observe that $\left(X_B^n\right)_{n \in \mathbb{N}}$ is not because of $X_B^0$. This property will play a crucial role in the proof of our main results stated in Sec.\;\ref{section_main_result_B}.
\end{rmk}
We define a function $\Psi_B$ in the following way
\begin{equation}
\forall k \in \mathbb{T}, \quad \forall i \in \{1,2\}, \quad \Psi_B(k,i) = \textbf{v}_B(k) \lambda_B(k,i) \ .
\end{equation}
The asymptotic behavior of the process $Z_{u,B}(\cdot)$ is fully determined by the tails of the function $\Psi_{B}$. In \cite{JKO09,SSS19} it is proved that 
 \begin{align}
\forall r >0, \quad  \lim\limits_{ N \rightarrow \infty} N^{\beta_B} \;\pi_{B} \left(\left\{ (k,i)\; \big| \quad \Psi_{B}(k,i) > N r  \right\} \right) &= \tilde{\kappa}_B \vert r\vert^{-\beta_B}\ , \\
 \forall r <0, \quad  \lim\limits_{ N \rightarrow \infty} N^{\beta_B} \;\pi_{B} \left(\left\{ (k,i)\; \big| \quad \Psi_{B}(k,i) < N r  \right\} \right) &= \tilde{\kappa}_B\vert r\vert^{-\beta_B} \ ,
 \end{align}
where 
\begin{align}
    \beta_B&= \frac{3}{2} \quad \text{if} \quad  B = 0 \quad \text{and} \quad \beta_B= \frac{5}{3}, \quad \text{if} \quad B\neq0 \ ,
  \label{expo_previous_result_hydro}
\end{align}
and
\begin{align}
\label{old_tails}
\tilde{\kappa}_B & = \left\{
    \begin{array}{ll}
        \tilde{\kappa}_1 (\widehat{\alpha}^{\prime \prime}(0))^{\tfrac{3}{4}}\gamma^{-\tfrac{3}{2}} \quad \text{if $B =0$} \ ,\\\\
        \tilde{\kappa}_2 \vert B \vert^{-\tfrac{1}{3}}\widehat{\alpha}^{\prime \prime}(0)\gamma^{-\tfrac{5}{3}} \quad \text{if $B\neq0$}\ ,
    \end{array} \right. 
\end{align}
with $\tilde{\kappa_1}$ and $\tilde{\kappa}_2$ two positive constants.
Then the following results are proved in \cite[Theorem 3.1] {JKO09} and \cite[Theorem 3]{SSS19}  respectively.
\begin{thm}[\cite{JKO09,SSS19}] 
\label{thm_fractionnal_diffusion_old} Let $u$ in $\R$ and define $u_N:=Nu$. We assume that $X_B(0)=\left(K_B(0), I_B(0)\right)=(k,i)$ with $k \neq 0$ and $i$  in $\{1,2\}$.\\
Then, under $\P_{(k,i)}$ the finite-dimensional distributions of the scaled process $N^{-1}Z_{u_N,B}\left(N^{\beta_B} \cdot\right)$ converge weakly  to the finite-dimensional distributions of a L\'evy process $Y_{u,B}(\cdot)$ generated by $-2 \gamma \tilde{\kappa}_{B}(-\Delta)^{\frac{\beta_B}{2}}$ where $\tilde{\kappa}_B$ and $\beta_B$ are defined in Eq.\;\eqref{expo_previous_result_hydro} and Eq.\;\eqref{old_tails} respectively.
\end{thm}
From these results, the authors of \cite{JKO09,SSS19} have been able to prove the following theorem.
\begin{thm}[\cite{JKO09,SSS19}] 
\label{thm_limi_hydro_u_old}
Let $f_B(\cdot,\cdot,\cdot):=\left(f_B(\cdot,\cdot,\cdot,1),f_B(\cdot,\cdot,\cdot,2)\right)$ be the solution of Eq.\;\eqref{equation_u_historique} with initial condition $f^{0,N}(\cdot,\cdot)$ $:=(f^{0,N}(\cdot,\cdot,1),f^{0,N}(\cdot,\cdot,2))$ defined as follows
\begin{align}
    \forall (u,k) \in \mathbb{R}\times \mathbb{T}, \quad f^{0,N}(u,k) = \left( f^0\left( \frac{u}{N},k,1 \right), f^0\left( \frac{u}{N},k,2 \right) \right) \ .
\end{align}
Here $(f^0(\cdot,\cdot,1),f^0(\cdot,\cdot,2))  \in \mathcal{C}^\infty_{\mathbf{c}} \left( \mathbb{R} \times \mathbb{T}  \right)^2$. Let $\overline{f}^0 : \mathbb{R} \rightarrow \mathbb{R}$ be the real-valued function defined for any $u$ in $\mathbb{R}$
\begin{align*} 
\overline{f}^0(u)& = \sum_{i=1}^2 \int_{\T } f^0(u,k,i) \; dk \ . \end{align*}
Then, for any time $t$ in $]0,T]$ and $u$ in $\mathbb{R}$ we have that
\begin{align*} 
 \lim\limits_{N \rightarrow \infty} \  \sum_{i=1}^2 \int_\T \left\vert f_B\left(N^{\beta_B}t, Nu,k,i   \right) - \frac{1}{2}\rho_B(t,u) \right\vert^2 \; dk = 0 \ , 
\end{align*}
where we recall that $\beta_B$ is defined in Eq.\;\eqref{expo_previous_result_hydro} and $\rho_B$ is the solution of 
\begin{equation}
\label{old_fractionnaire}
\forall t \in \; ]0,T], \quad \forall u \in \R, \quad \left\{
    \begin{array}{ll}
        \partial_t \rho_B(t,u) = -D_B(-\Delta)^{\frac{\beta_B}{2}} \left[\rho_B\right](t,u) \ ,\\
        \rho_B(0,u) = \overline{f}^0(u) \ ,
    \end{array}
\right.
\end{equation}
with $D_B$ a positive constant.
\end{thm}
\begin{rmk}
As mentioned in the introduction, in \cite{JKO09} the result is stated for a unidimensional Boltzmann equation, i.e , $f_B$ is a one component function. However, following the proof given in \cite{JKO09,SSS19} we can extend the result to a two components function. 
\end{rmk}
\begin{rmk}
In \cite{SSS19}, the authors stated that $D_B= \vert B\vert^{-\frac{1}{3}}C$ where $C$ was a non explicit positive constant. From this observation, we can have an intuition on the scaling in time we have to do in order to observe the transition between these two fractional diffusion equations when we replace the value of the magnetic field $B$ by $BN^{-\delta}$. In  Eq.\;\eqref{def_vrai_constante} we give the explicit values of the constant $C$.
\end{rmk}

\section{Statement of the results}
\label{section_result}
In this section, we state our main results. Theorem \ref{main_theorem_kinetic} shows that at kinetic time scale $\varepsilon^{-1}$ under assumption \eqref{hypo_small_k} there is no transition in the convergence of the Wigner distribution. Theorem \ref{main_theorem_hydro} and Theorem \ref{main_thm_limit_u} show that a transition can be observed at some hydrodynamic time scale with the appearance of an interpolation process for some critical intensity of the magnetic field. We end this section with Theorem \ref{main_thm_limit_B_u_barre} which shows that this interpolation process converges to the fractional process studied in \cite{JKO09} when $B$ goes to zero and to the one studied in \cite{SSS19} when $B$ is sent to infinity.
\subsection{\textsc{Triviality of the transition in the kinetic time scale}}
\label{section_kinetic_transition}
In this section, we assume that the microscopic dynamics $\left(\text{see Eq.}\;\eqref{microscopic_system} \right)$ is submitted to a magnetic field of intensity $B_\varepsilon:=B\varepsilon^\delta$ where $\delta >0$ and $B> 0\footnote{In order to lighten the presentation, we choose to only consider the case $B >0$. However, the case $B<0$ can be treated in a similar way.}$.
\begin{thm}
\label{main_theorem_kinetic}
Let $\delta >0$, $ \eta$ in $]0,1[$ such that $\delta > \eta$ and $\kappa > 0$. We define $\kappa_\varepsilon:=\kappa \varepsilon^\eta$ and  we assume that the condition (\ref{hypo_distrib}) holds and that $(\mathcal{W}^\varepsilon(0))_{\varepsilon >0}$ converges in $(\Sc\times \Sc)^\prime$
to a bounded positive distribution $\mathcal{W}^0$. We assume furthermore that 
\begin{equation}
\label{hypo_small_k}  \limsup\limits_{ \varepsilon \rightarrow 0} \frac{\varepsilon}{2} \int_{ \vert k \vert < \kappa_\varepsilon}\mathbb{E}_{\mu^\varepsilon}\left[ \left\vert\widehat{\psi}^\varepsilon_i\left( 0,k\right) \right\vert^2 \right] \; dk =0  \ .
\end{equation}
Then, there exists $\Wc$ in $\Cc\left([0,T], (\Sc\times \Sc)^\prime\right)$ such that $(\Wc^\varepsilon)_{\varepsilon > 0}$ converges pointwise to $\Wc$ in $\Fc\left([0,T], (\Sc\times \Sc)^\prime\right)$. 
Moreover for each time $t$ in $[0,T]$, the limit $\mathcal{W}(t)$ and $\mathcal{W}^0$ can be extended to couples of bounded Borel measures on $\R \times \T$ denoted respectively by $\mu(t):=(\mu_1(t), \mu_2(t))$ and $\mu^0:=(\mu_1^0,\mu_2^0)$.\\
Furthermore, $\mu$ belongs to $\mathcal{C}\left([0,T], (\Sc\times \Sc)^\prime \right)$ and is the unique \textbf{Borel measure valued} weak solution on $[0,T]$ of the  Boltzmann equation \eqref{BE} with $B=0$ and initial condition $\mu^0:=(\mu^0_1, \mu_2^0)$. 
\end{thm}
\begin{proof}
We refer the reader to Sec.\;\ref{section_proof_kinetic_transition} for the proof of Theorem \ref{main_theorem_kinetic}.
\end{proof}
From Theorem \ref{main_theorem_kinetic}, we deduce that under assumption \eqref{hypo_small_k} the transition in the kinetic time scale $\varepsilon^{-1}$ between the case of \cite{BOS09} (zero magnetic field) and \cite{SSS19} (magnetic field of order one) is trivial in the sense that it holds for $\delta =0$. We show in Theorem \ref{main_thm_limit_u} that it is not the case in a longer time scale. 
To prove Theorem \ref{main_theorem_kinetic} we need the following lemma which shows that the assumption \eqref{hypo_small_k} can be extended to times $t\varepsilon^{-1}$. 
\begin{lemma}
\label{lemma_generalisation_hypo_small_k}
Let $i$ in $\{ 1,2 \}$, then for any time $t$ in $[0,T]$
\begin{align*} \limsup\limits_{ \varepsilon \rightarrow 0} \frac{\varepsilon}{2} \int_{ \vert k \vert < \kappa_\varepsilon}\mathbb{E}_{\mu^\varepsilon}\left[ \left\vert\widehat{\psi}^\varepsilon_i\left( t\varepsilon^{-1},k\right) \right\vert^2 \right] \; dk = 0 \ , \end{align*} 
where we recall that $\kappa_\varepsilon = \kappa \varepsilon^\eta$ with $\eta$ in $]0,1[$ such that $\delta > \eta$ and $\kappa > 0$.
\end{lemma}
\begin{proof}
We refer the reader to Appendix \;\ref{appendix-generalisation_hypo_small_k} for the proof of Lemma \ref{lemma_generalisation_hypo_small_k}.
\end{proof}
\begin{rmk}
Observe that assumption \eqref{hypo_small_k} is weaker than the assumption \eqref{hypo_Stefano} made in \cite{BOS09} to compensate the lack of differentiability of the function $\mathbf{v}_0$.
\end{rmk}
\subsection{\textsc{Transition in the hydrodynamic time scale}}
\label{section_hydro_transition}
In this section, we study the Markov chain $\left(X_B^n\right)_{n \in \N}$ introduced in Sec.\;\ref{section_previous_results_hydro} but with a magnetic field of intensity $B_N:=BN^{-\delta}$ where $\delta$ is in  $\R^+$ and $B >0\footnote{As in the previous section, the case $B<0$ can be studied in a similar way.}$. Hence, for each $N$ we will denote this chain by $\left(X_{B_N}^n  \right)_{ n \in \N}:=\left(K_{B_N}^n,I_{B_N}^n\right)_{ n \in \N}$. In Eq.\;\eqref{omega}, Eq.\;\eqref{theta} and Eq.\;\eqref{omega_prime}  we replace $B$ by $B_N$ and in all the functions defined in Sec.\;\ref{section_previous_results_hydro}. 
The aim of this section, is to study the behavior of $f_{B_N}(t,u,k,i)$ where for any positive time $t$, $k$ in $\T$, $u$ in $\R$ and $i$ in $\{1,2\}$ 
\begin{equation}
\label{equation_u_N}
\partial_t f_{B_N}(t,u,k,i) + \frac{\textbf{v}_{B_N}(k)}{2\pi}  \partial_u f_{B_N}(t,u,k,i) =  \Lc_{B_N}\left[f_{B_N}\right](t,u,k,i) \ ,
\end{equation}
where we recall that $\textbf{v}_{B_N}$  and $\Lc_{B_N}$ are defined in Eq.\;\eqref{omega_prime} and Eq.\;\eqref{def_infinit_generator_boltzmann} respectively.
By Dynkin's formula we get that for any time $t$ in $[0,T]$, $u$ in $\R$, $k$ in $\T$ and $i$ in $\{1,2\}$
\begin{equation}
f_{B_N}(t,u,k,i) = \mathbb{E}_{(k,i)}\left[f^0\left(Z_{u,B_N}(t),K_{B_N}(t),I_{B_N}(t)\right)   \right] \ .
\end{equation} 
Here for $t$ in $[0,T]$ $u$ in $\R$ and $N$ in $\mathbb{N}$ we recall that
\begin{align}
\label{stich}Z_{u,B_N}(t) &= u - \int_0^t \frac{\textbf{v}_{B_N}\left(K_{B_N}(s)\right)}{2\pi}  \; ds \ , 
\end{align}
where
\begin{align*}
K_{B_N}(t) &= K_{B_N}^n, \quad \text{and} \quad 
I_{B_N}(t) = I_{B_N}^n, \quad \forall t \in \left[\Tc_{B_N}^n,\Tc^{n+1}_{B_N}\right[ \ ,
\end{align*}
with
\begin{align*}
\Tc_{B_N}^0 = 0 \quad \text{and} \quad \forall n \in \N \backslash \{ 0 \}, \quad  \Tc_{B_N}^{n} &= \sum_{m=1}^n \lambda_{B_N}\left(X_{B_N}^{m-1}\right) \tau_N^{m-1} \ .
\end{align*}
Here for any integer $N$, $\left(\tau_N^n\right)_{n \geq 0}$ is a sequence of i.i.d random variables independent of $\left(X_{B_N}^n  \right)_{n \in \N}$ with $\tau_N^0 \sim \mathcal{E}(1)$. For every $N$, we recall that the invariant probability measure of the Markov chain $\left(X_{B_N}^n\right)_{n \in \N}$ is denoted by $\pi_{B_N}$ and defined in Eq.\;\eqref{invariante_measure_delta}.
We define the function $\Psi_{B_N}$ on $\T \times \{1,2\}$ by
\begin{equation}
\label{def_Psi}
\forall (k,i) \in \T \times \{1,2\}, \quad \Psi_{B_N}(k,i) = \textbf{v}_{B_N}\left(k\right) \lambda_{B_N}\left(k,i\right) \ .
\end{equation}

\medskip
As we recalled in Sec.\;\ref{section_previous_results_hydro}, the asymptotic behavior of $f_{B_N}$ is fully determined by the one of the process $Z_{u,B_N}(\cdot)$. The next propositions allow us to compute the tails of the function $\Psi_{B_N}$ and to determine the asymptotic behavior of the process $Z_{u,B_N}(\cdot)$.
\begin{proposition}
\label{proposition_x_B}
We define implicitly two functions $x_{B,\pm}$ on $\mathbb R^*$ by
\begin{align}
\left( 2\sqrt{x_{B,\pm}^2(r) + \frac{B^2}{4}} \pm B \right)  x_{B,\pm}(r) &= \frac{\pi }{\gamma r} \label{def_x_B_1} \ , \quad r\ne 0 \ .
\end{align}
Then
\begin{itemize}
    \item[i)] $x_{B,\pm}$ are odd functions. 
    \item[ii)] The functions $x_{B,\pm}$ and $x_{B,\pm}^\prime$ converge pointwise when $B$ goes to zero and for any $r \neq 0$ we have
    \begin{align*}
    &\lim\limits_{B \rightarrow 0}  x_{B,\pm}(r)=\text{sign}(r) \sqrt{\frac{ \pi}{2\gamma}} \vert r \vert^{-\frac{1}{2}} \quad \text{and} \quad \lim\limits_{B \rightarrow 0}  x_{B,\pm}^\prime(r)= -\frac{\sqrt{ \pi}}{2\sqrt{2\gamma}} \vert r \vert^{-\frac{3}{2}} \ .
    \end{align*}
    \item[iii)] The functions $x_{B,+}$ and $x_{B,+}^\prime$  converge pointwise  to zero when $B$ goes to infinity and for any $r \neq 0$, we have
    \begin{align*}
    &\lim\limits_{B \rightarrow \infty} B x_{B,+}(r) = \text{sign}(r)\frac{\pi}{2\gamma} \vert r \vert^{-1}  \quad \text{and} \quad \lim\limits_{B \rightarrow \infty} B x_{B,+}^\prime(r) =- \frac{\pi}{2\gamma } \vert r \vert^{-2} \ . 
    \end{align*}
    The functions $x_{B,-}$ and $x_{B,-}^\prime$  diverge pointwise to infinity when $B$ goes to infinity and for any $r \neq 0$, we have
    \begin{align*}
    &\lim\limits_{B \rightarrow \infty} B^{-\frac{1}{3}} x_{B,-}(r) = \text{sign}(r) C \vert r\vert^{-\frac{1}{3}} \quad \text{and} \quad  \lim\limits_{B \rightarrow \infty} B^{-\frac{1}{3}} x_{B,-}^\prime(r) =- \frac{C}{3} \vert r \vert^{-\frac{4}{3}} \ ,
    \end{align*}  
    with $C=\left(\frac{\pi}{2\gamma }\right)^{\frac{1}{3}}$.
\end{itemize}
\end{proposition}
\begin{proof}
 We refer the reader to Appendix \:\ref{appendix_preuve_proposition_x_B}.
\end{proof}
\begin{proposition}
\label{proposition_f}
Let $h_{B,\pm}$ and $g_{B \pm}$ be the four real-valued functions defined for any $r \neq 0$ as follows 
\begin{align}
h_{B,\pm}(r) &= \frac{1}{4 \pi}\int_{0}^{x_{B,\pm}(r)} \left(\frac{1}{2} \pm \frac{B}{4 \sqrt{x^2 + \frac{B^2}{4}}} \right) x^2  \; dx  \; \mathds{1}_{ r > 0} \label{def_f_B,+}\\
& + \frac{1}{4 \pi}\int_{x_{B,\pm}(r)}^{0} \left(\frac{1}{2} \pm \frac{B}{4 \sqrt{x^2 + \frac{B^2}{4}}} \right) x^2 \; dx \; \mathds{1}_{ r < 0}\nonumber\ .\\
g_{B,\pm}(r) &= -h_{B,\pm}^\prime(r) \mathds{1}_{r >0} + h_{B,\pm}^\prime(r)\mathds{1}_{r<0}\ .\label{def_densite_intermediaire}
\end{align}
Then 
\begin{itemize}
    \item[i)] $g_{B,\pm}$ are positive even functions.
    \item[ii)] The measure $\nu_{B}$ with density $g_{B,+}+g_{B,-}$ with respect to the Lebesgue measure is a L\'evy measure on $\mathbb{R}^*$. 
    \item[iii)] The functions $g_{B,\pm}$ converge pointwise to $g_{0}$ when $B$ goes to zero where for any $r \neq 0$
    \begin{equation}
        g_{0}(r) = \left(\frac{\pi}{2^{11}\gamma^{3}}\right)^{\frac{1}{2}} \vert r \vert^{-\frac{3}{2} -1}.
    \end{equation}
    Moreover, $g_0$ is the density of a L\'evy measure on $\mathbb{R}^*$.
    \item[iv)] The function $B^{\frac{1}{3}}g_{B,+}$ converges pointwise to zero and $B^{\frac{1}{3}}g_{B,-}$ converges pointwise to $g_{\infty}$ when $B$ goes to infinity where for any $r \neq 0$
    \begin{equation}
        g_{\infty}(r) = \left(\frac{\pi^{2}}{2^{11}27 \gamma^{5}}\right)^{\frac{1}{3}} \vert r \vert^{-\frac{5}{3} - 1}\ .
    \end{equation}
    Moreover, $g_\infty$ is the density of a L\'evy measure on $\mathbb{R}^*$.
\end{itemize}
\end{proposition}
\begin{proof}
 We refer the reader to Appendix \;\ref{appendix_preuve_proposition_f}.
\end{proof}
\begin{proposition}
\label{proposition_properties_pi}
Let $\pi_{B_N}$ be the invariant probability measure defined in Eq.\;\eqref{invariante_measure_delta} and $\Psi_{B_N}$ the function defined in Eq.\;\eqref{def_Psi} then we have
 \vspace{0.2cm}
 \begin{align*}
\forall r >0, \quad  \lim\limits_{ N \rightarrow \infty} N^{\beta_\delta} \;\pi_{B_N} \left(\left\{ (k,i)\; \big| \quad \Psi_{B_N}(k,i) > N r  \right\} \right) &= \kappa_\delta(r) \ , \\
 \forall r <0, \quad  \lim\limits_{ N \rightarrow \infty} N^{\beta_\delta} \;\pi_{B_N} \left(\left\{ (k,i)\; \big| \quad \Psi_{B_N}(k,i) < N r  \right\} \right) &= \kappa_\delta(r) \ ,
 \end{align*}
where for any $r \neq 0$
 \begin{equation}
 \label{tails_Psi}
 \left\{
    \begin{array}{ll}
    \beta_\delta = \frac{3}{2} \quad  \text{and} \quad \kappa_\delta(r)=\gamma^{-\tfrac{3}{2}}\kappa_1 \vert r\vert^{-\frac{3}{2}} \quad \text{if} \quad \delta> \frac{1}{2}\ ,\vspace{0.1cm}\\
    \beta_\delta = \frac{3}{2} \quad  \text{and} \quad \kappa_\delta(r)= \left( h_{B,-}(r) + h_{B,+}(r) \right) \quad \text{if} \quad \delta=\frac{1}{2}\ ,\vspace{0.1cm}\\
        \beta_\delta = \frac{5-\delta}{3} \quad  \text{and} \quad \kappa_\delta(r)=\gamma^{-\tfrac{5}{3}}\kappa_2 \vert B \vert^{-\frac{1}{3}}\vert r\vert^{-\frac{5}{3}} \quad \text{if} \quad \delta < \tfrac{1}{2}\ , 
    \end{array}
\right.
\end{equation}
with $\kappa_1$ and $\kappa_2$ two positive constants defined respectively in Eq.\;\eqref{def_kappa_1} and in Eq.\;\eqref{def_kappa_2}.
 \end{proposition}
 \begin{proof}
 We refer the reader to Appendix \;\ref{appendix_proof_tails}.
 \end{proof}
 \begin{thm}
 \label{main_theorem_hydro}Let $\beta_\delta$ and $\kappa_\delta$ defined in Eq.\;\eqref{tails_Psi} and $\tau_0 \sim \mathcal{E}\left(1\right)$. Let $u$ in $\R$, we define $u_N:=Nu$ and we assume that $\left(X_{B_N}^0\right)=(k,i)$ with $k \neq 0$ and $i$  in $\{1,2\}$. We denote by $Z_u^\delta(\cdot)$ the L\'evy process starting from point $u$ with L\'evy measure $\nu_\delta$ where
\begin{equation}
\label{def_Levy_measure}
d\nu_\delta(r) = \left\{
    \begin{array}{ll}
   \gamma^{-\tfrac{1}{2}}\kappa_0 \mathbb{E}\left[ \tau_0^{\beta_\delta} \right]\vert r \vert^{-\frac{3}{2}-1} \; dr \quad \text{if} \quad \delta > \frac{1}{2}\ ,\\\\
        2\gamma \mathbb{E}\left[  \tau_0^{-1} \left( g_{B, +}\left(\frac{2\pi r}{\tau_0 }\right) + g_{B, -}\left(\frac{2\pi r}{\tau_0 }\right) \right) \right] \; dr \quad \text{if} \quad \delta = \frac{1}{2}\ ,\\\\ 
        \gamma^{-\tfrac{2}{3}}\vert B \vert^{-1/3}\kappa_\infty \mathbb{E}\left[ \tau_0^{\beta_\delta} \right]\vert r \vert^{-\frac{5}{3}-1}\; dr  \quad \text{if} \quad \delta < \frac{1}{2}\ ,
   \end{array}
\right. \\
\end{equation}
with
\begin{equation}
    \label{def_constante_laplacien}
    \kappa_0 = \left(\frac{1}{ 2^{10} \pi^2}\right)^{\frac{1}{2}} \quad \text{and} \quad \kappa_\infty= \left(\frac{1}{2^{13}27 \pi^3} \right)^\frac{1}{3} \ .
\end{equation}
\begin{flushleft}
Then, under $\P_{(k,i)}$ the finite-dimensional distributions of the  process $N^{-1}Z_{u_N,B_N}\left(N^{\beta_\delta}\cdot\right)$ converge weakly  to the finite-dimensional distributions of $Z_u^\delta (\cdot)$.
\end{flushleft}
 \end{thm}
\begin{proof}
We refer the reader to Sec.\;\ref{section_proof_hydro_transition}.
\end{proof}
\begin{rmk}
Observe that the process $Z_u^{1/2}$ defined in Eq.\;\eqref{def_Levy_measure}, which is not a Lévy stable process, belongs nevertheless to the domain of attraction of two Lévy stable laws, one of exponent $3/2$ and the other one with exponent $5/3$. Indeed, using the definition of $g_{B,\pm}$ given in Eq.\;\eqref{def_densite_intermediaire}, one can prove that there exists two positive constants $C$ and $C_B$ such that
\begin{align*}
\nu_{1/2}(r) \underset{0}{\sim} C \vert r \vert^{-\tfrac{3}{2}} \quad \text{and} \quad \nu_{1/2}(r) \underset{\infty}{\sim} C_B \vert r \vert^{-\tfrac{5}{3}} \ , 
\end{align*}
where $C_B$ converges to zero when $B$ goes to infinity and $C$ is independent of $B$. 
\end{rmk}
Observe that $\mathbb{E}\left[ \tau_0^{\beta_\delta} \right] = \Gamma( 1+\beta_\delta)$ where $\Gamma$ is the Gamma function.
Theorem \ref{main_theorem_hydro} allows us to obtain the hydrodynamic limit of the solution $f_{B_N}$ of the Boltzmann equation \eqref{equation_u_N}.
Let $Z_u^\delta(\cdot)$ be the L\'evy process starting from $u$ in $\R$ with L\'evy exponent $\Phi_\delta$ defined by
\begin{equation}
\label{Levy_exponent_delta}
\forall \theta \in \R, \quad \Phi_{\delta} (\theta) = \int_{\mathbb{R}^*}\left(\exp\left( \mathbf{i} \theta r \right)-1+\mathbf{i}\theta r \mathds{1}_{\{ \vert r \vert < 1 \}} \right) \; d\nu_\delta(r) \ ,
\end{equation}
where $\nu_\delta$ is the L\'evy measure defined in Eq.\;\eqref{def_Levy_measure}. Observe that for $\delta> \frac{1}{2}$ $\left(\text{resp. } \delta < \frac{1}{2}\right)$, $Z_u^\delta(\cdot)$ is the L\'evy process obtained in \cite{JKO09} (resp. \cite{SSS19}). For $\delta=\frac{1}{2}$ we have an interpolation process.
\begin{thm}
\label{main_thm_limit_u} 
Let $f_{B_N}(\cdot,\cdot,\cdot):=(f_{B_N}(\cdot,\cdot,\cdot,1),f_{B_N}(\cdot,\cdot,\cdot,2))$ be the solution of Eq.\;\eqref{equation_u_N} with initial condition  $f^{0,N}(\cdot,\cdot)$ $:=(f^{0,N}(\cdot,\cdot,1),f^{0,N}(\cdot,\cdot,2))$ defined as follows
\begin{align}
    \forall (u,k) \in \mathbb{R}\times \mathbb{T}, \quad f^{0,N}(u,k) = \left( f^0\left( \frac{u}{N},k,1 \right), f^0\left( \frac{u}{N},k,2 \right) \right) \ .
\end{align}
Here $(f^0(\cdot,\cdot,1),f^0(\cdot,\cdot,2))$ $\in$ $\mathcal{C}^\infty_{\mathbf{c}} \left( \mathbb{R} \times \mathbb{T}  \right)^2$.  We define on $\R$ the function $\overline{f}^0$ as follows
\begin{equation}
\label{def_u0_barre}
 \forall u \in \R, \quad \overline{f}^0(u) = \sum_{i=1}^2 \int_{\T } f^0(u,k,i) \; dk \ .
\end{equation}
We define the operator $\mathfrak{L}_\delta$ on the space $\Cc_{\mathbf{c}}^\infty\left(\R\right)$ in the following way
\begin{equation}
\label{def_Lev_gen_limit}
\forall \phi \in \Cc_{\mathbf{c}}^\infty(\R), \quad \forall p \in \R, \quad \mathfrak{L}_{\delta}\left[\phi\right](p) =  \int_{\mathbb{R}} \Fc  \left[\phi\right](\xi) \Phi_{\delta}(\xi) \exp\left( 2 \mathbf{i}\pi p \xi\right) \; d\xi \ ,
\end{equation}
where $\Phi_\delta$ is defined in Eq.\;\eqref{Levy_exponent_delta}. Let $\beta_\delta$ defined by Eq.\;\eqref{tails_Psi}, then
\[ \forall t \in \; ]0,T], \quad \forall u \in \R, \quad \lim\limits_{N \rightarrow \infty} \  \sum_{i=1}^2\int_\T \left\vert f_{B_N}\left(N^{\beta_\delta}t, Nu,k,i   \right) - \frac{1}{2} \rho_\delta(t,u) \right\vert \; dk = 0 \ , \]
where $\rho_\delta$ is the solution of 
\begin{equation}
\label{limit_hyrdo_B=0}
\forall t \in \; ]0,T], \quad \forall u \in \R, \quad \left\{
    \begin{array}{ll}
       \partial_t \rho_\delta(t,u) = \mathfrak{L}_{\delta} \left[\rho_\delta\right](t,u) \ ,\\
        \rho_\delta(0,u) = \overline{f}^0(u)\  .
    \end{array}
\right.
\end{equation}
\end{thm} 
\begin{proof}
We refer the reader to Sec.\;\ref{section_proof_main_theorem_u}.
\end{proof}
\begin{rmk} Observe that when $\delta > \frac{1}{2}$, $\mathfrak{L}_\delta = -\Gamma\left( 1+\frac{3}{2}\right)\gamma^{-\frac{1}{2}}D_0(-\Delta)^{\frac{3}{4}}$ and when $\delta < \frac{1}{2}$ we have that  $\mathfrak{L}_\delta = -\vert B \vert^{-\frac{1}{3}}\Gamma\left( 1+\frac{5}{3}\right)\gamma^{-\frac{2}{3}}D_\infty(-\Delta)^{\frac{5}{6}}$ where
\begin{equation}
    \label{def_vrai_constante}
    D_0= 2\kappa_0\int_0^{+\infty} \frac{1-\cos(r)}{r^{\frac{5}{2}}} \; dr \quad \text{and} \quad  D_\infty= 2\kappa_\infty\int_0^{+\infty} \frac{1-\cos(r)}{r^{\frac{8}{3}}} \; dr \ .
\end{equation}Hence, we recover the different cases studied in \cite{JKO09} and \cite{SSS19} respectively.
\end{rmk}

\medskip
When $\delta=\frac{1}{2}$, $\mathfrak{L}_{\frac{1}{2}}$ depends on $B$ hence we denote by $\mathfrak{L}_{B}$ the operator $\mathfrak{L}_{\frac{1}{2}}$, by $\Phi_{B}$ the L\'evy exponent $\Phi_{\frac{1}{2}}$ and by $\tilde{\rho}_{B}$ the function $\rho_{\frac{1}{2}}$. The next theorem shows that by sending $B$ to zero (resp. infinity), $\tilde{\rho}_{B}$ converges to the fractional diffusion equations of exponent $3/4$ $\left( \text{resp.} \; 5/6 \right)$.
\begin{thm}
\label{main_thm_limit_B_u_barre}
Let $\overline{f}^0$ be the function defined in Eq.\;\eqref{def_u0_barre} and $D_0$ and $D_\infty$ the constants defined in Eq.\;\eqref{def_vrai_constante}. Then 
\begin{itemize}
\item[i)] $\lim\limits_{B \rightarrow 0} \tilde{\rho}_{B} = \rho_0$ in $\mathbb{L}^1\left( [0,T], \mathbb{L}^2(\mathbb{R})  \right)$ where $\rho_0$ is the solution of 
\begin{equation}
\label{rho_zero}
\forall t \in \; ]0,T], \quad \forall u \in \R, \quad \left\{
    \begin{array}{ll}
       \partial_t \rho_0(t,u) = -\Gamma\left( 1+\frac{3}{2}\right)\gamma^{-\tfrac{1}{2}} D_0(-\Delta)^{\tfrac{3}{4}} \left[\rho_0\right](t,u) \ ,\\
        \rho_0(0,u) = \overline{f}^0(u) \ .
    \end{array}
\right.
\end{equation}
\item[ii)] $\lim\limits_{B \rightarrow \infty} \tilde{\rho}_B \left(B^{\frac{1}{3}} \cdot,\cdot\right) = \rho_\infty$ in $\mathbb{L}^1\left( [0,T], \mathbb{L}^2(\mathbb{R}) \right)$  where $\rho_\infty$ is the solution of 
\begin{equation}
\label{rho_infini}
\forall t \in \; ]0,T], \quad \forall u \in \R, \quad \left\{
    \begin{array}{ll}
       \partial_t \rho_\infty(t,u) = -\Gamma\left( 1+\frac{5}{3}\right)\gamma^{-\tfrac{2}{3}}D_\infty(-\Delta)^{\tfrac{5}{6}} \left[\rho_\infty \right] (t,u)\ ,\\
        \rho_\infty(0,u) = \overline{f}^0(u) \ .
    \end{array}
\right.
\end{equation}
\end{itemize}
\end{thm}
\begin{proof}
 We refer the reader to Sec.\;\ref{section_main_result_B}.
\end{proof}
\section{Proof of Theorem \ref{main_theorem_kinetic}, Theorem \ref{main_theorem_hydro}, Theorem \ref{main_thm_limit_u} and Theorem \ref{main_thm_limit_B_u_barre}}
 \label{section_proof}
 \subsection{Sketch of the proof of Theorem \ref{main_theorem_kinetic}}
 \label{section_proof_kinetic_transition}
This section, is devoted to the proof of Theorem \ref{main_theorem_kinetic}. The strategy of our proof is similar to the one developed by the authors of \cite{SSS19}, so we will only give the ideas of the different steps and we refer the reader to the corresponding sections  of \cite{SSS19} for detailed proofs.\\\\
To prove Theorem \ref{main_theorem_kinetic}, we first prove that there exists $\mathcal{W}$ in $\Cc\left([0,T], \left(\Sc\times\Sc\right)^\prime \right)$ and a sequence $(\varepsilon_n)_{n \in \N}$ such that for any $J$ in $\mathcal{S}\times \mathcal{S}$ and $t$ in $[0,T]$
\begin{equation} 
\label{limit_Wigner}
\lim\limits_{n \rightarrow \infty} \left\vert\left\langle \mathcal{W}^{ \varepsilon_n} (t) , J \right\rangle - \left\langle \mathcal{W}(t) , J \right\rangle \right\vert= 0 \ . 
\end{equation}
Let $J:=(J_1,J_2)$ in $\Sc \times \Sc$, to prove Eq.\;\eqref{limit_Wigner} we start to show that for any $i$ in $\{1,2 \}$ and $t$ in $[0,T]$ we have
\begin{align}
\partial_t \left\langle \mathcal{W}_i^\varepsilon(t), J_i \right\rangle & =  \frac{1}{2\pi}\left\langle \mathcal{W}_i^\varepsilon(t), \textbf{v}_{0} \partial_u J_i \right\rangle  +   \gamma \left\langle \mathcal{W}_i^\varepsilon(t), [C_0J]_i \right\rangle  \nonumber  \\
& + \gamma \left( \langle\mathcal{A}_i^\varepsilon(t),[C_0J]_i \rangle + \langle\left({\mathcal{A}}_i^\varepsilon\right)^*(t),[C_0J]_i \rangle \right)  +  \mathcal{O}_{t,J}(\varepsilon) \ ,\label{evolution_eq_ib}
\end{align}
where $\mathcal{A}_i^\varepsilon$ and $\left(\mathcal{A}_i^\varepsilon\right)^*$ are defined in Eq.\;\eqref{def_anti_wigner} and Eq.\;\eqref{def_anti_wigner_conjugate} and $\vert \mathcal{O}_{t,J}(t,\varepsilon) \vert$  $\leq K(J) \times \varepsilon$ with $K(J)$ a constant independent of $\varepsilon$ and $t$. The derivation of this equation is presented in Appendix \;\ref{appendix_behavior_transport_term} and Appendix \;\ref{appendix_behavior_colli_term}. To end the proof we need the following lemma.  
\begin{lemma}
\label{lemma_anti_wigner}
Let $t$ in $[0,T]$ and $J:=(J_1,J_2)$ in $\Sc \times \Sc$ then
\begin{itemize}
\item[i)] $\Vert \langle\mathcal{A}^\varepsilon,J \rangle   \Vert_{\infty} \lesssim K_0 \Vert J \Vert$ and  $\Vert \langle\tilde{\mathcal{A}}^\varepsilon,J \rangle   \Vert_{\infty} \lesssim K_0 \Vert J \Vert.$   
\vspace{0.2cm}
\item[ii)] For any $i$ in $\{1,2\}$ \[ \lim\limits_{ \varepsilon \rightarrow 0} \left\vert \int_0^T \left\langle\mathcal{A}_i^\varepsilon(t),[C_0J]_i \right\rangle  \; dt  \right\vert = 0 \quad \text{and} \quad  \lim\limits_{ \varepsilon \rightarrow 0} \left\vert \int_0^T \left\langle\left({\mathcal{A}}_i^\varepsilon\right)^*(t),[C_0J]_i \right\rangle \; dt  \right\vert = 0 \ .\]
\end{itemize}
\end{lemma}
\begin{proof}
 We refer the reader to \cite[Sec.\;3.4]{BOS09}.
\end{proof}
By Eq.\;\eqref{evolution_eq_ib}, item ii) of Lemma \ref{Lem_Wigner} and item i) of Lemma \ref{lemma_anti_wigner} we deduce that for all $J$ in $\Sc\times \Sc$ the family of functions $\left(\left\langle \Wc^\varepsilon, J \right\rangle\right)_{\varepsilon > 0}$ is equicontinuous and bounded in $\left(\Cc\left([0,T], \C\right), \Vert \cdot \Vert_{\infty} \right).$ Hence, for any $J$ in $\mathcal{S}\times \mathcal{S}$, there exists $\mathcal{W}_J(\cdot) \in \left(\mathcal{C}\left([0,T], \mathbb{C}\right), \Vert \cdot \Vert_{\infty} \right)$ such that $\left\langle \mathcal{W}^\varepsilon (\cdot) , J \right\rangle$ converges (up to a subsequence) to $\mathcal{W}_J(\cdot)$ in $\left(\mathcal{C}\left([0,T], \mathbb{C}\right), \Vert \cdot \Vert_{\infty} \right)$.

\medskip 

At this point, the subsequence depends on $J$ but since $\left( \mathcal{S} \times \mathcal{S}, \Vert \cdot \Vert \right)$ is a separable space, we can use a diagonal argument to erase this dependency. Hence, there exists a countable subset  $\mathcal{D} = (J^r)_{r \in \mathbb{N}}$ $\subset$ $\mathcal{S} \times \mathcal{S}$ dense in $\left(\mathcal{S} \times \mathcal{S}, \Vert \cdot \Vert \right)$, a subsequence $\left(\varepsilon_n \right)_{n \in \N}$ going to zero as $n$ goes to infinity and a sequence $\left( \mathcal{W}_J(\cdot) \right)_{J \in \mathcal{D}}$ $\in$ $\left( \mathcal{C}\left([0,T], \mathbb{C} \right), \Vert \cdot \Vert_{\infty} \right)$  such that  
\begin{align*}   \forall J \in \mathcal{D}, \quad \lim\limits_{ n \rightarrow \infty} \langle \mathcal{W}^{\varepsilon_n}(\cdot),J\rangle =  \mathcal{W}_J (\cdot) \ . 
\end{align*}
Using item ii) of Lemma \ref{Lem_Wigner}, we deduce that
\begin{equation*}
\forall t \in [0,T], \quad \forall J \in \mathcal{D}, \quad \left\vert \Wc_J (t) \right\vert \leq K_0 \Vert J \Vert \ .
\end{equation*}
Hence the application 
\begin{align*}
    J \in \mathcal{D} \subset \mathcal{S} \times \mathcal{S} \mapsto \mathcal{W}_J(\cdot) \in \left( \mathcal{C}\left([0,T], \mathbb{C} \right), \Vert \cdot \Vert  _{\infty}\right)
\end{align*}
is Lipschitz and linear. It can be uniquely extended into a linear Lipschitz function on $\mathcal{S} \times \mathcal{S}$. Hence, for any $t$ in $[0,T]$,
\begin{align*}
    J \in \mathcal{S} \times \mathcal{S} \mapsto \mathcal{W}_J(t)
\end{align*}
defines an element of $\left( \mathcal{S} \times \mathcal{S} \right)^\prime$ denoted by $\left\langle \mathcal{W}(t), J\right\rangle$. It remains to prove that $\mathcal{W}(\cdot)$ belongs to $\mathcal{C}\left([0,T] , \left( \mathcal{S} \times \mathcal{S} \right)^\prime\right)$.

\medskip 

Let $s$ in $[0,T]$ and $\left(s_p \right)_{p \geq 0}$ in $\left[0,T \right]^\mathbb{N}$ which converges to $s$ when $p$ goes to infinity. Let $J$ in $\mathcal{S}\times \mathcal{S}$, $\eta > 0$ and $J^r$ in $\mathcal{D}$ such that 
\begin{align*}
    \Vert J-J^r \Vert < \eta \ .
\end{align*}
We have that 
\begin{align*}
    \left\vert \left\langle \mathcal{W}\left(s_p\right), J \right\rangle - \left\langle \mathcal{W}\left(s\right), J \right\rangle \right\vert & \leq \left\vert \left\langle \mathcal{W}\left(s_p\right)- \mathcal{W}(s), J^r \right\rangle \right\vert  +  \left\vert \left\langle \mathcal{W}\left(s_p\right), J-J^r \right\rangle \right\vert \\
    &+ \left\vert \left\langle \mathcal{W}\left(s\right), J-J^r \right\rangle \right\vert \\
    & \leq \left\vert \left\langle \mathcal{W}\left(s_p\right)- \mathcal{W}(s), J^r \right\rangle \right\vert + 2 K_0 \Vert J- J^r \Vert \\
    & \leq \left\vert \left\langle \mathcal{W}\left(s_p\right)- \mathcal{W}(s), J^r \right\rangle \right\vert + 2 K_0 \eta \ .
\end{align*}
Recall that 
\begin{align}
    t \in [0,T] \mapsto \left\langle \mathcal{W}(t),J^r \right\rangle = \mathcal{W}_{J^r}(t) \in \mathbb{C}
\end{align}
is continuous because $J^r$ is in $\mathcal{D}$. Hence, it exists $p_0$ in $\mathbb{N}^*$ such that 
\begin{align*}
    \forall p > p_0, \quad \left\vert \left\langle \mathcal{W}\left(s_p\right)- \mathcal{W}(s), J^r \right\rangle \right\vert \leq \eta \ .
\end{align*}
Consequently, we have
\begin{align*}
    \left\vert \left\langle \mathcal{W}\left(s_p\right), J \right\rangle - \left\langle \mathcal{W}\left(s\right), J \right\rangle \right\vert \leq \left(2K_0 +1\right) \eta \ .
\end{align*}
This proves that $\mathcal{W}(\cdot)$ is in $\mathcal{C}\left([0,T], \left( \mathcal{S}\times \mathcal{S}\right)^\prime \right)$ and ends the proof.

\medskip 

Hence, we deduce that there exists $\Wc$ in $\Cc \left( [0,T], \left(\Sc\times \Sc\right)^\prime \right)$ such that $(\mathcal{W}^\varepsilon)_{\varepsilon >0}$ converges pointwise (up to a subsequence) to $\Wc$. Using item ii) of Lemma \ref{lemma_anti_wigner} and sending $\varepsilon$ to zero in Eq.\;\eqref{evolution_eq_ib} we get that for any $t$ in $[0,T]$ and $J$ in $\Sc \times \Sc$
\begin{align}
\label{Bolt_D} \left\langle \Wc(t), J \right\rangle - \left\langle \Wc(0), J \right\rangle & = \int_0^t \frac{1}{2 \pi} \left\langle \Wc(s), \textbf{v}_0 \partial_u J \right\rangle \; ds +  \gamma\int_0^t \left\langle \Wc(s), C_0 J \right\rangle \; ds \ .
\end{align}
This proves that $\Wc$ is a weak solution of the Boltzmann equation\;\eqref{BE}.

\medskip

Let $t$ in $[0,T]$. It remains to extend $\Wc(t)$ into a couple $\mu(t):=(\mu_1(t),\mu_2(t))$ of bounded Borel measure on $\R \times \T$. This is proved in \cite[Lemma D.1]{SSS19}. The idea is to prove that the Wigner distribution is a positive linear form on $\Cc_{\mathbf{c}}( (\mathbb{R}\times \T)^2)$ and then by using Riesz's representation theorem we can extend $\Wc(t)$ into a couple of Borel measure on $\R \times \T$ denoted by $\mu(t):=(\mu_1(t), \mu_2(t))$.

\medskip

To end the proof, it remains to prove that for any $t$ in $[0,T]$ and $i$ in $\{1,2\}$, $\mu_i(t)$ is a bounded measure on $\R \times \T$.
Let $J$ in $\mathcal{S}$ defined by $J(u,k):=J^{\lambda,r}(u)$ where
\[ \forall u \in \R, \quad J^{\lambda,r}(u) := \exp\left(- \frac{\lambda }{r^2- u^2} \right) \mathds{1}_{[-r,r]}(u),\quad \lambda >0, \quad r >0 \ . \]
Let $\lambda$ be fixed, by using Eq.\;\eqref{K0} and item ii) of Lemma \ref{Lem_Wigner} we get that
\begin{align*} \forall t \in [0,T], \quad \forall i \in \{1,2\}, \quad \lim\limits_{\varepsilon \rightarrow 0} \left\langle \mathcal{W}_i^\varepsilon(t), J^{\lambda,r} \right\rangle = \int_{\mathbb{R} \times \mathbb{T}} J^{\lambda, r} (u) \; \mu_i(t, du,  dk) \leq K_0 \ .
\end{align*}
By sending first $\lambda$ to zero and then $r$ to infinity we get 
\begin{align*} 
\mu_i\left(t,\mathbb{R} , \mathbb{T}\right) \leq K_0 \ . 
\end{align*}
Hence, $\mu$ is a \textbf{Borel measure valued} weak solution of the Boltzmann equation (\ref{BE}) with initial condition $\mu^0:=(\mu_1^0,\mu_2^0)$. By Lemma \ref{lemma_unicite}, there is a unique \textbf{Borel measure valued} weak solution of Eq.\;\eqref{BE} hence we deduce that $(\mathcal{W}^\varepsilon)_{\varepsilon >0}$ converges pointwise to $\mu:=(\mu_1,\mu_2)$ in $\Fc \left( [0,T] , (\Sc\times\Sc)^\prime \right)$.
  \subsection{Sketch of the proof of Theorem \ref{main_theorem_hydro}}
  \label{section_proof_hydro_transition} 
Let $u$ in $\R$, to lighten the notations we define for each time $t$, $Z_u^N(t) := N^{-1}Z_{u_N,B_N}\left(N^{\beta_\delta} t\right)$ where we recall that $Z_{u_N,B_N}$ is defined in Eq.\;\eqref{stich}, $\beta_\delta$ is defined in Eq.\;\eqref{tails_Psi} and $u_N=Nu$. The proof of Theorem \ref{main_theorem_hydro} is divided into two steps.

\medskip

In the first instance we show that the finite-dimensional distributions of $Z_u^N(\cdot)$ converge weakly under $\mathbb{P}_{ \pi_{B_N}}$ to the finite-dimensional distributions of a L\'evy process $Z_u^\delta(\cdot)$ where for any positive time $t$ and $\theta$ in $\R$
\begin{equation}
\mathbb{E}\left[e^{\mathbf{i}\theta Z_u^\delta(t)}  \right]=\exp\left(t \Phi_{\delta}(\theta)+ \mathbf{i}\theta u\right) \ .
\end{equation}
Here for any $\theta$ in $\mathbb{R}$ 
\begin{equation}
\label{Z_x}
\Phi_{\delta}(\theta) = \int_{\R}\left(e^{\mathbf{i} r \theta} -1+\mathbf{i}\theta r \mathds{1}_{\{ \vert r \vert < 1 \}}   \right) \; d\nu_\delta(r) \ ,
\end{equation}
where $\nu_\delta$ is defined in Eq.\;\eqref{tails_Psi}. Then we prove that we can extend this result under $\mathbb{P}_{(k,i)}$ when $k$ is in $\T^*$ and $i$ is in $\{1,2\}$. 

\medskip

We start to show the first step, we follow the strategy of \cite[Sec. 6]{JKO09} and only prove that one-dimensional distribution of $\left(Z^N_u(\cdot)\right)_{N \in \N}$ converges weakly to one-dimensional distribution of $Z_u^\delta(\cdot)$. Let $t >0$, we define the random integer $j_N(t)$ as follows 
 \begin{equation}
\sum_{n=0}^{j_N(t)} \lambda_{B_N}\left(X_{B_N}^n\right) \tau_N^n \leq t < \sum_{n=0}^{j_N(t)+1} \lambda_{B_N}\left(X_{B_N}^n\right) \tau_N^n \ ,
 \end{equation}
where we recall that for any integer $N$, $\left(\tau_N^n\right)_{n \in \N}$ is a sequence of i.i.d random variables independent of the Markov chain $\left( X_{B_N}^n \right)_{n \in \N}$ and such that $\tau_N^0 \sim \mathcal{E}(1)$ and $\lambda_{B_N}$ is defined in Eq.\;\eqref{def_lambda}. Let $t \geq 0$, we define 
\begin{align*}
Y_u^N(t) &= u- N^{-1} \sum_{n=0}^{\lfloor N^{\beta_\delta} t \rfloor} \frac{\Psi_{B_N} \left(X_{B_N}^n\right)}{2\pi}\tau_N^n = u- N^{-1} \sum_{n=0}^{\lfloor N^{\beta_\delta} t \rfloor} \frac{\textbf{v}_{B_N}\left(K_{B_N}^n\right)}{2\pi}\lambda_{B_N}\left(X_{B_N}^n\right)\tau_N^n \ .\\
S_N(t) &= \frac{j_N(N^{\beta_\delta} t)}{N^{\beta_\delta} }=\inf\left\{ u >0 \; \Bigg| \; \; \sum_{n=0}^{\lfloor N^{\beta_\delta} u \rfloor} \lambda_{B_N}\left(X_{B_N}^n\right)\tau_N^n > t \right\} \ . \\
S(t) & = t\overline{\lambda}^{-1}=2\gamma t \ .
\end{align*}
Let $t$ be a positive time, we can decompose $Z_u^N(t)$ into two parts, the first one is given by all the jumps made by the particle until time $N^{\beta_\delta} t$ and the second part is given by the displacement made by the particle since the last jump. Hence, using the definition of $\Psi_{B_N}$ given in Eq.\;\eqref{def_Psi} we can write
\begin{align*}
&Z_u^N(t) =u- \sum_{n=0}^{j_N\left(N^{\beta_\delta} t\right) -1} \frac{\Psi_{B_N} \left(X_{B_N}^n   \right)\tau_N^n}{2\pi N} - \left(\frac{N^{\beta_\delta}t-t_{j_N(N^{\beta_\delta} t)}}{2\pi N}\right)\textbf{v}_{B_N}\left(K^{j_N\left(N^{\beta_\delta}t\right)}_{B_N}\right) \ .\\
&Y_u^N\left(S_N(t)\right) = u- \sum_{n=0}^{j_N\left(N^{\beta_\delta} t\right)} \frac{\Psi_{B_N} \left(X_{B_N}^n \right)\tau_N^n}{2\pi N} \ .
\end{align*}
We first prove that $Y_u^N(S_N(t))$ converges weakly to $Z_u^\delta(t)$, then we prove that $Z_u^N(t)$ and $Y_u^N(S_N(t))$ are close.
\begin{proposition}
\label{proposition_B_N} Let $\nu_\delta$ defined in Eq.\;\eqref{def_Levy_measure}. Then, under $\mathbb{P}_{\pi_{B_N}}$ the finite-dimensional distributions of $\left(Y_u^N(\cdot) \right)_{N \in \N}$ converge weakly to the finite dimensional distributions of a L\'evy process $Y_u^\delta$   where  for any positive time $t$ and $\theta$ in $\R$
\begin{equation}
\mathbb{E}\left[e^{\mathbf{i}\theta Y_u(t)}  \right]=\exp\left((2\gamma)^{-1} t \Phi_\delta(\theta)+ \mathbf{i}\theta u\right)\ ,
\end{equation}
where the expression of $\Phi_\delta$ is recalled in Eq.\;\eqref{Z_x}.
\end{proposition}
\begin{proof}[Idea of the proof]
Proposition \ref{proposition_properties_pi} allows us to understand the behavior of the tails of the random variables $\left(\Psi_{B_N} \left(X_{B_N}^n\right) \right)_{n \in \mathbb{N}}$ when $N$ goes to infinity. Using this proposition and \cite[Lemma 4.2]{JKO09} which describes the convergence of an array of random variables we can conclude the proof. We refer the reader to Appendix \;\ref{appendix_proof_Levy} for a complete proof of Proposition \ref{proposition_B_N}.
\end{proof}
\begin{flushleft}
To complete the proof of Theorem \ref{main_theorem_hydro} we need the following two lemmas.
\end{flushleft}
\begin{lemma}
\label{lemma_S_N} Let $t_0 >0$ and $T \geq t_0$ then 
\begin{equation}
\forall \varepsilon >0, \quad \lim\limits_{N \rightarrow \infty} \mathbb{P}_{\pi_{B_N}}\left[ \sup_{ t \in [t_0,T]} \vert S_N(t) -S(t) \vert > \varepsilon   \right] =0\ .
\end{equation}
\end{lemma}
\begin{proof}
We refer the reader to Appendix \ref{section_proof_lemma_weak}.
\end{proof}
\begin{lemma}
\label{lemma_BN}
For any time $t$ and $\varepsilon >0$ be fixed we have
\begin{equation}
\label{result_lemma_6.2_stefano}
\lim\limits_{N \rightarrow \infty} \mathbb{P}_{\pi_{B_N}}\left[  \left\vert Z_u^N(t) - Y_u^N\left(S_N(t)\right)\right\vert > \varepsilon \right]=0 \ .
\end{equation}
\end{lemma}
\begin{proof}[Proof of Lemma \ref{lemma_BN}] We recall only the main ideas and we refer the reader to \cite[Lemma 6.2]{JKO09} for a complete proof. Let $\sigma >0$ be fixed, then by Lemma \ref{lemma_S_N} we have
\begin{align*}
& \mathbb{P}_{\pi_{B_N}}\left[  \left\vert Z_u^N(t) - Y_u^N\left(S_N(t)\right)\right\vert > \varepsilon \right]  \\
&\leq  \P_{\pi_{B_N}} \left[ \left\vert S_N(t) - t \right\vert  > \sigma \right] \\
& +  \P_{\pi_{B_N}} \left[ \left\vert S_N(t) - t \right\vert  \leq \sigma  \cap \left\vert Y_N(t) - Y_u^N\left( S_N(t)\right) \right\vert > \varepsilon\right]\\
 & \leq \P_{\pi_{B_N}} \left[ \left\vert S_N(t) - t \right\vert  \leq \sigma \ , \  \left\vert \tfrac{\Psi_{B_N}\left( X^{j_N\left(N^{\beta_\delta}t\right)}_{B_N} \right)}{N} \right\vert > \varepsilon \right]  + \sigma \ .
\end{align*}
Using the stationarity of the chain $\left( \Psi_{B_N}\left(X_{B_N}^n\right)\tau_N^n  \right)_{n \in \N}$ we have
\begin{align*}
& \P_{\pi_{B_N}} \left[ \left\vert S_N(t) - t \right\vert  \leq \sigma \ ,  \ \left\vert N^{-1}\Psi_{B_N}\left( X^{j_N(N^{\beta_\delta}t)}_{B_N} \right) \right\vert > \varepsilon \right] \\
& \leq \P_{\pi_{B_N}} \left[ \sup\left\{ \Psi_{B_N}\left(X_{B_N}^n\right)\tau_N^n  , \; n \in \left[ (S(t)-\sigma)N^{\beta_\delta} , (S(t)+ \sigma)N^{\beta_\delta}    \right]\right\} > N\varepsilon \right]\\
&= \P_{\pi_{B_N}} \left[ \sup\left\{ \Psi_{B_N}\left(X_{B_N}^n\right)\tau_N^n  , \; n \in \left[ 0,2\sigma N^{\beta_\delta}   \right]\right\} > N\varepsilon \right]\\
&= N^{\beta_\delta} \int_0^{+\infty} e^{-\tau} \pi_{B_N} \left(\left\{ u \; \big| \; \Psi_{B_N} (u) > N  \tau^{-1} \varepsilon \right\}\right) \; d\tau \\
& \lesssim \frac{\sigma}{\varepsilon^{\beta_\delta}} \ .
\end{align*}
This ends the proof.
\end{proof}
We recall that if a sequence $(x_n)_{n \in \N}$ in $\mathcal{D}\left([t_0,T], \mathbb{R}   \right)$ converges uniformly to $x$ in $\mathcal{D}\left([t_0,T], \mathbb{R}   \right)$ then the convergence follows in the sense of Skorohod topology. Hence, from Lemma \ref{lemma_S_N}, we conclude that $S_N(\cdot)$ converges in probability to $S(\cdot)$ in $ \mathcal{D}\left([t_0,T], \mathbb{R}   \right)$.

Using Proposition \ref{proposition_B_N} and Lemma \ref{lemma_S_N} we can conclude that $\left(Y_u^N(\cdot), S_N(\cdot)\right)$ converges in law to $(Y_u^\delta(\cdot), S(\cdot))$ in $\mathcal{D}([t_0,T], \mathbb{R})$. Hence, by Skorokhod's representation theorem there exists $\left(\tilde{Y}_u^N(\cdot), S_N(\cdot)\right)$ with values in $\mathcal{D}([0, \infty) \times [0, \infty) )$ such that the pairs $\left(Y_u^N(\cdot),S_N(\cdot)\right)$ and  $\left(\tilde{Y}_u^\delta(\cdot),\tilde{S}_N(\cdot)\right)$ have the same law and $\left(\tilde{Y}_u^N(\cdot), \tilde{S}_N(\cdot)\right)$ converges almost surely to $\left(Y_u^\delta(\cdot), S(\cdot)\right)$ in the Skorokhod topology.

From this we deduce that there exists a sequence of increasing homeomorphisms $(h_N)_{N \in \N}$ from $[t_0,T]$ to $[t_0,T]$ such that
\begin{align}
\lim\limits_{ N \rightarrow \infty} \sup_{ t \in [t_0,T]} \left\vert h_N(t) - t \right\vert &=0 \quad \text{and} \quad 
\lim\limits_{ N\rightarrow \infty} \sup_{t \in [t_0,T]} \left\vert \tilde{Y}_u^N\left(h_N(t)\right) - Y^\delta_u\left(S(t)\right) \right\vert =0 \ .
\end{align}
Then we have
\begin{align*}
 \left\vert \tilde{Y}_u^N\left( \tilde{S}_N(t) \right) - Y_u^\delta\left( S(t) \right) \right\vert  &\leq  \left\vert \tilde{Y}_u^N\left( \tilde{S}_N(t) \right) - Y_u^\delta \left(h_N^{-1}\left( \tilde{S}_N(t) \right) \right) \right\vert \\
 & + \left\vert Y_u^\delta \left(h_N^{-1}\left(\tilde{S}_N(t) \right)\right) - Y_u^\delta\left( S(t) \right) \right\vert \ ,
\end{align*}

\begin{flushleft}
which goes to zero when $N$ goes to infinity. Hence, $\left(\tilde{Y}_u^N\left( \tilde{S}_N(t) \right)\right)_{N \in \N}$ converges to $Y_u^\delta(S(t))$. Using that $\tilde{Y}_u^N\left( \tilde{S}_N(t) \right)$ and $Y_u^N\left( \tilde{S}_N(t) \right)$ have the same law we conclude that $\left(Y_u^N\left( \tilde{S}_N(t) \right)\right)_{N \in \N}$ converges in law to $Y_u^\delta(S(t))$. 
Using Eq.\;\eqref{result_lemma_6.2_stefano} we conclude that $\left(Z_u^N(t)\right)_{N \geq 0}$ converges weakly to $Y_u^\delta(S(t))$. Observe that $Y_u^\delta(S(t))$ and $Z_u^\delta(t)$ have the same law hence we proved that under $\P_{\pi_{B_N}}$ the finite-dimensional distribution of $\left(Z_u^N(\cdot)\right)_{N \geq 0}$ converge to the finite-dimensional distribution of $Z_u^\delta(\cdot)$.
\end{flushleft}

\medskip

To conclude the proof, it remains to extend this convergence under $\mathbb{P}_{(k,i)}$ when $k$ is in $\T^*$ and $i$ is in $\{1,2\}$. Since $k$ is not equal to zero by Eq.\;\eqref{def_Psi} we have
\[ \lim\limits_{N \rightarrow \infty} N^{-1} \Psi_{B_N} \left(X_{B_N}^0  \right) = \lim\limits_{N \rightarrow \infty} N^{-1} \Psi_{B_N} \left(k,i  \right)= 0 \; \; \; \mathbb{P}_{(k,i)}\; \; \; \text{a.s.} \] 
Hence, we deduce that $\left( N^{-1}\Psi_{B_N}\left(X_{B_N}^0\right) \tau_N^0 \right)_{N \in \N}$ goes to zero in $\mathbb{L}_{\P_{(k,i)}}^1$. From Markov inequality we deduce that
\begin{equation}
\label{Probabilite_X_et_tilde_X}
\forall \varepsilon >0, \quad \lim\limits_{N \rightarrow \infty} \P_{(k,i)} \left( N^{-1} \left\vert \Psi_{B_N}\left(X_{B_N}^0\right)  \right\vert \tau_0^N > \varepsilon  \right) =0 \ .
\end{equation}
We define the stochastic process $\tilde{Z}_u^N(\cdot)$ for any positive time $t$ by 
\begin{align*}
&{\tilde Z}_u^N(t) =u- \sum_{n=1}^{j_N(N^{\beta_\delta} t) -1} \frac{\Psi_{B_N} \left(X_{B_N}^n   \right)\tau_N^n}{2\pi N}   - \left(\frac{N^{\beta_\delta}t-t_{j_N\left(N^{\beta_\delta} t\right)}}{2\pi N}\right)\textbf{v}_{B_N}\left(K^{j_N\left(N^{\beta_\delta}t\right)}_{B_N}\right) \ .
\end{align*}
Whatever the distribution of $X_0^{B_N}$ is, by using Eq.\;\eqref{def_P}, we observe that $\left(X_{B_N}^n \right)_{n \geq 1}$ is a sequence of i.i.d random variables distributed according to $\pi_{B_N}$. Hence, since $\tilde{Z}_u^N(\cdot)$ depends only on $\left( X_{B_N}^n\right)_{n \geq 1}$ we deduce by Proposition \ref{proposition_B_N} that under $\P_{(k,i)}$ the finite-dimensional distributions of $\tilde{Z}_u^N(\cdot)$ converge weakly to the finite-dimensional distributions of the L\'evy process $Z_u^\delta(\cdot)$ defined in Eq.\;\eqref{Z_x}.\\
Using Eq.\;\eqref{Probabilite_X_et_tilde_X} we deduce that for any positive time $t$
\begin{equation}
\forall \varepsilon >0, \quad \lim\limits_{N \rightarrow \infty} \P_{(k,i)} \left( \left\vert Z_u^N(t)-\tilde{Z}_u^N(t) \right\vert > \varepsilon \right) =0 \ .
\end{equation}
This concludes the proof. 

\subsection{Sketch of the proof of Theorem \ref{main_thm_limit_u}}
\label{section_proof_main_theorem_u}
In this section, we prove Theorem \ref{main_thm_limit_u}, to prove it we need the following lemma.
\begin{lemma}
\label{lemma_convergence_P} Let $\rho$ be the counting measure on $\{1,2\}$. Observe that the measure $\sigma:=\frac{1}{2} dk d\rho$ is an invariant measure for the process $\left(K_{B_N}(\cdot),I_{B_N}(\cdot)  \right)$. Let $a\in (0,1)$ such that 
\begin{equation}
    \sup_N \int_{\mathbb{T} \times \{1,2\}} \lambda^a_{B_N}(k,i) \; d\sigma(k,i) < \infty \ ,
\end{equation}
where $\lambda_{B_N}$ is defined in Eq.\;\eqref{def_lambda} and define $d\sigma_N^a:= \lambda_{B_N}^{a} d\sigma $. Then for any \textbf{centered} $f$ in $ \mathbb{L}^1\left( \sigma \right)$ and positive time $t$
\[  \left\Vert P_N^{t} f \right\Vert_{\mathbb{L}^1\left( \sigma \right)} \leq \frac{C}{(1+t)^a} \Vert f \Vert_{\mathbb{L}^1(\sigma^a_N)}\ ,\]
where $C$ is a positive constant which dos not depends on $N$ and $P_N$ is the semi-group associated to the process $\left(K_{B_N}(\cdot),I_{B_N}(\cdot)\right)$.
\end{lemma}

\begin{proof}
This result is proved in a general setting in \cite[Theorem 1.1]{K14}. In our context, following the proof of \cite{K14} we can check that $C$ does not depend on $N$.
\end{proof}
We follow the strategy developed in \cite{JKO09} and \cite{SSS19}. The fundamental tool is Theorem \ref{main_theorem_hydro}. Let $\left(K_{B_N}(0),I_{B_N}(0)\right)=(k,i)$ with $k$ in $\T^*$ and $i$ in $\{1,2\}$.
By definition for any positive time $t$ and $u$ in $\R$
\begin{align*} f_{B_N}\left(N^{\beta_\delta} t,Nu,k,i\right) = \mathbb{E}_{(k,i)} \left[ f^0 \left( Z_{u}^{N}(t),K_{B_N}\left(N^{\beta_\delta} t\right),I_{B_N}\left(N^{\beta_\delta} t\right) \right) \right] \ ,\end{align*}
where we recall that
\begin{align*} 
Z_u^N(t) = N^{-1}Z_{u_N,B_N}\left(N^{\beta_\delta} t\right)= u- \frac{1}{2\pi N} \int_0^{N^{\beta_\delta} t} ds \; \textbf{v}_{B_N}\left(K_{B_N}(s)\right) \ . 
\end{align*}
Let $(m_N)_{N \in \N}$ be an increasing sequence of positive numbers such that
\begin{align}
&\lim\limits_{N \rightarrow \infty} m_N = + \infty \quad \text{and} \quad 
\lim\limits_{N \rightarrow \infty} \frac{m_N}{N} = 0 \label{def_sequence_m_N} \ .
\end{align}
Since $X_{B_N}(0)=\left(K_{B_N}(0),I_{B_N}(0)\right)=(k,i)$ with $k$ in $\T^*$ by  Theorem \ref{main_theorem_hydro} under $\mathbb{P}_{(k,i)}$ the finite-dimensional distributions of $Z_{u}^N\left(\left[1-N^{-\beta_\delta}m_N\right]\cdot\right)$ converge weakly to the finite-dimensional distributions of a L\'evy process $Z_u^\delta(\cdot)$ generated by $\mathfrak{L}_\delta$ where $\mathfrak{L}_\delta$ is defined in Eq.\;\eqref{def_Lev_gen_limit}. 

\medskip 
For any $i$ in $\{1,2\}$, $f^0(i) \in \Cc_{\mathbf{c}}^\infty\left(\R\times \T \right)$ hence $\overline{f}^0$ is in $\Cc_{\mathbf{c}}^\infty(\R)$ where $\overline{f}^0$ is defined in Eq.\;\eqref{def_u0_barre}. We have then
\begin{equation}
\label{convergence_laplaciel_frac}
\frac{1}{2}\lim\limits_{N \rightarrow \infty} \mathbb{E}_{(k,i)} \left[ \overline{f}^0 \left( Z_{u}^{N}\left(\left[1-N^{-\beta_\delta}m_N\right]t )  \right)\right) \right] = \frac{1}{2}\rho_\delta(u,t) \quad \mathbb{P}_{(k,i)} \quad a.s \ , 
\end{equation}
where $\rho_\delta$ is the unique solution of Eq.\;\eqref{limit_hyrdo_B=0}.
By using Fourier's inverse formula we have
\begin{align*}
&f_{B_N}\left(N^{\beta_\delta} t,Nu,k,i\right) \\
&= \mathbb{E}_{(k,i)} \left[ f^0 \left( Z_{u}^{N}\left( t\right),K_{B_N}\left(N^{\beta_\delta} t\right),I_{B_N}\left(N^{\beta_\delta} t \right) \right) \right]\\
& = \sum_{r \in \Z} \int_{\mathbb{R}} dp \; \sum_{j=1}^2 \mathcal{F}[\tilde{f}^0](p,r,j)\mathbb{E}_{(k,i)}\left[ e^{\mathbf{i}\left(pZ_{u}^{N}\left(t\right)+rK_{B_N}\left(N^{\beta_\delta} t\right)\right)} \mathds{1}_{I_{B_N}\left(N^{\beta_\delta} t\right)=j}   \right]\ ,
\end{align*}
where $\tilde{f}^0(i)$ is the function defined for any $z$ in $\mathbb{R}$, $i$ in $\{1,2\}$,  by 
\begin{equation}
f^0(z,k,i) := \sum_{r \in \Z} \tilde{f}^0 (z,r,i) \exp\left( 2 \mathbf{i} \pi kr  \right).
\end{equation}
By using Fourier's inverse formula we get
\begin{align*}
 &\frac{1}{2} \mathbb{E}_{(k,i)} \left[ \overline{f}^0\left( Z_{u}^{N} \left(  t-N^{-\beta_\delta}m_Nt \right ) \right) \right] \\ &=\sum_{j=1}^2\mathbb{E}_{(k,i)} \left[\frac{1}{2}\int_{\T \times \{1,2\}} dk^\prime d\rho(j) f^0 \left(Z_{u}^{N}\left( t-N^{\beta_\delta}m_Nt\right),k^\prime,j   \right)  \right]\\
&=\sum_{j=1}^2 \sum_{r \in \Z} \int_{\R} dp \; \mathcal{F}[\tilde{f}^0](p,r,j) \mathbb{E}_{(k,i)} \left[e^{\mathbf{i}p Z_{u}^{N}\left(t-N^{-\beta_\delta}m_Nt\right)} \frac{1}{2}\int_{\T} dk^\prime e^{\mathbf{i} k^\prime}   \right] \ .
\end{align*}
By using Eq.\;\eqref{convergence_laplaciel_frac} and the dominated convergence theorem we get
\begin{align*}
 &\limsup_{ N \rightarrow \infty}  \int_{\T \times \{1,2 \}} dk \; d\rho(i)  \; \left\vert f_{B_N} \left(N^{\beta_\delta} t, Nu,k  ,i \right) -\frac{1}{2}\rho_\delta(u,t) \right\vert \\
 &\leq \limsup_{ N \rightarrow \infty}  \int_{\T \times \{1,2 \}} dk \; d\rho(i) \; \left\vert A_{i,r,u,k,N} \right\vert \ ,
\end{align*}
where 
\[A_{i,r,u,k,N}:=  f_{B_N} \left(N^{\beta_\delta} t, Nu,k,i   \right) -\frac{1}{2}\mathbb{E}_{(k,i)} \left[  \overline{f}^0 \left( Z_{u}^{N}\left(\left[1-N^{-\beta_\delta}m_N\right]t \right)  \right) \right] \ .\]
By using Fourier's inverse formula we get
\begin{align*}
 \left\vert A_{i,r,u,k,N} \right\vert 
& \leq   \sum_{r \in \Z} \int_{\R} dp \; \sum_{j=1}^2 \left\vert \mathcal{F}[\tilde{f}^0_j](p,r) \right\vert \left( \left\vert I_1^N(t,k,i,j,p,r)\right\vert + \left\vert I_2^N(t,k,i,j,p,r) \right\vert \right)\ .
\end{align*}
where 
\begin{align*}
I_1^N(t,k,i,j,p,r) & = \mathbb{E}_{(k,i)}\left[e^{\mathbf{i}\left(pZ_{u}^{N}\left( t\right)+rK_{B_N}\left(N^{\beta_\delta} t\right)\right)} \mathds{1}_{I_{B_N}\left(N^{\beta_\delta} t\right)=j}  \right]\\
&- \mathbb{E}\left[ e^{\mathbf{i}\left(pZ_{u}^{N}\left( t-N^{-\beta_\delta}m_Nt\right)+rK_{B_N}\left(N^{\beta_\delta} t\right)\right)} \mathds{1}_{I_{B_N}\left(N^{\beta_\delta} t\right)=j}   \right]\ ,
\end{align*}
and
\begin{align*}
I_2^N(t,k,i,j,p,r) & = \mathbb{E}_{(k,i)}\left[e^{\mathbf{i}\left(pZ_{u}^{N}\left( t-N^{-\beta_\delta}m_Nt\right)+rK_{B_N}\left(N^{\beta_\delta} t\right)\right)} \mathds{1}_{I_{B_N}\left(N^{\beta_\delta} t\right)=j} \right]\\
&-\mathbb{E} \left[  e^{\mathbf{i}p Z_{u}^{N}\left(t-N^{-\beta_\delta}m_Nt\right)} \frac{1}{2}\int_{\T} e^{\mathbf{i} k^\prime} \; dk^\prime  \right] \ .
\end{align*}
To conclude the proof it is sufficient to show that for any positive time $t$,  $k\neq 0$, $(i,j)$ in $\{1,2\}^2$, $p$ in $\R$ and $r$ in $\Z$ we have 
\begin{equation}
\label{limit_I_i_1_2}
\forall q \in \{1,2\}, \quad \lim\limits_{N \rightarrow \infty}  \left\vert I_{q}^N(t,k,i,j,p,r) \right\vert  = 0 \ .
\end{equation} 
We start to deal with $I_1^N$.\\
We recall that for any $a$ in $\mathbb{R}$ we have $\left\vert 1-e^{\mathbf{i}a} \right\vert \leq \vert a \vert$ and that the function $\textbf{v}_{B_N}$  is a bounded function independently of $N$ (see Eq.\;\eqref{omega_prime}). Hence, for any positive $t$, $p$ in $\R$, $r$ in $\Z$ and $j$ in $\{1,2\}$ we have
\begin{align*}
 \left\vert I_1^N(t,k,i,j,p,r)  \right\vert
&  \leq \mathbb{E}_{k,i}\left[ \left\vert 1- e^{\mathbf{i}p\left(Z_{u}^{N}\left(t\right) - Z_{u}^{N}\left( t-N^{-\beta_\delta}m_N t\right)  \right)} \right\vert  \right]\\
&  \leq \mathbb{E}_{k,i}\left[ \left\vert p\left(Z_{u}^{N}( t) - Z_{u}^{N}\left( t-N^{-\beta_\delta}m_N t\right)  \right) \right\vert  \right] \\
& \leq \left\Vert \textbf{v}_{B_N} \right\Vert_{\infty} t \vert p \vert m_N N^{-1}\ ,
 \end{align*}
which goes to zero when $N$ goes to infinity by definition of the sequence $(m_N)_{N \in \N}$ $\left(\text{see Eq.}\;\eqref{def_sequence_m_N}\right)$. This proves Eq.\;\eqref{limit_I_i_1_2} for $i=1$. It remains to prove the convergence of $I_2^N$.
Using Markov Property we get
\begin{small}
 \begin{align*}
 & \mathbb{E}_{(k,i)}\left[ e^{\mathbf{i}\left(pZ_{u}^{N}\left( t-N^{-\beta_\delta}m_Nt\right)+rK_{B_N}\left(N^{\beta_\delta} t\right)\right)} \mathds{1}_{I_{B_N}\left(N^{\beta_\delta} t\right)=j}   \right] \\
 & =\mathbb{E}_{(k,i)}\left[ e^{\mathbf{i}\left(pZ_{u}^{N}\left( t-N^{-\beta_\delta}m_Nt\right) \right)}  \mathbb{E}_{\left(K_{B_N}\left(N^{\beta_\delta} t-m_Nt\right), I_{B_N}\left(N^{\beta_\delta} t-m_Nt\right)  \right)} \left[  e^{\mathbf{i} \left(rK_{B_N}(m_Nt)\right)} \mathds{1}_{I_{B_N}(m_Nt)=j}   \right] \right] \ . 
 \end{align*}
\end{small}
Hence, by Markov Property we get
 \begin{align*}
&\left\vert I_2^N(t,k,i,j,p,r) \right\vert \leq \mathbb{E}_{(k,i)} \left[\left\vert  \mathbb{E}_{\left(K_{B_N}\left(N^{\beta_\delta} t-m_Nt\right), I_{B_N}\left(N^{\beta_\delta} t-m_Nt\right)  \right)} g_r\left(K_{B_N},I_{B_N}\right) \right\vert   \right] ,
\end{align*}
where 
\begin{align*}
& g_r\left(K_{B_N},I_{B_N}\right) = e^{\mathbf{i}r K_{B_N}(m_Nt)} \mathds{1}_{I_{B_N}(m_Nt)=j} -\frac{1}{2}\int_{\T} e^{\mathbf{i}rk^\prime} \; dk^\prime \  .
\end{align*}
Let $r$ in $\N$ and $g_r$ the $\textbf{centered}$ function defined for any $k$ in $\T$ and $i$ in $\{1,2\}$ by
\[ g_r(k,i) = e^{\mathbf{i}rk} \mathds{1}_{j}(i) - \frac{1}{2}\int_{T} e^{\mathbf{i}rk^\prime} dk^\prime.\]
Using the fact that $\frac{1}{2}dkd\rho(i)$ is the reversible probability measure of the stochastic process $\left(K_{B_N}(\cdot),I_{B_N}(\cdot)\right)$ we get
\begin{small}
\begin{align*}
& \mathbb{E}_{(k,i)} \left[\left\vert  \mathbb{E}_{\left(K_{B_N}\left(N^{\beta_\delta} t-m_Nt\right), I_{B_N}\left(N^{\beta_\delta} t-m_Nt\right)  \right)} \left[ e^{\mathbf{i}r K_{B_N}(m_Nt)} \mathds{1}_{I_{B_N}(m_Nt)=j} -\frac{1}{2}\int_{\T} e^{\mathbf{i}rk^\prime} dk^\prime    \right] \right\vert   \right]  \\
& = \left\Vert P_N^{m_N t} g \right\Vert_{\mathbb{L}^1(\sigma)} \ ,
\end{align*}
\end{small}
which goes to zero when $N$ goes to infinity according to Lemma \ref{lemma_convergence_P}. 
This ends the proof of Theorem \ref{main_thm_limit_u} . 
\subsection{Proof of Theorem \ref{main_thm_limit_B_u_barre}}
\label{section_main_result_B} We recall that $\tau_0 \sim \Ec\left( 1\right)$.
By applying Fourier's formula to Eq.\;\eqref{limit_hyrdo_B=0} we get for any $p$ in $\mathbb{R}$ and positive time $t$ in $]0,T]$
\begin{equation}
    \mathcal{F}\left[\tilde{\rho}_{B}\right](t,p) = \exp\left(t \Phi_{B}(p) \right) \mathcal{F}\left[\overline{f}^{0}\right](p) \ ,
\end{equation}
where we recall that for any $p$ in $\mathbb{R}$ 
\begin{align*}
\Phi_{B}(p) &= 
2\int_{0}^\infty \left( \cos(pr) - 1 \right)\left( \mathbb{E}\left[ \tau_0^{-1} \left( g_{B,+}\left(\frac{2\pi r}{\tau_0}\right) + g_{B,-}\left(\frac{2\pi r}{\tau_0}\right) \right)\right] \right) \; dr \ ,
\end{align*}
since $g_{B,\pm}$ are even functions by item i) of Proposition \ref{proposition_f}. Observe that $\Phi_{B}$ is negative, to conclude the proof we need the following lemma the proof of which is left to the readers.
\begin{lemma}
\label{lemma_tcd}
There exists functions $h_\pm$ and $f_\pm$ such that for almost every $r$
\begin{align}
   &\forall B>1, \quad  B^{\frac{1}{3}} g_{B,\pm}(r) < f_\pm(r) \ ,  \label{hypothese_domination_2} \\
   &\forall B<1, \quad  g_{B,\pm}(r) < h_\pm(r) \ , \label{hypothese_domination_1}
\end{align}
where 
\begin{align}
\int_{0}^1 f_\pm(r)r^2 \; dr < \infty \quad \text{and} \quad \int_{1}^\infty f_\pm(r) \; dr < \infty \ , \label{hypothese_integrabilite_2}\\
    \int_{0}^1 h_\pm(r)r^2 \; dr < \infty \quad \text{and} \quad \int_{1}^\infty h_\pm(r) \; dr < \infty \ . \label{hypothese_integrabilite_1}
\end{align}
\end{lemma}
By using item iii), item iv) of Proposition \ref{proposition_f}, Lemma \ref{lemma_tcd} and the dominated convergence theorem we get that for almost every $p$ in $\R$
\begin{align*}
    \lim\limits_{B \rightarrow 0}  \Phi_{B}(p) &= -\Gamma\left( 1+\frac{3}{2}\right)\gamma^{-\tfrac{1}{2}} D_0\vert p \vert^{\frac{3}{2}} \ , \\
      \lim\limits_{B \rightarrow \infty} B^\frac{1}{3} \Phi_{B}(p) &= -\Gamma\left( 1+\frac{5}{3}\right) \gamma^{-\tfrac{2}{3}}D_\infty\vert p \vert^{\frac{5}{3}} \  ,
\end{align*}
where $D_0$ and $D_\infty$ are defined in Eq.\;\eqref{def_vrai_constante} and $\Gamma$ denotes the Gamma function.
Since the proofs of item i) and ii) of Theorem \ref{main_thm_limit_B_u_barre} are similar we only prove item i). By applying Fourier's formula to Eq.\;\eqref{rho_zero} we obtain 
\begin{equation}
    \mathcal{F}\left[\rho_{0}\right](t,p) = \exp\left(-t \Gamma\left( 1+\frac{3}{2}\right)\gamma^{-\tfrac{1}{2}}D_0\vert p \vert^{\tfrac{3}{2}} \right) \mathcal{F}\left[\overline{f}^{0}\right](p) \ .
\end{equation}
We recall that $\overline{f}^0$ is in $\mathcal{C}_{\mathbf{c}}^\infty\left( \mathbb{R}\right)$ hence is in $\mathbb{L}^2(\mathbb{R})$. By the dominated convergence theorem we deduce that for any positive time $t$ in $[0,T]$
\begin{align*}
    &\lim\limits_{B \rightarrow 0} \Vert \tilde{\rho}_{B}(t,\cdot) - \rho_0(t,\cdot) \Vert_{\mathbb{L}^2(\mathbb{R})}^2 \\
    &=\lim\limits_{B \rightarrow 0} \Vert \mathcal{F}\left[\tilde{\rho}_{B}\right](t,\cdot) - \mathcal{F}\left[\rho_0\right](t,\cdot) \Vert_{\mathbb{L}^2(\mathbb{R})}^2 \\
    & = \lim\limits_{B \rightarrow 0} \int_{\R} \left\vert \mathcal{F}\,\left[\overline{f}^{0}\right](p) \left( \exp\left(t \Phi_{B}(p) \right) -  \exp\left(-t \Gamma\left( 1+\tfrac{3}{2}\right)\gamma^{-\tfrac{1}{2}}D_0\vert p \vert^{\tfrac{3}{2}} \right) \right) \right\vert^2 dp\\
    & = 0 \ .
\end{align*}
This proves that for each positive time $t$, $\rho_{B}(t,\cdot)$ converges in $\mathbb{L}^2(\R)$ to $\rho_{0}(t,\cdot)$. By the dominated convergence theorem we conclude the proof of Theorem \ref{main_thm_limit_B_u_barre} . 
\appendix
\section{Proof of the evolution equation of the Wigner distribution}
\label{appendix_kinetic}
In this appendix we want to prove Eq.\;\eqref{evolution_eq_ib}, this is the aim of Appendix \ref{appendix_behavior_transport_term} and Appendix \ref{appendix_behavior_colli_term}. In a first instance, in Appendix \ref{appendix_properties_Wigner} we prove Lemma \ref{Lem_Wigner}. We end this section by the proof of Lemma \ref{lemma_generalisation_hypo_small_k}.\\

In order to lighten the notations we denote, for $t$ in $[0,T]$, $i$ in $\lbrace 1,2 \rbrace$, $k$ in $\mathbb{T}$ and $p$ in $\mathbb{R}$, $\widehat{\psi}^\varepsilon_i\left( t\varepsilon^{-1}, k\pm \frac{\varepsilon p}{2}\right)$ by $\widehat{\psi}^\varepsilon_i\left(k\pm\frac{\varepsilon p}{2}\right) $. 
\subsection{Proof of Lemma \ref{Lem_Wigner}}
\label{appendix_properties_Wigner}
Before proving Lemma \ref{Lem_Wigner}, we prove that the integral in Eq.\;\eqref{wigner_def} is well defined. Let $J:=(J_1,J_2)$ in $\Sc \times \Sc$ and $i$ in $\{1,2\}$, by using Cauchy-Schwarz's inequality, Fubini's theorem and the periodicity of $\widehat{\psi}^\varepsilon_i$ we have
\begin{align*}
&\frac{\varepsilon}{2} \int_{\mathbb{R}} dp \;  \left\vert \Fc \left[J_i\right](p,k) \right\vert \int_{\mathbb{T}} dk  \; \mathbb{E}_{\mu^\varepsilon} \left[ \left\vert \widehat{\psi}^\varepsilon_i\left( k+\frac{\varepsilon p}{2}\right) \widehat{\psi}^\varepsilon_i\left( k-\frac{\varepsilon p}{2}\right)^* \right\vert \right]  \\
&\leq \frac{\varepsilon}{2} \left(\int_{\mathbb{R}} dp \;  \sup_{k \in \T} \left\vert \Fc \left[J_i\right](p,k) \right\vert \right) \left( \mathbb{E}_{\mu^\varepsilon} \left[ \int_\mathbb{T} dk \;  \left\vert\widehat{\psi}^\varepsilon_i\left(k \right)  \right\vert^2 \right] \right)\\
& \leq K_0 \Vert J \Vert \ ,
\end{align*}
where we used Eq.\;\eqref{K0} to obtain the last inequality. This proves the existence of the integral in Eq.\;\eqref{wigner_def}. The other expressions of the Wigner distribution presented in Sec.\;\ref{section_wigner} are consequences of Fubini's theorem and Fourier's inverse formula. Observe that we proved in fact
\begin{eqnarray}
\label{bound_wigner}
\forall t \in [0,T], \quad \left\vert\left\langle \mathcal{W}^\varepsilon_i(t), J \right\rangle \right\vert  \leq K_0 \Vert J \Vert \ .
\end{eqnarray}
Using the same argument we can prove that 
\begin{eqnarray}
\label{bound_anti_wigner}
\forall t \in [0,T], \quad \left\vert\left\langle \mathcal{A}^\varepsilon_i(t), J \right\rangle \right\vert  \leq K_0 \Vert J \Vert \quad \text{and} \quad \left\vert\left\langle \left({\mathcal{A}}_i^\varepsilon\right)^*, J \right\rangle \right\vert  \leq K_0 \Vert J \Vert \ .
\end{eqnarray}
\paragraph{\textbf{Proof of Lemma \ref{Lem_Wigner}}}
By Eq.\;\eqref{bound_wigner} and the linearity of the Fourier transform the proof of item i) is immediate.
To prove item ii) we use the evolution equation $( \text{see Eq.\;\eqref{evolution_eq_ib}})$ satisfied by the Wigner distribution. Let $J$ in $\Sc\times\Sc$, in Appendix \ref{appendix_behavior_transport_term} and Appendix \ref{appendix_behavior_colli_term} we prove that
\begin{eqnarray*}
\left\langle \mathcal{W}^\varepsilon(t), J \right\rangle - \left\langle \mathcal{W}^\varepsilon(0), J \right\rangle   & = & \int_0^t ds \;  \left\langle \mathcal{W}^\varepsilon(s),  \textbf{v}_{0} \partial_u J \right\rangle  + \gamma \int_0^t ds \; \left\langle \mathcal{W}^\varepsilon(t), C_0 J \right\rangle \\
& + & \gamma \left( \int_0^t ds \;  \langle\mathcal{A}^\varepsilon(s),C_0J \rangle    + \int_0^t  ds \;  \langle\left({\mathcal{A}}^\varepsilon\right)^*(s),C_0J \rangle \right) \\
&+ & \int_{0}^t ds \;  \mathcal{O}_{s,J}(\varepsilon) \ ,
\end{eqnarray*}
with $\vert \Oc_{s,J} \vert \leq K_J$ a constant independent of $\varepsilon$ and time $s$.
Since for any $t$ in $[0,T]$, $ \left\langle \mathcal{W}^\varepsilon(t), J \right\rangle$ is the integral of a bounded function this is a continuous function. Using Eq.\;\eqref{bound_wigner} and Eq.\;\eqref{bound_anti_wigner} we conclude the proof of Lemma \ref{Lem_Wigner}.

\subsection{Behavior of the transport term} 
\label{appendix_behavior_transport_term}
We recall that $B_{\varepsilon}=B\varepsilon^{\delta}$ with $\delta > 0$.
To prove Eq.\;\eqref{evolution_eq_ib} we use Dynkin's formula, Lemma \ref{Lem_Wigner} and Lemma \ref{lemma_anti_wigner}. Let $J:=(J_1,J_2)$ in $\Sc \times \Sc$, $t$ in $[0,T]$ and $i$ in $\{1,2 \}$ we want to prove that
\begin{align*}
\partial_t \left\langle \mathcal{W}_i^\varepsilon(t), J_i \right\rangle & =  \frac{1}{2\pi}\left\langle \mathcal{W}_i^\varepsilon(t), \textbf{v}_0 \partial_u J_i \right\rangle  +   \left\langle \mathcal{W}_i^\varepsilon(t), [C_0J]_i \right\rangle \\
& + \langle\mathcal{A}_i^\varepsilon(t),[C_0J]_i \rangle + \langle\left({\mathcal{A}}_i^\varepsilon\right)^*(t),[C_0J]_i \rangle + \mathcal{O}_{t,J}(\varepsilon) \ .
\end{align*}
For any $i$ in $\{ 1,2\}$ we define $\phi_i$ by
\begin{align*}
 \phi_i \left(q^\varepsilon(t), p^\varepsilon(t)\right) &=  \int_\mathbb{R} dp \int_\mathbb{T} dk \;  \widehat{\psi}^\varepsilon_i\left(t\varepsilon^{-1},k-\tfrac{\varepsilon p}{2} \right)^* \widehat{\psi}^\varepsilon_i\left(t\varepsilon^{-1},k+\tfrac{\varepsilon p}{2}\right) \Fc \left[J_i\right](p,k)^* \ .
 \end{align*}
By the equation above we can write that
\begin{align} \left\langle \mathcal{W}_i^\varepsilon (t), J_i \right\rangle = \frac{\varepsilon}{2}\E_{\mu^\varepsilon} \left[ \phi_i \left(q^\varepsilon\left(t\right), p^\varepsilon\left(t\right)\right) \right] \ .
\end{align}
By Dynkin's formula we have
\begin{align*}
 \partial_t \left\langle \mathcal{W}_i^\varepsilon (t), J_i \right\rangle & =  \varepsilon^{-1}    \mathbb{E}_{\mu^\varepsilon} \left\{ (A+B_{\varepsilon}G) \left[ \phi_i \left(q^\varepsilon\left(t\right), p^\varepsilon\left(t\right)\right) \right] \right\} + \gamma \mathbb{E}_{\mu^\varepsilon} \left\{ S \left[ \phi_i \left(q^\varepsilon\left(t\right), p^\varepsilon\left(t\right)\right) \right]\right\} \ .
\end{align*}
To complete the proof of Eq.\;\eqref{evolution_eq_ib} it is sufficient to prove that 
\begin{align}
     &\varepsilon^{-1}    \mathbb{E}_{\mu^\varepsilon} \left\{ (A+B_{\varepsilon}G) \left[ \phi_i \left(q^\varepsilon\left(t\right), p^\varepsilon\left(t\right)\right) \right] \right\} = \frac{1}{2 \pi}\langle \mathcal{W}_i^{\varepsilon} (t),  \textbf{v}_{0} \partial_u J_i \rangle + \mathcal{O}_J(\varepsilon) \label{Sandra_1} \ , \\ 
     &\mathbb{E}_{\mu^\varepsilon} \left\{ S \left[ \phi_i \left(q^\varepsilon\left(t\right), p^\varepsilon\left(t\right)\right) \right]\right\} = \left( \langle\mathcal{A}_i^\varepsilon(t),[C_0J]_i \rangle + \langle\left({\mathcal{A}}_i^\varepsilon\right)^*(t),[C_0J]_i \rangle \right)  +  \mathcal{O}_{J}(\varepsilon)  \label{Sandra_2} \ . 
\end{align}
In this section we prove Eq.\;\eqref{Sandra_1} and in Appendix \ref{appendix_behavior_colli_term} we prove Eq.\;\eqref{Sandra_2}. \\\\
Let $i$ in $\{1,2\}$, by Lemma \ref{eig} and the fact that $A+B\varepsilon^\delta G$ is a first order differential operator we have
\begin{align}
   &\varepsilon^{-1}\mathbb{E}_{\mu^\varepsilon} \left\{ (A+B_{\varepsilon}G) \left[ \phi_i \left(q^\varepsilon\left(t\right), p^\varepsilon\left(t\right)\right) \right] \right\} \nonumber \\
   & = \frac{\varepsilon}{2}  \int_\mathbb{R} dp  \int_\mathbb{T} dk \; \mathbb{E}_{\mu^\varepsilon} \left[ \widehat{\psi}^\varepsilon_i\left(k-\frac{\varepsilon p}{2}\right)^* \widehat{\psi}^\varepsilon_i\left(k+\frac{\varepsilon p}{2}\right)\right]  \Fc \left[J_i\right](p,k)^*   \nonumber \\
& \times  \frac{\mathbf{i}}{\varepsilon }\left[\omega_{i,B_\varepsilon}\left(k+\frac{\varepsilon p}{2}\right)- \omega_{i,B_\varepsilon}\left(k-\frac{\varepsilon p}{2}\right)\right] \label{I am batman} \ . 
\end{align}
We split the proof of Eq.\;\eqref{Sandra_1} into two steps, the first one which follows \cite{BOS09} is to prove the following 
\begin{equation}
\label{Billel}
\varepsilon^{-1}\mathbb{E}_{\mu^\varepsilon} \left\{ (A+B_{\varepsilon}G) \left[ \phi_i \left(q^\varepsilon\left(t\right), p^\varepsilon\left(t\right)\right) \right] \right\} =  \frac{1}{2 \pi}\langle \mathcal{W}_i^{\varepsilon} (t),  \textbf{v}_{B_\varepsilon} \partial_u J \rangle + \mathcal{O}_J(\varepsilon)  \ , 
\end{equation}
where  $\vert \mathcal{O}_J(\varepsilon) \vert \leq \kappa_J(B) \times \varepsilon$ and $\kappa_J(B)$ is a constant which depends only on J and $B$. The second step will be to prove the following equality
\begin{align}
\label{goal}
&\left\langle \Wc_i^{\varepsilon}(t),  \textbf{v}_{B_\varepsilon} \partial_u J_i \right\rangle = \left\langle \Wc_i^{\varepsilon}(t),  \textbf{v}_{0} \partial_u J_i \right\rangle + \mathcal{O}_J(\varepsilon) \ .
\end{align} 
To prove Eq.\;\eqref{Billel}, we need the following lemma which is proved at the end of this section.
\begin{lemma}
\label{lemma_BOB}We recall that the functions $\omega_{1/2,B_{\varepsilon}}$ and $\mathbf{v}_{B_\varepsilon}$ are defined in Eq.\;\eqref{omega} and Eq.\;\eqref{omega_prime} respectively. They satisfie the following properties.
\begin{itemize}
    \item [i)] $\mathbf{v}_{B_\varepsilon}$ is a bounded function on $\mathbb{T}$. 
    \item[ii)] There exists a positive constant $C$ such that 
    \begin{align*}
      \forall p \in \mathbb{R}, \quad \forall k \in \mathbb{T}, \quad   \varepsilon^{-1}\left\vert \omega_{i,B_\varepsilon}\left(k+\frac{\varepsilon p}{2}\right)- \omega_{i,B_\varepsilon}\left(k-\frac{\varepsilon p}{2}\right) \right\vert \leq C \vert p \vert  \ . 
    \end{align*}
    \item[iii)] For all $p$ in $\mathbb{R}$ and $k$ in $\mathbb{T}$ such that $ \vert k \vert > \varepsilon \vert p \vert$ we have
    \begin{align*}
    \varepsilon^{-1}\left\vert \omega_{i,B_\varepsilon}\left(k+\frac{\varepsilon p}{2}\right)- \omega_{i,B_\varepsilon}\left(k-\frac{\varepsilon p}{2}\right) - \varepsilon p \mathbf{v}_{B_\varepsilon}(k) \right\vert \lesssim  \varepsilon \frac{\vert p \vert^2}{\vert k \vert } \ . 
    \end{align*}
\end{itemize}
\end{lemma}
Let $\kappa > 0$ and $\eta$ in $]0,1[$ with $\delta > \eta$. We denote by $\kappa_\varepsilon$, the real $\kappa\varepsilon^{\eta}$ and define $I^\varepsilon$ as follows 
\begin{align}
    I^\varepsilon = \varepsilon^{-1}\mathbb{E}_{\mu^\varepsilon} \left\{ (A+B_{\varepsilon}G) \left[ \phi_i \left(q^\varepsilon\left(t\right), p^\varepsilon\left(t\right)\right) \right] \right\} - \left\langle \Wc_i^{\varepsilon}(t),  \textbf{v}_{B_\varepsilon} \partial_u J_i \right\rangle  \ . 
\end{align} 
Using Eq.\;\eqref{I am batman} we cut $I^\varepsilon$ into two terms denoted by $I^\varepsilon_<\left(\kappa_\varepsilon\right)$ and $I^\varepsilon_>\left(\kappa_\varepsilon\right)$ with
\begin{align*}
I_<^\varepsilon\left(\kappa_\varepsilon\right) & =  \frac{\varepsilon}{2}  \int_\mathbb{R} dp \int_{ \vert k \vert < \kappa_\varepsilon} dk  \; \mathbb{E}_{\mu^\varepsilon} \left[ \widehat{\psi}^\varepsilon_i\left(k-\frac{\varepsilon p}{2}\right)^* \widehat{\psi}^\varepsilon_i\left(k+\frac{\varepsilon p}{2}\right)\right] \Fc \left[J_i\right](p,k)^*  \\
& \times  \frac{\mathbf{i}}{\varepsilon }\left[\omega_{i,B_\varepsilon}\left(k+\frac{\varepsilon p}{2}\right)- \omega_{i,B_\varepsilon}\left(k-\frac{\varepsilon p}{2}\right) - \varepsilon p \textbf{v}_{B_\varepsilon}(k)\right] \ .
\end{align*}   
By using point i) and point ii) of Lemma \ref{lemma_BOB} and the fact that $J_i$ belongs to $\mathcal{S}$ we have
\begin{align*}
&\left\vert I^\varepsilon_<\left(\kappa_\varepsilon\right) \right\vert \lesssim \int_{\vert k \vert < \kappa_\varepsilon} dk \; \E\left[\left\vert \widehat{\psi}^\varepsilon_i(k) \right\vert^2 \right]\ . 
\end{align*}
By Lemma \ref{lemma_generalisation_hypo_small_k} we obtain then
\begin{align*}
& \limsup_{\varepsilon \rightarrow 0} I^\varepsilon_<\left(\kappa_\varepsilon\right) = 0 \ .
\end{align*}
It remains to deal with $I^\varepsilon_>\left(\kappa_\varepsilon\right)$. We have
\begin{align*}
I_>^\varepsilon\left(\kappa_\varepsilon\right) & =  \frac{\varepsilon}{2}  \int_{\varepsilon \vert p \vert > \kappa_\varepsilon}dp \int_{ \vert k \vert > \kappa_\varepsilon} dk \; \mathbb{E}_{\mu^\varepsilon} \left[ \widehat{\psi}^\varepsilon_i\left(k-\frac{\varepsilon p}{2}\right)^* \widehat{\psi}^\varepsilon_i\left(k+\frac{\varepsilon p}{2}\right)\right]\Fc \left[J_i\right](p,k)^*\\
 & \times   \frac{\mathbf{i}}{\varepsilon }\left[\omega_{i,B_\varepsilon}\left(k+\frac{\varepsilon p}{2}\right)- \omega_{i,B_\varepsilon}\left(k-\frac{\varepsilon p}{2}\right) - \varepsilon p \textbf{v}_{B_\varepsilon}(k)\right]   \\
 &+  \frac{\varepsilon}{2}  \int_{\varepsilon \vert p \vert < \kappa_\varepsilon} dp \int_{ \vert k \vert > \kappa_\varepsilon} dk \; \mathbb{E}_{\mu^\varepsilon} \left[ \widehat{\psi}^\varepsilon_i\left(k-\frac{\varepsilon p}{2}\right)^* \widehat{\psi}^\varepsilon_i\left(k+\frac{\varepsilon p}{2}\right)\right] \Fc \left[J_i\right](p,k)^*  \\
 & \times \frac{\mathbf{i}}{\varepsilon }\left[\omega_{i,B_\varepsilon}\left(k+\frac{\varepsilon p}{2}\right)- \omega_{i,B_\varepsilon}\left(k-\frac{\varepsilon p}{2}\right) - \varepsilon p \textbf{v}_{B_\varepsilon}(k)\right] \\
 & = I^\varepsilon_{>,1}\left(\kappa_\varepsilon\right) + I^\varepsilon_{>,2}\left(\kappa_\varepsilon\right) \ .
\end{align*}
By the points i) and ii) of Lemma \ref{lemma_BOB} and Eq.\;\eqref{K0} we have
\begin{align*} 
\left\vert I^\varepsilon_{>,1}\left(\kappa_\varepsilon\right) \right\vert \lesssim  &  K_0 \int_{ \vert p \vert  > \kappa \varepsilon^{\eta-1}} dp \;    (p+ 1) \sup_{k \in \T} \vert\Fc \left[J_i\right] (p,k) \vert   \ ,
\end{align*}
which goes to zero when $\varepsilon$ goes to zero since $\eta <1$. For the remaining term we have that $\vert k \vert > \kappa_\varepsilon > \varepsilon \vert p \vert $ hence by point iii) of Lemma \ref{lemma_BOB} we get
\begin{align*}
    \left\vert I^\varepsilon_{>,2}\left(\kappa_\varepsilon\right) \right\vert 
    & \lesssim K_0 \int_{\mathbb{R}} dp \; \frac{\varepsilon \vert p \vert^2}{\kappa_\varepsilon} \sup_{k \in \T} \vert\Fc \left[J_i\right] (p,k) \vert \ ,
\end{align*}
which goes to zero since $\eta < 1$. \\\\
In order to end the computations for the transport term it remains to prove Eq.\;\eqref{goal}. We define $R^\varepsilon $ in the following way
\begin{align*} 
& R^\varepsilon=\left\langle \Wc_i^{\varepsilon}, \left[\textbf{v}_{B_\varepsilon}- \textbf{v}_{0}\right] \partial_u J_i \right\rangle \ .
\end{align*}
Hence, we want to prove that $(R^\varepsilon)_{\varepsilon >0}$ converges to zero when $\varepsilon$ goes to zero. As we did before for $I^\varepsilon$, we cute $R^\varepsilon$ into two terms $R^\varepsilon_{<}(\kappa_\varepsilon)$ and $R^\varepsilon_{>}(\kappa_\varepsilon)$ where 
\begin{align*}
 R_{<}^\varepsilon\left(\kappa_\varepsilon\right) &=\left\langle \Wc_i^{\varepsilon}, \left[\textbf{v}_{B_\varepsilon}- \textbf{v}_{0}\right] \partial_u J_i \mathds{1}_{\{k \leq \kappa_\varepsilon \}} \right\rangle \quad \text{and} \quad R_{>}^\varepsilon\left(\kappa_\varepsilon\right) &=\left\langle \Wc_i^{\varepsilon}, \left[\textbf{v}_{B_\varepsilon}- \textbf{v}_{0}\right] \partial_u J_i \mathds{1}_{\{k \geq \kappa_\varepsilon \}} \right\rangle \ .
 \end{align*}
We start to show that $(R^\varepsilon_{>}\left(\kappa_\varepsilon\right))_{\varepsilon >0}$ converges to zero. Using the fact that $\textbf{v}_{0}$ is bounded and the fact that $k> \kappa_\varepsilon$ we can bound $\vert R_{>}^\varepsilon \left(\kappa_\varepsilon\right) \vert$ by
\begin{align*}
&  \frac{\varepsilon}{2}  \int_\mathbb{R} dp \int_{ \vert k \vert > \kappa_\varepsilon} dk \; \mathbb{E}_{\mu^\varepsilon}\left[  \left\vert \widehat{\psi}^\varepsilon_i\left(k-\frac{\varepsilon p}{2}\right)^* \widehat{\psi}^\varepsilon_i\left(k+\frac{\varepsilon p}{2}\right) \right\vert\right] \sup_{k \in \mathbb{T}} \left\vert \mathcal{F}\left[J_i\right](p,k) \right\vert  \frac{\vert p \vert}{\kappa_\varepsilon}  \\
& \times  \left\vert \sqrt{\widehat{\alpha} (k)} - \sqrt{\widehat{\alpha}(k)+\frac{B^2 \varepsilon^{2\delta}}{4}} \right\vert \ .
\end{align*}
Since $k > \kappa_\varepsilon$, $\widehat{\alpha}(k) > 0$ and by using a Taylor expansion we get 
\begin{align*}
\left\vert R^\varepsilon_{>}( \rho) \right\vert & \lesssim  \frac{\varepsilon}{2}  \int_\mathbb{R} dp \int_{ \vert k \vert > \kappa_\varepsilon} dk \; \mathbb{E}_{\mu^\varepsilon}\left[ \left\vert \widehat{\psi}^\varepsilon_i\left(k-\frac{\varepsilon p}{2}\right)^* \widehat{\psi}^\varepsilon_i\left(k+\frac{\varepsilon p}{2}\right) \right\vert\right] \\
& \times \sup_{k \in \mathbb{T}} \left\vert \Fc \left[J_i\right](p,k)\right\vert  \times  \left(  \frac{\vert p \vert B^2  \varepsilon^{2\delta}}{\kappa_\varepsilon\sqrt{\widehat{\alpha}(k)}}  \right) \\
& \lesssim  \frac{\varepsilon}{2}  \int_\mathbb{R} dp \int_{ \vert k \vert > \kappa_\varepsilon} dk \; \mathbb{E}_{\mu^\varepsilon}\left[ \left\vert \widehat{\psi}^\varepsilon_i\left(k-\frac{\varepsilon p}{2}\right)^* \widehat{\psi}^\varepsilon_i\left(k+\frac{\varepsilon p}{2}\right) \right\vert\right] \\
& \times \sup_{k \in \mathbb{T}} \left\vert \Fc \left[J_i\right](p,k)\right\vert  \times  \left( \frac{\vert p \vert B^2 \varepsilon^{2\delta}}{\vert \kappa_\varepsilon \vert^2}  \right)\\
& \lesssim  K_0 \varepsilon^{2\left(\delta-\eta\right)} \ ,
\end{align*}
which goes to zero when $\varepsilon$ goes to zero since $\delta > \eta$ by assumption. 
It remains to show that $\left(R^\varepsilon_<\left(\kappa_\varepsilon\right)\right)_{\varepsilon > 0}$ converges to zero. By using the fact that $\textbf{v}_{0}$ and $\textbf{v}_{B_\varepsilon}$ are bounded and the fact that $J_i$ belongs to $\Sc$ we obtain for any time $t$ in $[0,T]$
\begin{align*}
 \left\vert R^\varepsilon_{<}\left(\kappa_\varepsilon\right) \right\vert \lesssim \frac{\varepsilon}{2} \int_{ \vert k \vert < \kappa_\varepsilon  } dk \; \mathbb{E}_{\mu^\varepsilon}\left[  \left\vert \widehat{\psi}^\varepsilon_i\left(k\right) \right\vert^2\right]  = \frac{\varepsilon}{2} \int_{ \vert k \vert < \kappa_\varepsilon  } dk \; \mathbb{E}_{\mu^\varepsilon}\left[  \left\vert \widehat{\psi}^\varepsilon_i\left(t\varepsilon^{-1},k\right) \right\vert^2\right] \ ,
\end{align*} 
which goes to zero by Lemma \ref{lemma_generalisation_hypo_small_k}. Hence, we proved that for any $i $ in $\lbrace 1,2 \rbrace$ 
\begin{align*}
    \mathbb{E}_{\mu^\varepsilon} \left\{ (A+B_\varepsilon G) \left[ \phi_i \left(q^\varepsilon\left(t\right), p^\varepsilon\left(t\right)\right) \right] \right\} = \left\langle \Wc_i^{\varepsilon}(t),  \textbf{v}_{0} \partial_u J_i \right\rangle + \mathcal{O}_J(\varepsilon) \ .
\end{align*}
So far we obtained the first term of  Eq.\;\eqref{evolution_eq_ib} which corresponds to the transport term of the Boltzmann equation $\left(\text{see Eq.}\;\eqref{BE}\right)$. In the next section, we obtain the last  terms of Eq.\;\eqref{evolution_eq_ib}. Before doing this, we give the sketch of the proof of Lemma \ref{lemma_BOB}. 
\begin{proof}[Sketch of the proof of Lemma \ref{lemma_BOB}]
We recall that $\widehat{\alpha}(k) = 4\pi^2 \sin^2\left( \pi k \right)$.\newline Hence, by definition we have for every $k$ in $\mathbb{T}$
\begin{align*}
    \textbf{v}_{B_\varepsilon}(k) = \frac{\widehat{\alpha}^\prime(k)}{2\sqrt{\widehat{\alpha}(k) + \frac{B_\varepsilon^2}{4}}} = \frac{4\pi^3 \cos\left( \pi k \right) \sin\left(\pi k \right)}{\sqrt{\sin^2\left( \pi k \right) +  \frac{B_\varepsilon^2}{4} }} \lesssim  4\pi^3 \ . 
\end{align*}
This ends the proof of point i). Using a Taylor expansion, the proof of point ii) follows using the item i). It remains to deal with point iii). For $p=0$, the results is obvious since $\vert k \vert > \varepsilon \vert p \vert $ hence it is sufficient to discuss the case $\vert k \vert > 0$. Let $f_i$ and $g_i$ be two real-valued functions defined on $]0,\varepsilon[$ by 
\begin{align*}
    f_i(s) = \omega_{i,B_\varepsilon}\left(k + \frac{s p}{2} \right) \quad \text{and} \quad  g_i(s) = \omega_{i,B_\varepsilon}\left(k - \frac{s p}{2} \right) \ .
\end{align*}
Then we have
\begin{align*}
    \varepsilon^{-1}\left\vert \omega_{i,B_\varepsilon}\left(k+\frac{\varepsilon p}{2}\right)- \omega_{i,B_\varepsilon}\left(k-\frac{\varepsilon p}{2}\right) - \varepsilon p \textbf{v}_{B_\varepsilon}(k) \right\vert  & = \varepsilon^{-1} \left[ f_i(\varepsilon) - f_i(0) - \varepsilon f_i^\prime(0) \right] \\ 
    &- \varepsilon^{-1} \left[ g_i(\varepsilon) - g_i(0) - \varepsilon g_i^\prime(0) \right] \ .
\end{align*} 
We conclude the proof by a Taylor expansion and using the fact that $\widehat{\alpha}(k) > C \vert k \vert $ where $C$ is a positive constant. 
\end{proof}
\subsection{\textsc{Behavior of the collisional term}}
\label{appendix_behavior_colli_term}
Let $i$ in $\{1,2\}$ and $i^*=3-i$, following the proof did by the authors of \cite{SSS19} we can prove that
\begin{align}
\gamma \mathbb{E}_{\mu^\varepsilon} \left[ S \left[ \phi_i \left(q^\varepsilon\left(t\varepsilon^{-1}\right), p^\varepsilon\left(t\varepsilon^{-1}\right)\right) \right]\right] & = \gamma \left( \left\langle \Wc^\varepsilon_i(t), [C_{B_\varepsilon}J]_i  \right\rangle + \left\langle \mathcal{A}^\varepsilon_i(t), [K_{B_\varepsilon}J]_i \right\rangle \right) \label{peakly} \\
&+ \gamma \left(\left\langle (\mathcal{A}^\varepsilon_i)^*(t), [K_{B_\varepsilon}J]_i \right\rangle \right)  + \mathcal{O}_{J}(\varepsilon) \nonumber \ ,
\end{align}
where for any $(u,k)$ in  $\mathbb{R} \times \mathbb{T}$
\begin{align*}
[C_{B_\varepsilon}J]_i (u,k) &= \sum_{j=1}^2 \int_\mathbb{T} dk^\prime  \; \theta_{i,B_\varepsilon}^2(k) R\left(k,k^\prime\right) \theta_{j,B_\varepsilon}^2\left(k^\prime\right) \left[J_j(u,k^\prime) - J_i(u,k) \right] \ , \\
[K_{B_\varepsilon}J]_i (u,k) &=  \int_\mathbb{T} dk^\prime \; \theta_{1,B_\varepsilon}\left( k\right)  \theta_{2,B_\varepsilon}(k)R\left(k,k^\prime\right)  \left[\theta_{i^*,B_\varepsilon}^2\left(k^\prime\right)J_{i^*}(u,k^\prime) -\frac{J_i(u,k)}{2} \right] \ .
\end{align*}
To conclude the proof of Eq.\;\eqref{Sandra_2} we have to replace $K_{B_\varepsilon}J$ and $C_{B_\varepsilon}J$ in Eq.\;\eqref{peakly} by $C_{0}J$  where for any $(u,k)$ in $\mathbb{R} \times \mathbb{T}$
\[ [C_0J]_i (u,k) = \sum_{j=1}^2 \int_\mathbb{T} dk^\prime \; \frac{1}{4} R\left(k,k^\prime\right) \left[J_j\left(u,k^\prime\right) - J_i(u,k) \right] \ .\]
Since the arguments are similar we only prove that
\begin{equation} \label{behav_coll}
 \left\langle \Wc_i^\varepsilon(t), [C_{B_\varepsilon}J]_i \right\rangle = \left\langle \Wc_i^\varepsilon(t), [C_0J]_i \right\rangle + \mathcal{O}_J(\varepsilon)\ .
 \end{equation}
We define an operator $\tilde{C}_{B_\varepsilon}$ on $\Sc \times \Sc$ where for any $i$ in $\{1,2\}$ and $(u,k)$ in $\R \times \T$ 
 \[ [\tilde{C}_{B_\varepsilon} J]_i(u,k) =  \sum_{j=1}^2 \int_\mathbb{T} dk^\prime \; \frac{\theta_{i,B_\varepsilon}^2(k)}{2} R\left(k,k^\prime\right)  \left[J_j\left(u,k^\prime\right) - J_i(u,k) \right]   \ . \]
Let $i$ in $\lbrace 1,2 \rbrace$ we first prove that
\begin{align}
\label{behavior_colli_inter}
&\left\langle \Wc^\varepsilon_i(t), [C_{B_\varepsilon}J]_i - [\tilde{C}_{B_\varepsilon} J]_i \right\rangle = \mathcal{O}_{J}(\varepsilon) \ .
\end{align}
Let $(p,k)$ in $\R \times \mathbb{T}$ and $j$ in $\lbrace 1,2 \rbrace$, we recall that 
\begin{align*}
    R\left(k,k^\prime\right) = 16 \sin^2(\pi k) \sin^2 \left(\pi k^\prime\right) \ .
\end{align*} 
Then
\begin{align*}
    &\left\vert \int_\mathbb{T} dk^\prime \; R\left(k,k ^\prime\right)\theta_{i,B_\varepsilon}^2(k) \left[ \theta_{j,B_\varepsilon}^2\left(k^\prime\right) - \frac{1}{2} \right] \left[ \mathcal{F}\left[J_j\right]\left(p,k^\prime\right)^* - \Fc \left[J_i\right](p,k)^* \right] \right\vert  \lesssim  \vert B \varepsilon^\delta \vert  \kappa_J(p) \ , \end{align*}
where 
\begin{align*}
&\kappa_J(p) = \left( \sup_{k \in \mathbb{T}}\left\vert \Fc \left[J_i\right](k,p)\right\vert + \sup_{k \in \mathbb{T}} \left\vert \mathcal{F}\left[J_j\right](k,p) \right\vert  \right) \ .
\end{align*}
Hence, using the fact that $J_i$ belongs to $\Sc$ and the usual argument we conclude that
\begin{align*}
&\left\vert \left\langle \Wc^\varepsilon_i(t), [C_{B_\varepsilon}J]_i - [\tilde{C}_{B_\varepsilon} J]_i \right\rangle \right\vert \lesssim K_0 \Vert J \Vert B \varepsilon^\delta\ .
\end{align*}
This proves Eq.\;\eqref{behavior_colli_inter}. Similarly , one can prove that
\begin{align*}
\label{behavior_colli_inter_2}
&\left\langle \Wc^\varepsilon_i(t), [\tilde{C}_{B_\varepsilon} J]_i - [C_{0}J]_i \right\rangle = \mathcal{O}_{J}(\varepsilon)\  .
\end{align*}
By the triangle inequality we obtain Eq.\;\eqref{behav_coll} and conclude the proof of Eq.\;\eqref{evolution_eq_ib}.
\subsection{Proof of Lemma \ref{lemma_generalisation_hypo_small_k}}
\label{appendix-generalisation_hypo_small_k}
In this section, we prove Lemma \ref{lemma_generalisation_hypo_small_k} which allows us to extend the assumption \eqref{hypo_small_k}, on the initial distribution $\mu^\varepsilon$ to times $t\varepsilon^{-1}$. The proof of this Lemma follows the one of \cite[Lemma 7]{BOS09}. Let $J=(J_1,J_2)$ be a bounded measurable function on $\mathbb{T}$. Then for any $i$ in $\lbrace 1,2 \rbrace $ and $t$ in $[0,T]$
\begin{align}
\left\langle \Wc^\varepsilon_{i}(t), J \right\rangle & =  \frac{\varepsilon}{2} \int_\mathbb{T} \mathbb{E}_{\mu^\varepsilon}\left[ \left\vert\widehat{\psi}^\varepsilon_i\left(t\varepsilon^{-1},k\right) \right\vert^2 \right]  J_i(k)^* \; dk\ .\\
\left\langle \mathcal{A}^\varepsilon_{i}(t), J \right\rangle & =  \frac{\varepsilon}{2} \int_\mathbb{T} \mathbb{E}_{\mu^\varepsilon}\left[  \widehat{\psi}^\varepsilon_i\left(t\varepsilon^{-1},k\right) \widehat{\psi}^\varepsilon_i\left(t\varepsilon^{-1},-k\right)  \right] J_i(k)^* \; dk\ .
\end{align}
Since $J$ is a bounded measurable function the previous objects are well defined. 
Indeed, we have
\begin{align}
    \left\vert \left\langle \Wc_i^\varepsilon(t) , J \right\rangle \right\vert \leq K_{0} \sup\limits_{k \in \mathbb{T}}  \vert J(k) \vert \quad \text{and} \quad \left\vert\left\langle \mathcal{A}_i^\varepsilon(t), J \right\rangle \right\vert \leq K_0 \sup_{k \in \T} \left\vert J (k) \right\vert  \ . 
\end{align}
Using Lemma \ref{eig} we get
\begin{align} 
\label{The_batman}
(A+B_\varepsilon G)\left[\left\vert \widehat{\psi}^\varepsilon_i \right\vert^2 \right]\left(k\right) =0 \ .
\end{align}
Following the proof presented in Appendix \;\ref{appendix_behavior_colli_term} and using Eq.\;\eqref{The_batman} we obtain that
\begin{align*}
\partial_t \left\langle \mathcal{W}_i^\varepsilon(t), J_i \right\rangle & =    \gamma \left(  \left\langle \mathcal{W}_i^\varepsilon(t), [C_0J]_i \right\rangle + \left\langle\mathcal{A}_i^\varepsilon(t),[C_0J]_i \right\rangle + \left\langle\left({\mathcal{A}}_i^\varepsilon\right)^*(t),[C_0J]_i \right\rangle \right) +  \mathcal{O}_{t,J}(\varepsilon)\ ,\nonumber
\end{align*}
where we recall that for any $i$ in $\{1,2\}$,  $[C_0J]_i$ is defined in Eq.\;\eqref{collisional_operator},  $\kappa$ is a positive constant and $\eta$ is in $]0,1[$. Let $\kappa_\varepsilon:=\kappa\varepsilon^{\eta}$ and $J^{\kappa_\varepsilon}:=(\mathds{1}_{[-\kappa_\varepsilon, \kappa_\varepsilon]},\mathds{1}_{[-\kappa_\varepsilon, \kappa_\varepsilon]})$. Then we have 
\begin{align*}
& \left\vert  \left[C_0J^{\kappa_\varepsilon}\right]_i(k) \right\vert \lesssim \left(\kappa_\varepsilon +J^{\kappa_\varepsilon}_i(k)\right) \ ,  
\end{align*}
and 
\begin{align*}
& \left\langle \mathcal{A}_i^\varepsilon(t) + (\mathcal{A}_i^\varepsilon)^*(t), \left[C_0J^{\kappa_\varepsilon}\right]_i \right\rangle \lesssim \left\vert \langle \Wc_i^\varepsilon(t), [C_0J^{\kappa_\varepsilon}]_i \rangle \right\vert \ . \end{align*}
Hence, using Eq.\;\eqref{evolution_eq_ib} we obtain that for any $i$ in $\{1,2\}$ and $t$ in $[0,T]$ 
\begin{align*}
\left\langle \Wc_i^\varepsilon(t), J^{\kappa_\varepsilon}_i \right\rangle  
& \lesssim  \left\langle \Wc_i^\varepsilon(0), J^{\kappa_\varepsilon}_i \right\rangle + (C_1 \rho +C_2 \varepsilon )t +  \int_0^t  \left\langle \Wc_i^\varepsilon(s), J^{\kappa_\varepsilon}_i \right\rangle \; ds \ ,
\end{align*}
where $C_1$ and $C_2$ are two positive constants.
By Gronwall's lemma we conclude that 
\begin{align}
\label{equation_final_hypo_weaker}
\left\langle \Wc_i^\varepsilon(t), J^{\kappa_\varepsilon}_i \right\rangle &\lesssim \exp(C_3) \left( C_1 \kappa_\varepsilon + C_2 \varepsilon + \left\langle \Wc_i^\varepsilon(0), J^{\kappa_\varepsilon}_i \right\rangle \right) \ .
\end{align}
Using assumption \eqref{hypo_small_k} the previous term goes to zero when $\varepsilon$ goes to zero which proves that
\[  \limsup_{\varepsilon \rightarrow 0} \left\langle \Wc_i^\varepsilon(t), J^{\kappa_\varepsilon}_i \right\rangle = \limsup_{\varepsilon \rightarrow 0} \frac{\varepsilon}{2} \int_{\vert k \vert < \kappa_\varepsilon} \mathbb{E}_{\mu^\varepsilon}\left[ \left\vert\widehat{\psi}^\varepsilon_i\left(t\varepsilon^{-1},k \right) \right\vert^2 \right] \; dk  = 0 \ .\]
This ends the proof of Lemma \ref{lemma_generalisation_hypo_small_k}. 
\subsection{\textsc{Proof of the uniqueness in the Boltzmann equation}}
\label{appendix_unicite_Boltzmann}
In this section, we want to prove Lemma \ref{lemma_unicite}. Let $\mu$ be a \textbf{Borel measure valued} solution of Eq.\;\eqref{BE}. Then, for any $J:=(J_1,J_2) \in \mathcal{S} \times \mathcal{S}$ we have that 
\begin{small}
\begin{align*}
\sum_{i=1}^2 \int_{\mathbb{R} \times \mathbb{T}} \mathbb{E}\left[ J(u,k) \;  d\mu_{i}(t,u,k) \right]  = \sum_{i=1}^2 \int_{\mathbb{R} \times \mathbb{T}} \mathbb{E}\left[  J_{\tilde{I}_B(t)}\left( \tilde{Z}_{u,B}(t) , \tilde{K}_B(t) \right) \; d\mu_i(0,u,k) \right] 
\end{align*}
\end{small}
where $\left(\tilde{Z}_{u,B}(\cdot), \tilde{I}_B(\cdot), \tilde{K}_B(\cdot) \right)$ is the dual process of $\left(Z_{u,B}(\cdot), I_B(\cdot), K_B(\cdot) \right)$ starting from $(u,k,i)$. Hence, the infinitesimal generator of the dual process is defined by $\mathcal{L}+\textbf{v}_B \partial_u$ where $\mathcal{L}_B$ is defined in Eq.\;\eqref{def_infinit_generator_boltzmann}. Since the Markov process $\left(\tilde{Z}_{u,B}(\cdot), \tilde{I}_B(\cdot), \tilde{K}_B(\cdot) \right)$ is unique, we conclude the proof. 
\section{Proof of Proposition \ref{proposition_x_B} and Proposition \ref{proposition_f}}
In this section, we prove Proposition \ref{proposition_x_B} and Proposition \ref{proposition_f}.
\subsection{Proof of Proposition \ref{proposition_x_B}}
\label{appendix_preuve_proposition_x_B}
We recall that $B>0$ and that $x_{B,\pm}$ is defined in Eq.\;\eqref{def_x_B_1}. 
Observe that $\text{sign}(x_{B,\pm}(r))=\text{sign}(r)$, moreover
Eq.\;\eqref{def_x_B_1} has exactly two solutions with opposite sign. This proves that $x_{B,\pm}(-r)=-x_{B,\pm}(r)$ and ends the proof of item i).

\medskip

We prove item ii), since the proof  are similar we only give the details for $x_{B,+}$.
Using Eq.\;\eqref{def_x_B_1} we get that for any $r \neq 0$
\[ \vert x_{B,+}(r) \vert < \sqrt{\frac{ \pi}{2\gamma}} \vert r \vert^{-\frac{1}{2}} \ .  \]
Hence for each $r \neq 0$, $(x_{B,+}(r))_{B >0}$ is a bounded sequence and therefore admits an accumulation point. Observing that there is only one accumulation point we conclude that the whole sequence converges and
\[ \lim\limits_{B \rightarrow 0} x_{B,+}(r) = \text{sign}(r) \sqrt{\frac{ \pi}{2\gamma}} \vert r \vert^{-\frac{1}{2}} = x_0(r) \ .\]
To prove the convergence of the sequence $(x_{B,+}^\prime)_{B >0}$ we derive Eq.\;\eqref{def_x_B_1} with respect to $r$ and we send $B$ to zero. This concludes the proof of item ii). 

\medskip 

We prove item iii). We can write Eq.\;\eqref{def_x_B_1} in the following way
\begin{equation}
\label{x_B_intermediaire}
\left( \sqrt{\frac{4x_{B,\pm}^2(r)}{B^2} +1} \pm 1 \right)  x_{B,\pm}(r) = \frac{\pi }{B\gamma r} \ . 
\end{equation}
Using Eq.\;\eqref{x_B_intermediaire} we observe that the sequence $\left(B^{-1}x_{B,\pm}(r)\right)_{B >0}$ converges to zero when $B$ goes to infinity. Hence by performing a Taylor expansion in Eq.\;\eqref{x_B_intermediaire} we obtain the first part of item iii).
To conclude the proof we derive Eq.\;\eqref{x_B_intermediaire} with respect to $r$ and send $B$ to infinity.
\subsection{Proof of Proposition \ref{proposition_f}}
\label{appendix_preuve_proposition_f}
Let $r\neq0$ we have 
\begin{align}
g_{B,\pm}(r) &= -\frac{x_{B,\pm}^\prime(r)}{4 \pi} \left( \frac{1}{2} \pm \frac{B}{4\sqrt{x_{B,\pm}^2 (r) +\frac{B^2}{4}}} \right) x_{B,\pm}^2(r)  \; \mathds{1}_{r > 0}\nonumber \\
&-\frac{x_{B,\pm}^\prime(r)}{4 \pi} \left( \frac{1}{2} \pm \frac{B}{4\sqrt{x_{B,\pm}^2 (r) +\frac{B^2}{4}}} \right) x_{B,\pm}^2(r) \; \mathds{1}_{r < 0} \ . \label{equation_densite_intermediaire}
\end{align} 
By Proposition \ref{proposition_x_B}, $x_{B,\pm}$ is an odd function, hence its derivative is an even function and this proves item i).\\
Since the arguments are similar, to prove item ii) it is sufficient to show that 
\begin{equation}
    \int_{1}^\infty g_{B,\pm}^\prime(r) \; dr < \infty \quad  \text{and} \int_{0}^1 r^2 g_{B,\pm}^\prime(r) \; dr < \infty \ .  
\end{equation}
Let $A >1$, then
\begin{align*}
    \int_{1}^A g_{B,\pm}^\prime(r) \; dr  \lesssim 3^{-1}(x_{B,\pm}^3(1)- x_{B,\pm}^3(A)) \ .
\end{align*}
By Eq.\;\eqref{def_x_B_1} we deduce that $(x_{B,\pm}(A))_{A >1}$ goes to zero when $A$ goes to infinity. The monotone convergence theorem ends the proof. Let $\varepsilon <1$, we have
\begin{align*}
    \int_{\varepsilon}^1 r^2h_{B,\pm}^\prime(r) \; dr  & \lesssim 3^{-1}(x_{B,\pm}^3(1)- \varepsilon^2 x_{B,\pm}^3(\varepsilon)) + \frac{2}{3} \int_{\varepsilon}^1 r x_{B,\pm}^3(r) \; dr \ .
\end{align*}
By Eq.\;\eqref{def_x_B_1} and using that $x_{B,\pm}(r) >0$ for $r>0$ we get
\begin{align}
\label{Bob}
\int_{\varepsilon}^1 r x_{B,\pm}^3(r) \; dr  =  \int_{\varepsilon}^1 \frac{\pi x^2_{B,\pm}(r) \; dr}{\gamma \left( 2\sqrt{x_{B,\pm}^2(r) + \frac{B^2}{4}} \pm B \right) } \ .
\end{align}
By sending $r$ to zero in Eq.\;\eqref{def_x_B_1} we obtain that for any $B$ and $r$ in\footnote{Here, $\mathcal{V}(0)$ denotes a neighborhood of $0$.}  $\mathcal{V}(0)$
\[ x_{B,\pm}(r) \sim \vert r \vert^{-\frac{1}{2}} C(B) \ ,   \]
where $C(B)$ is a constant which depends on $B$.
Hence we deduce that 
\[ \lim\limits_{ \varepsilon \rightarrow  0 } x_{B,\pm}^3(\varepsilon) \varepsilon^2 = 0 \quad \text{and} \quad \lim\limits_{\varepsilon \rightarrow 0} \int_{\varepsilon}^1 r x_{B,\pm}^3(r) \; dr < \infty \ .  \]
This ends the proof of item ii). Using item ii) (resp. item iii)), of Proposition \ref{proposition_x_B} and sending $B$ to zero (resp. to infinity) in Eq.\;\eqref{equation_densite_intermediaire} we get item iii)  (resp. item iv)) of Proposition \ref{proposition_f} and conclude the proof. 
\section{Proof of Proposition \ref{proposition_properties_pi} and Proposition \ref{proposition_B_N}}
\label{section_durrett_extended}
We recall that $B >0$. In this section, first we prove Proposition \ref{proposition_properties_pi} which gives us the tails of the function $\Psi_{B_N}$ under the measure $\pi_{B_N}$. Then we prove Proposition \ref{proposition_B_N} which allows us to prove Theorem \ref{main_theorem_hydro}.
\subsection{Proof of Proposition \ref{proposition_properties_pi}}
\label{appendix_proof_tails}
We recall that $\Psi_{B_N}$ is defined by Eq.\;\eqref{def_Psi}. Since for any $i$, $\Psi_{B_N}(\cdot,i)$ is an odd function and that the density of $\pi_{B_N}$ with respect to the Lebesgue measure on $\T$ is even we have for any $r\neq0$
\[  \pi_{B_N} \left( \{ (k,i), \quad \Psi_{B_N}(k,i) > Nr  \} \right) 
= \pi_{B_N} \left( \{ (k,i), \quad \Psi_{B_N}(k,i) < -Nr  \} \right) \ .\]
Hence, we only prove the result for $r>0$. We make the change of variables $x=\sin(\pi k)N^\delta$ for $k$ in $\T$ and we get for $r>0$
\begin{align*}
&\pi_{B_N} \left( \{ (k,i), \quad \Psi_{B_N}(k,i) > Nr  \} \right) \\
&= \frac{N^{-3 \delta}}{4\pi }\int_0^{N^{\delta}} \mathds{1}_{A_{N,B,+}(r)}(x)  \left(\frac{1}{2} + \frac{B}{4 \sqrt{x^2 + \frac{B^2}{4}}} \right) \frac{x^2}{\sqrt{1-x^2 N^{-2\delta}}} \; dx   \\
&+ \frac{N^{-3 \delta}}{4\pi }\int_0^{N^{\delta}} \mathds{1}_{A_{N,B,-}(r)}(x)  \left(\frac{1}{2} - \frac{B}{4 \sqrt{x^2 + \frac{B^2}{4}}} \right) \frac{x^2}{\sqrt{1-x^2 N^{-2\delta}}} \; dx   \\
& =: I_{B,+}^N(r) + I_{B,-}^N(r) \ ,
\end{align*}
where
\begin{align*}
A_{N,B,\pm}(r)  &=  \left\{x \; \Bigg{|} \; \frac{\left( 2\sqrt{x^2 + \frac{B^2}{4}} \pm B \right)  x}{ \sqrt{1-x^2 N^{-2\delta}}} < \frac{ \pi N^{2 \delta - 1}}{\gamma r}     \right\} \ .
\end{align*}
Let $r \neq 0$, we define $x_{N,B,\pm}(r)$ the solutions on $\left[-N^{\delta},N^{\delta}\right]$ of the following equations \begin{align}
\frac{\left( 2\sqrt{x_{N,B,\pm}^2(r) + \frac{B^2}{4}} +B \right)  x_{N,B,\pm}(r)}{ \sqrt{1-x_{N,B,\pm}^2(r) N^{-2\delta}}} &= \frac{\pi N^{2 \delta - 1}}{\gamma r} \ . \label{def_x_N_1}
\end{align}
Observe that sign$(x_{N,B,\pm}(r))=$ sign$(r)$. To complete the proof of Proposition \ref{proposition_properties_pi} we need the following lemma which is proved at the end of this section.
\begin{lemma}
\label{lemma_x_N} Let $r \neq0$, we have the following results
\begin{itemize}
\item[i)] If $\delta >\frac{1}{2}$ then 
\begin{align}
\lim\limits_{N \rightarrow \infty}  N^{\tfrac{1}{2}-\delta} x_{N, B,\pm}(r) &=\sqrt{\pi}(2\gamma)^{-\tfrac{1}{2}} \text{sign}(r)\vert r \vert^{-\tfrac{1}{2}} \ .
\end{align}
\item[ii)] If $\delta =\frac{1}{2}$ then 
\begin{align}
\lim\limits_{N \rightarrow \infty}   x_{N,B,\pm}(r) &= x_{B,\pm}(r) \ , \label{preuve_x_N_2,dsupb}
\end{align}
where $x_{B,\pm}$ is defined in Eq.\;\eqref{def_x_B_1}.
\item[iii)] If $\delta <\frac{1}{2}$ then 
\begin{align}
 \lim\limits_{N \rightarrow \infty} N^{1 - 2\delta} x_{N,B,+}(r) &=\pi (r \gamma)^{-1} B.\\ 
 \lim\limits_{N \rightarrow \infty}  N^{\tfrac{1-2\delta}{3}}x_{N,B,-}(r)&= (\pi)^{\tfrac{1}{3}}(2\gamma)^{-\tfrac{1}{3}} \text{sign}(r) \vert r\vert^{-\tfrac{1}{3}} B^{\tfrac{1}{3}} \ . \label{preuve_x_N_2,dinfb}
\end{align}
\end{itemize}
\end{lemma}
\begin{flushleft}
From Lemma \ref{lemma_x_N} we deduce that for any $r\neq0$
\end{flushleft}
\begin{align*}
I_{B,\pm}^N(r)&= \frac{N^{-3 \delta}}{4\pi }\int_0^{x_{N,B,\pm}(r)}   \left(\frac{1}{2} \pm \frac{B}{4 \sqrt{x^2 + \frac{B^2}{4}}} \right) \frac{x^2}{\sqrt{1-x^2 N^{-2\delta}}} \; dx \; \mathds{1}_{r>0}(r) \ .
\end{align*}
For $\delta > \frac{1}{2}$,  we make the change of variables $y=xN^{\tfrac{1}{2}-\delta}$ in $I_{B,\pm}^N$ and by the dominated convergence theorem we get that 
\begin{align*}
& \lim\limits_{N \rightarrow \infty} N^{\frac{3}{2}} I^N_{B,\pm}(r) = \frac{\sqrt{\pi} r^{-\tfrac{3}{2}}}{2^{\frac{9}{2}} 3 \gamma^{\frac{3}{2}}} \ .
\end{align*}
From this we deduce that for $r>0$
\begin{equation}
\lim\limits_{N \rightarrow \infty} N^{\frac{3}{2}} \pi_{B_N} \left(\left\{(k,i) \vert \; \Psi_{B_N}(k,i) > N r   \right\}\right) = \kappa_1 r^{-\frac{3}{2}} \gamma^{-\frac{3}{2}} \ ,
\end{equation} 
with 
\begin{equation}
\label{def_kappa_1}
\kappa_1=\frac{\sqrt{\pi} }{ 2^{\frac{7}{2}} 3 } \ .
\end{equation} 
$ \bullet$ Let $\delta=\frac{1}{2}$, by using the dominated convergence theorem we get 
\begin{align*}
\lim\limits_{N \rightarrow \infty} N^{3\delta} I_{B,\pm}^N(r) &= \frac{1}{4\pi }\int_{0}^{x_{B,1}(r)} \left(\frac{1}{2} \pm \frac{B}{4 \sqrt{x^2 + \frac{B^2}{4}}} \right) x^2 \; dx=h_{B,\pm}(r) \ ,
\end{align*}
where $h_{B,\pm}$ is defined in Eq.\;\eqref{def_f_B,+}.\\
$\bullet$ Let $\delta < \frac{1}{2}$, then we make the change of variables $y=xN^{\tfrac{1-2\delta}{3}}$ in $I_{B,-}^N(r)$ and the change of variables $y=xN^{1-2\delta} $ in $I^N_{B,+}(r)$ and by the dominated convergence theorem we get that
\begin{align*}
 \lim\limits_{N \rightarrow \infty} N^{1-\delta} I^N_{B,+} = 3^{-1} \pi^{3} (r \gamma)^{-3}  B^3 \quad \text{and} \quad  \lim\limits_{N \rightarrow \infty} N^{\frac{5-\delta}{2}} I^N_{B,-} = \kappa_2 \gamma^{-\frac{5}{3}}B^{-\frac{1}{3}} \vert r \vert^{-\frac{5}{3}} \ ,
\end{align*}
with 
\begin{equation}
\label{def_kappa_2}
\kappa_2= \frac{\pi^{\frac{2}{3}}}{ 2^{\frac{11}{3}}3}.
\end{equation} 
From this we deduce that 
\begin{align*}
\lim\limits_{N \rightarrow \infty} N^{\frac{5 - \delta}{2}}\pi_{B_N} \left\{(k,i) \vert \; \Psi_{B_N}(k,i) > N r   \right\} 
& = \kappa_2 \gamma^{-\frac{5}{3}}B^{-\frac{1}{3}} \vert r \vert^{-\frac{5}{3}} \ .
\end{align*}
This ends the proof of Proposition \ref{proposition_properties_pi}. 
\begin{proof}[Proof of Lemma \ref{lemma_x_N}]Since the proofs are similar we will only prove the results for $r>0$ and  $(x_{N,B,-}(r))_{N \in \N}$.
\medskip 

When $2 \delta - 1>0$ , $N^{2 \delta - 1}$ goes to infinity when $N$ goes to infinity, hence by Eq.\;\eqref{def_x_N_1}  $(x_{N,B,-}(r))_{N \in \N}$ is not bounded. Since $r>0$, $(x_{N,B,-}(r))_{N \in \N}$ is positive we deduce that $(x_{N,B,-}(r))_{N \in \N}$ goes to infinity when $N$ goes to infinity.  
By a Taylor expansion in $N$ in Eq.\;\eqref{def_x_N_1} we get Eq.\;\eqref{preuve_x_N_2,dsupb}.

\medskip
When $2\delta-1 \leq0$, $N^{2\delta-1}$ is bounded, then we deduce that the sequence $(x_{N,B,-}(r))_{N \in \N}$ is bounded and has an accumulation point denoted by $l_B(r)$. By sending $N$ to infinity in Eq.\;\eqref{def_x_N_1} we get that for $2\delta-1 <0$ (resp. $2\delta=1$), $l_B(r)=0$ (resp. $l_{B}(r)=x_{B,\pm}(r)$ defined in Eq.\;\eqref{def_x_B_1}). By a Taylor expansion in Eq.\;\eqref{def_x_N_1} we obtain Eq.\;\eqref{preuve_x_N_2,dsupb} and Eq.\;\eqref{preuve_x_N_2,dinfb}. This ends the proof of Lemma \ref{lemma_x_N}.
\end{proof}
\subsection{Proof of Proposition \ref{proposition_B_N}}
\label{appendix_proof_Levy}
To prove Proposition \ref{proposition_B_N} we need the following result which is adapted from \cite[Theorem 4.1]{DR78} and \cite[Lemma 4.2]{JKO09}. 
\begin{proposition} 
\label{proposition_array}
Let $\left(Z_N^n\right)_{(N,n) \in \N^2}$  an array of random variables  and its natural array of filtration $\left( \mathcal{G}_N^n \right)_{(N,n) \in \N^2}$ . Let $\beta$  in $(1,2)$ and $\nu$ a L\'evy measure on $\mathbb{R}^*$. We define a sequence of stochastic processes  $\left(M_N(\cdot)\right)_{N \in \N}$ by 
\begin{align}
\label{definition_M_N}
\forall t \geq0, \quad \forall N \in \N, \quad M_N(t) &= \sum_{n=1}^{\lfloor N^\beta t \rfloor} Z_n^N \ .  
\end{align}
Let $g$ in $\Cc_{\mathbf{c}}^\infty \left( \R^* \right)$ and $r \neq 0$ we assume that
\begin{align}
  \forall N \geq 1, \quad \forall n \geq 1, \quad \mathbb{E}\left[ Z_N^n  \Bigg{|} \mathcal{G}_{N}^{n-1} \right] &=0 \  ,\label{first_assumption_array} \\
 \forall r >0, \quad \lim\limits_{N \rightarrow \infty} N^{ \beta}\mathbb{P}\left[Z_N^1 > N r \right] &= \nu(r, +\infty) \ ,\label{second_assumption_array} \\
  \forall r < 0, \quad\lim\limits_{N \rightarrow \infty} N^{ \beta}\mathbb{P}\left[Z_N^1 < N r \right]& =\nu(-\infty,r)\label{thirdt_assumption_array}\  , \\
   \lim\limits_{N \rightarrow \infty} \mathbb{E}\left[ \left\vert \sum_{n=1}^{\lceil N^\beta t \rceil} \mathbb{E}\left[ g\left( \frac{Z_N^n}{N} \right)   \Bigg{|}  \mathcal{G}_{N}^{n-1}\right] - t \int_{\mathbb{R}} g(r) \; d\nu(r)  \right\vert\right]&=0 \ , \label{fourth_assumption_array} \\
 \lim\limits_{ N \rightarrow \infty} N \mathbb{E}\left[  \E\left(  g\left( \frac{Z_N^1}{N} \right)   \Bigg{|} \mathcal{G}_{N}^0  \right) \right]^2 &= 0 \ .\label{fifth_assumption_array} 
\end{align}
Then the finite-dimensional distributions of $\left(N^{-1}M_N(\cdot)\right)_{N \in \N}$ converge weakly to $Z(\cdot)$ in $\mathcal{D}\left( [0,T], \R\right)$ where $Z(\cdot)$ is a L\'evy process with L\'evy measure $\nu$.
\end{proposition}
\begin{proof}
We refer the reader to \cite[Appendix A]{JKO08V1} for the proof of Proposition \ref{proposition_array}.
\end{proof}
We recall that $\pi_{B_N}$ is the stationary measure of the chain $\left(K_{B_N}^n,I_{B_N}^n\right)_{n \in \mathbb{N}}.$
Let $\beta_\delta$ defined in Eq.\;\eqref{tails_Psi}. We recall that for any $u$ in $\R$ and positive time $t$ in $[0,T]$
\begin{align*}
Y_{u_N}^N(t)& = u - N^{-1} \sum_{n=0}^{\lfloor N^{\beta_\delta} t \rfloor} \Psi_{B_N}\left(X_{B_N}^n\right)\frac{\tau_n^N}{2\pi } \ . 
\end{align*}
We introduce the array $\left(Z_N^n\right)_{(N,n) \in \mathbb{N}^2} := \left(\Psi_{B_N}\left(X_{B_N}^n\right)\tau_N^n\right)_{(N,n) \in \mathbb{N}^2}$ and its natural array of filtration $(\Gc_N^n)_{(N,n) \in \mathbb{N}^2}$. To prove Proposition \ref{proposition_B_N} it is sufficient to show that the array $(Z_N^n)_{(N,n) \in \N^2}$ satisfies the assumption of Proposition \ref{proposition_array}. \\

Since $\Psi_{B_N}$ is a centered function,  Eq.\;\eqref{first_assumption_array} is satisfied. By Proposition \ref{proposition_properties_pi}, Eq.\;\eqref{second_assumption_array} and Eq.\;\eqref{thirdt_assumption_array} are satisfied.  It remains to prove Eq.\;\eqref{fourth_assumption_array} and Eq.\;\eqref{fifth_assumption_array}.
\begin{proof}[Proof of Eq.\;\eqref{fourth_assumption_array}] Let $g$ in $\Cc^\infty_{\mathbf{c}} \left(\R^* \right)$, we want to prove that
\begin{align}
    \lim\limits_{N \rightarrow \infty}  \mathbb{E}_{\pi_{B_N}}\left[ \left\vert \sum_{n=1}^{\lfloor N^{\beta_\delta} t \rfloor} \mathbb{E}_{\pi_{B_N}}\left[ g\left( \frac{Z_N^n}{N} \right)  \Bigg{|} \mathcal{G}_N^{n-1}\right] - t \int_{\mathbb{R}} g(r) \; d\nu_\delta(r)  \right\vert\right] =0 \ ,
\end{align}
where $\nu_\delta$ is the L\'evy measure defined in Eq.\;\eqref{tails_Psi}. We have
\begin{align*}
    &\mathbb{E}_{\pi_{B_N}}\left[  \sum_{n=1}^{\lfloor N^{\beta_\delta} t \rfloor} \mathbb{E}_{\pi_{B_N}}\left[ g\left( \frac{Z_N^n}{N} \right)  \Bigg{|} \mathcal{G}^{n-1}_N\right]\right]\\
    &= N^{\beta_\delta}  \int_{0}^{+\infty} e^{-u} du \int_{\mathbb{T} \times \{1,2\}} g \left( \frac{\Psi_{B_N}(k,i)u}{2\pi N}  \right)\; d\pi_{B_N}(k,i)   \\
 & =  N^{\beta_\delta}  \int_{0}^{+\infty} e^{-u} du \int_{\mathbb{T} \times \{1,2\}} \int_0^\infty \left(2\pi N\right)^{-1} g^\prime \left( \frac{ r}{2\pi N}\right) \; \mathds{1}_{\left[0, \Psi_{B_N}(k,i)u\right]}(r) \; d\pi_{B_N}(k,i) \; d r     \\
& =  N^{\beta_\delta}  \int_{0}^{+\infty} e^{-u} du \int_0^\infty g^\prime(r) \; \pi_{B_N} \left(\Psi_{B_N} \geq \frac{2\pi Nr}{u}\right) \; d r \ .
\end{align*}
Using the tails condition (item i) of Proposition \ref{proposition_properties_pi}) we get that
\begin{align*}
\lim\limits_{N \rightarrow \infty} N^{\beta_\delta}  \int_{0}^{+\infty} e^{-u} du \int_0^\infty g^\prime(r)\;  \pi_{B_N} \left(\Psi_{B_N} \geq \frac{2\pi Nr}{u}\right) \; d r  =  \int_\mathbb{R} g(r) \; d\nu_\delta(r)\ .
\end{align*}
This proves Eq.\;\eqref{fourth_assumption_array}. From this result we deduce that Eq.\;\eqref{fifth_assumption_array} is satisfied and the proof of Proposition \ref{proposition_B_N} is complete.
\end{proof}
\section{Proof of Lemma \ref{lemma_S_N}}
\label{section_proof_lemma_weak}
We recall that $\pi_{B_N}$ is the stationary measure of the chain $\left(K_{B_N}^n,I_{B_N}^n\right)_{n \in \mathbb{N}}.$ To prove Lemma \ref{lemma_S_N} we need the following result.
 \begin{lemma}
\label{lemma_T_N}
Let $u >0$ then 
\begin{equation}
\forall \sigma >0, \quad \lim\limits_{N \rightarrow \infty} \mathbb{P}_{\pi_{B_N}}\left[ \left\vert \frac{1}{\lfloor N^{\beta_\delta} u \rfloor} \sum _{n=0}^{\lfloor N^{\beta_\delta} u \rfloor} \lambda_{B_N}\left(X_{B_N}^n \right)\tau_N^n
 - \frac{u}{2} \right\vert > \sigma \right] =0 \ .
\end{equation}
\end{lemma}
  \begin{proof}[Proof of Lemma \ref{lemma_T_N}] In order to prove the convergence in probability we prove the convergence in $\mathbb{L}^1$. Let $N$ be fixed, we define the sequence of function $\left(\lambda_{B_N}^\varepsilon\right)_{ \varepsilon >0}$ where for any $k$ in $\T$ and $i$ in $\{1,2\}$ we have
 \begin{small}
\begin{align*}
\lambda_{B_N}^\varepsilon(k,i) &= \lambda_{B_N}(k,i) \mathds{1}_{[-\tfrac{1}{2}+\varepsilon] \cup [\varepsilon,\tfrac{1}{2}-\varepsilon]} \quad \text{and} \quad 
\overline{\lambda}_{B_N}^\varepsilon= \int_{\T \times \{1,2\}} \lambda_{B_N}^\varepsilon(k,i) \; d\pi_{B_N}(k,i) \ .
\end{align*} 
 \end{small}
 \begin{flushleft}
Observe that for any $\varepsilon$, $\lambda_{B_N}^\varepsilon$ is  in $\mathbb{L}^{2}(\pi_{B_N})$. 
Using the stationnarity of the Markov chain $\left(X_{B_N}^n\right)_{n \in \N}$ and the definition of $\pi_{B_N}$ given by Eq.\;\eqref{invariante_measure_delta} we have
 \end{flushleft}
 \begin{align*}
\frac{1}{\lfloor N^{\beta_\delta} u \rfloor} \left\Vert \sum _{n=0}^{\lfloor N^{\beta_\delta} u \rfloor} \left\vert \lambda_{B_N} \left(X_{B_N}^n \right)\tau_N^n - \lambda_{B_N}^\varepsilon\left(X_{B_N}^n \right)  \tau_N^n \right\vert   \right\Vert_{\mathbb{L}^1} &\leq C(\varepsilon) \ , \\
\frac{1}{\lfloor N^{\beta_\delta} u \rfloor} \left\Vert \sum _{n=0}^{\lfloor N^{\beta_\delta} u \rfloor} \left\vert \bar{\lambda}_{B_N}^\varepsilon - \frac{u}{2}  \right\vert   \right\Vert_{\mathbb{L}^1}& \leq C(\varepsilon) \ ,
 \end{align*}
 where $\vert C(\varepsilon) \vert \lesssim \varepsilon$.\\
By the triangle inequality, to end the proof it is sufficient to prove that for any $\varepsilon$ we have
\begin{align*}
\lim\limits_{N \rightarrow \infty} \frac{1}{\lfloor N^{\beta_\delta} u \rfloor} \left\Vert \sum _{n=0}^{\lfloor N^{\beta_\delta} u \rfloor} \lambda_{B_N}^\varepsilon\left(X_{B_N}^n \right)\tau_N^n  - \bar{\lambda}_{B_N}^\varepsilon \right\Vert_{\mathbb{L}^1} = 0 \ .
\end{align*}
Using Cauchy-Schwarz's inequality and that for any $N$ the random variables $\left(X_{B_N}^n,\tau_N^n\right)$ are independent and centered we get
\begin{align*}
\frac{1}{\lfloor N^{\beta_\delta} u \rfloor} \left\Vert \sum _{n=0}^{\lfloor N^{\beta_\delta} u \rfloor}  \lambda_{B_N}^\varepsilon\left(X_{B_N}^n \right)\tau_n^N  - \bar{\lambda}_{B_N}^\varepsilon \right\Vert_{\mathbb{L}^1}
& \leq  \frac{K(\varepsilon)}{\sqrt{N^{\beta_\delta}}} \ ,
\end{align*}
where $K(\varepsilon) \lesssim \varepsilon$.
Finally we get that 
\begin{align*}
\limsup_{ N \rightarrow \infty}  \frac{1}{\lfloor N^{\beta_\delta} u \rfloor} \left\Vert \sum _{n=0}^{\lfloor N^{\beta_\delta} u \rfloor} \lambda_{B_N}\left(X_{B_N}^n \right)\tau_N^n  - \bar{\lambda}_{B_N}^\varepsilon   \right\Vert_{\mathbb{L}^1} &
 \leq  C(\varepsilon) \ .
\end{align*}
By sending $\varepsilon$ to zero we end the proof of Lemma \ref{lemma_T_N}.
 \end{proof}
From this result we can prove Lemma \ref{lemma_S_N}.
\begin{proof}[Proof of Lemma \ref{lemma_S_N}]
 Let $t_0 >0$, $T \geq t_0$ and $\varepsilon >0$.
 \begin{eqnarray*}
 \mathbb{P}_{\pi_{B_N}}\left[ \sup_{ t \in [t_0,T]} \vert S_N(t) -S(t) \vert > \varepsilon   \right] & \leq & \mathbb{P}_{\pi_{B_N}}\left[ \exists t \in [t_0,T], \; \;  S_N(t) \leq S(t) - \tfrac{\varepsilon}{2}  \right]\\
 & + & \mathbb{P}_{\pi_{B_N}}\left[ \exists t \in [t_0,T], \; \;  S_N(t) \geq S(t) + \tfrac{\varepsilon}{2}    \right]
 \end{eqnarray*}
Let $m$ in $\N$, since $S$ is a continuous function there exists a subdivision $(t_i)_{i \in \{0, \cdots, m\} }$ such that $t_0 < t_1 < \cdots < t_m=T$ with
 \[ \forall i \in \{ 0, \cdots, m\}, \quad S(t_{i+1}) - S(t_i) \leq \frac{\varepsilon}{10} \ .\]
Using the fact that $S$ and $S_N$ are increasing functions we get 
 \begin{align*}
&\mathbb{P}_{\pi_{B_N}}\left[ \exists t \in [t_0,T], \; \;  S_N(t) \leq S(t) - \frac{\varepsilon}{2}  \right] \\
&\leq  \sum_{i=0}^{m-1} \mathbb{P}_{\pi_{B_N}}\left[ \exists t \in [t_i,t_{i+1}], \; \;  S_N(t) \leq S(t) - \frac{\varepsilon}{2}  \right]\\
 & \leq  \sum_{i=0}^{m-1} \mathbb{P}_{\pi_{B_N}}\left[ S_N(t_i) \leq S(t_{i+1}) - \frac{\varepsilon}{2}  \right] \\
 & \leq  \sum_{i=0}^{m-1} \mathbb{P}_{\pi_{B_N}}\left[ S_N(t_i) \leq S(t_{i}) - \frac{2\varepsilon}{5}  \right] \ . 
 \end{align*}
 Using the same techniques we get 
 \begin{equation*}
 \mathbb{P}_{\pi_{B_N}}\left[ \exists t \in [t_0,T], \; \;  S_N(t) \geq S(t) + \tfrac{\varepsilon}{2}    \right] \leq \sum_{i=0}^{m-1} \mathbb{P}_{\pi_{B_N}}\left[ S_N(t_{i+1}) - S(t_{i+1}) \geq \tfrac{2\varepsilon}{5}  \right] \ . 
 \end{equation*}
 Hence we proved that
 \begin{equation*}
 \mathbb{P}_{\pi_{B_N}}\left[ \sup_{ t \in [t_0,T]} \vert S_N(t) -S(t) \vert > \varepsilon   \right] \leq 2m \sup_{i=0, \cdots, m} \mathbb{P}_{\pi_{B_N}} \left[ \left\vert S_N(t_i) - S(t_i) \right\vert \geq \tfrac{2 \varepsilon}{5} \right] \ .
 \end{equation*}
 To conclude the proof it is sufficient to prove that
 \[ \forall t \in [t_0, T], \quad \forall \delta >0, \quad \lim\limits_{N \rightarrow \infty} \mathbb{P}_{\pi_{B_N}}\left[ \left\vert S_N(t) - S(t) \right\vert > \delta   \right]  =0 \ .\]
 This result follows from Lemma \ref{lemma_T_N}, indeed let $\delta >0$ then
 \begin{align*}
\mathbb{P}_{\pi_{B_N}}\left[ \left\vert S_N(t) - S(t) \right\vert > \delta   \right]
 & \leq   \mathbb{P}_{\pi_{B_N}}\left[  S_N(t)   < S(t)+ \delta   \right] \\
 &+ \mathbb{P}_{\pi_{B_N}}\left[  S_N(t)   >  S(t)+ \delta   \right]\\
  & \leq  \mathbb{P}_{\pi_{B_N}}\left[ \frac{1}{\lfloor N^{\beta_\delta} u \rfloor} \sum_{n=0}^{\lfloor N^{\beta_\delta} [S(t) + \delta] \rfloor} \lambda_{B_N}\left(X_{B_N}^n \right)\tau_N^n  > t \right]\\
  & + \mathbb{P}_{\pi_{B_N}}\left[\frac{1}{\lfloor N^{\beta_\delta} u \rfloor} \sum_{n=0}^{\lfloor N^{\beta_\delta} [S(t) - \delta] \rfloor} \lambda_{B_N}\left(X_{B_N}^n \right)\tau_N^n  \leq t \right] \ .
 \end{align*}
 Using the dominated convergence theorem and Lemma \ref{lemma_T_N}, we conclude the proof. 
 \end{proof}
 
 \section*{Acknowledgements}{The author is very grateful to Tomasz Komorowski for discussions and in particular to point out the reference \cite{K14}. 
The author would to thank his thesis advisor C\'edric Bernardin for his advices, remarks and his valuable support.}

 \bibliographystyle{plain}
\bibliography{main}
 \end{document}